\documentclass[reqno]{amsart}

\usepackage{amsthm}
\usepackage[usenames,dvipsnames]{pstricks}
\usepackage{epsfig}
\usepackage{pst-grad} 
\usepackage{pst-plot} 
\usepackage{graphicx,subfigure}
\usepackage{amsmath,amssymb}

\usepackage{acronym}
\acrodef{BO}{{\sl Benjamin-Ono}}
\acrodef{rBO}{{\sl regularized Benjamin-Ono}}
\acrodef{rILW}{{\sl regularized Intermediate Long Wave}}
\acrodef{DSW}{{\sl Dispersive Shock Wave}}
\acrodef{DSWs}{{\sl Dispersive Shock Waves}}
\acrodef{ILW}{{\sl Intermediate Long Wave}}
\acrodef{CGN}{{\sl Conjugate Gradient-Newton}}
\acrodef{SW/SW}{{\sl Shallow water / Shallow water}}
\acrodef{B/B}{{\sl Boussinesq / Boussinesq}}

\makeatletter
\newcommand{\sech}{\mathop{\operator@font sech}}
\newcommand{\sign}{\mathop{\operator@font sign}}
\makeatother

\newtheorem{lemma}{Lemma}[section]
\newtheorem{theorem}{Theorem}[section]
\newtheorem{proposition}{Proposition}[section]
\newtheorem{corolary}{Corolary}[section]

\newtheorem{remark}{Remark}[section]

\numberwithin{equation}{section}

\begin{document}

\title[]{Notes on solitary-wave solutions of Rosenau-type equations}

\author{Angel Dur\'an}
\address{\textbf{A.~Dur\'an:} Applied Mathematics Department, University of Valladolid, P/ Belen 15, 47011, Valladolid, Spain}
\email{angeldm@uva.es}

\author{Gulcin M. Muslu}
\address{\textbf{G.~Muslu:} Istanbul Technical University, Department of Mathematics, Maslak 34469, Istanbul, Turkey, and Istanbul Medipol University, School of Engineering and Natural Sciences, Beykoz 34810, Istanbul, Turkey}
\email{gulcin@itu.edu.tr}







\begin{abstract}
The present paper is concerned with the existence of solitary wave solutions of Rosenau-type equations. By using two standard theories, Normal Form Theory and Concentration-Compactness Theory, some results of existence of solitary waves of three different forms are derived. The results depend on some conditions on the speed of the waves with respect to the parameters of the equations. They are discussed for several families of Rosenau equations present in the literature. The analysis is illustrated with a numerical study of generation of approximate solitary-wave profiles from a numerical procedure based on the Petviashvili iteration.
\end{abstract}

\maketitle
\tableofcontents

\section{Introduction}
In this study, we consider Rosenau-type equations of the form
\begin{eqnarray}
u_{t}+\epsilon u_{x}+\alpha u_{xxt}+\eta u_{xxx}+\beta u_{xxxxt}+\gamma u_{xxxxx}+(g(u))_{x}=0,\label{GM1}
\end{eqnarray}
where $\alpha,\beta,\gamma,\epsilon,\eta\in\mathbb{R}$ and $g$ is a smooth nonlinear function satisfying one of these two conditions:
\begin{itemize}
\item[(P1)] $g=g(u,u',u'',u''')$ and its Taylor series vanishes at the origin along with its first derivatives, and the dependence of $g$ on $u'$ and $u'''$ occurs as sums of even-order products of these two variables, \cite{Champ}.
\item[(P2)] $g=g(u)$ is a pure function of $u$, with no dependence on the derivatives.
\end{itemize}
Some particular cases of (\ref{GM1}), present in the literature and that will be analyzed below, are:
\begin{itemize}
\item The Rosenau equation (\cite{Rosenau1988,Park1993,Demirci2022} and references therein) corresponds to $\epsilon=\beta=1, \alpha=\eta=\gamma=0$, and the cases of $g$:
\begin{itemize}
\item Classical: $g(u)=u^{2}/2$.
\item Single power: $g(u)=\frac{u^{p+1}}{p+1},\; p\geq 1$.
\item Cubic-quintic, \cite{Erbay2020}: $g(u)=\frac{u^{3}}{3}+\frac{r}{5}u^{5}$, with $r$ constant.
\end{itemize}
\item The Rosenau-RLW equation (\cite{ZZZ2010,EsfahaniP2014} and references therein) corresponds to $\eta=\gamma=0, \epsilon, \beta>0, \alpha\in\mathbb{R}$ and $g$ quadractic, say $g(u)=u^{2}/2$. It may be generalized to a single power $g(u)=\frac{u^{p+1}}{p+1},\; p\geq 1$.
\item The Rosenau-KdV equation (\cite{Zuo2006,Esfahani2011,EsfahaniP2014} and references therein) corresponds to taking $\alpha=\gamma=0, \epsilon, \beta>0, \eta\in\mathbb{R}$ and $g$ quadractic or, in general,  $g(u)=\frac{u^{p+1}}{p+1},\; p\geq 1$.
\item The Rosenau-Kawahara equation (\cite{Zuo2006,He2015} and references therein) corresponds to $\alpha=0, \epsilon, \beta>0, \gamma\in\mathbb{R}$ (with $\gamma=-1$ in \cite{Zuo2006}) and $g(u)=\frac{u^{p+1}}{p+1},\; p\geq 1$, \cite{He2015} (with $p=1$ in \cite{Zuo2006}).
\item The Rosenau-RLW-Kawahara equation (\cite{He2016,su} and references therein) corresponds to $\epsilon=2k>0, \alpha=-\mu, \eta\in\mathbb{R},\beta>0$, and, \cite{su}
$$g(u)=A\frac{u^{2}}{2}+B\frac{u^{m+1}}{m+1}+s\left(\frac{u_{x}^{2}}{2}+uu_{xx}\right),\; A,B,s\in\mathbb{R}.$$ (Other cases in \cite{He2016} may not satisfy (P1).)
\end{itemize}
The literature on the previous families of Rosenau-type equations, sketched above, is now described in more detail. The Rosenau equation was originally derived in \cite{Rosenau1988} to study the dynamics of dense discrete lattices. Some mathematical poperties of the equation with different types of nonlinearities are studied in \cite{Park1993,Erbay2020,Demirci2022}. They concern well-posedness of the initial-value problem (ivp), Lie symmetries, exact periodic traveling wave solutions (of cnoidal type), as well as the numerical generation of solitary wave solutions and computational study of their dynamics. In \cite{Zeng2003}, the Concentration-Compactness theory of \cite{Lions} is applied to study the existence of solitary wave solutions of equations of BBM type which include the Rosenau equation. This is mentioned by the authors in \cite{Erbay2020} as starting point of their numerical investigation of the waves with the Petviashvili iteration, \cite{Petv1976}. The orbital stability, \cite{GrillakisSS1987} of solitary wave solutions of the Rosenau-RLW equation (whose existence is assumed as critical points of the enegy subject to constant charge) is investigated in \cite{EsfahaniP2014}. Some references are concerned with the derivation of exact solitary wave solutions for particular values of the speed. This is the case of \cite{Esfahani2011,Zuo2006} for the Rosenau-RLW, Rosenau-KdV equations and generalized versions, \cite{Zuo2006,He2015} for the Rosenau-Kawahara equation, and \cite{He2016} for the Rosenau-RLW-Kawahara equation, among others. The numerical treatment of the equations is mainly based on finite differences, \cite{He2015,He2016,HeP}, finite element methods, (\cite{ShiJ2023} and references therein), and Fourier pseudospectral discretizations in space with an explicit, fourth-order Runge-Kutta time integrator for the Rosenau equation in \cite{Demirci2022}.

The main contributions of the present paper are concerned with the existence of solitary wave solutions of (\ref{GM1}), focused on three types: Classical Solitary Waves (CSW), with monotone and nonmonotone decay, and Generalized Solitary Waves (GSW). To this end, some mathematical properties of (\ref{GM1}) are described in section \ref{sec2}. They include well-posedness of the ivp, conserved quantities and Hamiltonian formulation. This section is completed with the description of the problem of searching for solitary wave solutions. The problem is then analized in the following two sections using two classical theories. In section \ref{sec3} we study the existence with the Normal Form Theory (NFT), \cite{HaragusI2011,IoossA}, which is used to identify solitary-wave structures as homoclinic orbits of the solitary wave problem considered as a dynamical system. Our study will be based on different approaches, \cite{Champ,Choudhury2007,Iooss95}, and will be applied to the families of Rosenau-type equations mentioned above. This includes a discussion, for each family, on the range of values of the speed and depending on the parameters of the equations, required for the existence results. The details are in Appendix \ref{GMapp1}. The existence of the three types of waves will be illustrated here with a computational study of generation of approximate profiles. This will be done by solving numerically the fourth-order differential equation satisfied by the solitary-wave profiles. The numerical procedure consists of discretizing the periodic, ivp on a long enough interval with a Fourier collocation method and solving iteratively the resulting algebraic system for the discrete Fourier coefficients of the approximation with Petviashvili-type methods, \cite{Petv1976,pelinovskys,AlvarezD2014,AlvarezD2015}. The scheme is described in Appendix \ref{GMapp2} and from the numerical experiments some additional properties of the waves, such as the asymptotic decay and the speed-amplitude relation, are explored.

A different point of view is taken in section \ref{sec4}, where the existence of CSW's is analyzed from the Concentration-Compactness theory (CCT) of Lions, \cite{Lions}. In this case, the CSW's are found as minimizers of a constrained variational problem. Compared to the study performed in \cite{Zeng2003} for BBM-type equations, our approach consists of the analysis of different minimization problems, more related to other applications of the theory, cf. e.~g. \cite{AnguloS2020,EsfahaniL2021,DDS1}. We will make a general assumption on a key property of coercivity of the functional to be minimized, that will be characterized in terms of the speed of the wave. The study of this condition for each of the families of Rosenau equations introduced above will lead us to the corresponding, specific conditions on the speed in order to ensure the existence. We will also prove that the results obtained are extensions of some derived from the NFT, in the sense that the conditions for the speed are similar but do not impose restrictions on the proximity to limiting values, that are required by the local character of the NFT. The numerical procedure described in Appendix \ref{GMapp2} will be used again to illustrate the generation of some approximate profiles in this case; all were shown to be nonmonotone. The paper will be closed with a section of concluding remarks.

Throughout the paper, $C$ will denote a generic constant, that may depend on the corresponding parameters involved in the context.
\section{Some mathematical properties of the Rosenau equations}
\label{sec2}
In this section, some mathematical properties of (\ref{GM1}), that will be used in the paper, are collected.
In order to delimit the \lq admissible\rq\ equations (\ref{GM1}) from a modelling point of view, a first step would be the study of linear well-posedness of the ivp in $L^{2}$, meaning, as usual, existence and uniqueness of solutions and continuous dependence on the initial data. If
\begin{eqnarray*}
\widehat{f}(k)=\int_{\mathbb{R}}f(x)e^{-ikx}dx,\quad k\in\mathbb{R},
\end{eqnarray*}
denotes the Fourier transform of $f\in L^{2}$, we consider the ivp of the linearized equation
\begin{eqnarray}
&&u_{t}+\epsilon u_{x}+\alpha u_{xxt}+\eta u_{xxx}+\beta u_{xxxxt}+\gamma u_{xxxxx}=0,\label{GM1L}\\
&&u(x,0)=u_{0}(x),\quad x\in\mathbb{R}, t>0,\nonumber
\end{eqnarray}
in Fourier space (with respect to $x$) to have
\begin{eqnarray}
&&\frac{d}{dt}\widehat{u}(k,t)+ikl(k)\widehat{u}(k,t)=0,\quad k\in\mathbb{R}, t>0,\label{GM1L2}\\
&&\widehat{u}(k,0)=\widehat{u_{0}}(k),\nonumber
\end{eqnarray}
where
\begin{eqnarray*}
l(k)=\frac{\epsilon-\eta k^{2}+\gamma k^{4}}{1-\alpha k^{2}+\beta k^{4}}.
\end{eqnarray*}
The solution of (\ref{GM1L2}) is
\begin{eqnarray*}
\widehat{u}(k,t)=e^{-ikl(k)t}\widehat{u_{0}}(k),\quad k\in\mathbb{R}, \quad t>0.
\end{eqnarray*}
Then, well-posedness of (\ref{GM1L}) holds if $m(k,t)=e^{-ikl(k)t}$ is bounded in finite intervals of $t$. It looks clear that this happens when the denominator of $l(k)$ has no zeros on $k\in\mathbb{R}$. If we consider the polynomial
\begin{eqnarray*}
P(x)=1-\alpha x+\beta x^{2},\quad x\geq 0,
\end{eqnarray*}
then linear well-posedness requires
\begin{eqnarray}
\alpha^{2}<4\beta.\label{linearw}
\end{eqnarray}
Condition (\ref{linearw}) will be assumed throughout the paper. In particular, (\ref{linearw}) implies $\beta>0$. (Note that the case $\alpha=\beta=0$ corresponds to the Kawahara or fifth-order KdV equation, \cite{Levandosky1999,EsfahaniL2021,KMolinet2018}.) We will also assume $\epsilon>0$, as in all the literature on Rosenau-type equations that we are aware of.

On the other hand, local well-posedness results of the full nonlinear problem (\ref{GM1}) with suitable nonlinear terms can be derived from several approaches, \cite{AlbertB1991,BonaChS2004}. In what follows, for $s\in\mathbb{R}$, $H^{s}=H^{s}(\mathbb{R})$ will be the $L^{2}$-based Sobolev space over $\mathbb{R}$, with norm denoted by $||\cdot||_{s}$, and for $T>0$, $X_{s}=C(0,T,H^{s})$ will stand for the space of continuous functions $u:[0,T]\rightarrow H^{s}$ with norm
$$||u|_{X^{s}}=\max_{0\leq t\leq T}||u(t)||_{s}.$$
\begin{theorem}
\label{GMt1}
Let $s\geq 0$. Assume that $\epsilon>0$, that (\ref{linearw}) holds, and that $g$ in (\ref{GM1}) is locally Lipschitz in $H^{s}$ with $g(0)=0$. Let $u_{0}\in H^{s}$. Then there exists $T>0$ and a unique solution $u\in X^{s}$ of (\ref{GM1}) with initial condition $u_{0}$.
\end{theorem}
\begin{proof}
We write (\ref{GM1}) in the form
\begin{eqnarray}
\frac{du}{dt}+\mathcal{L}u=\mathcal{G}(u),\label{GM01}
\end{eqnarray}
where $\mathcal{L}$ is the linear operator with Fourier symbol $ikl(k), k\in\mathbb{R},$ and
\begin{eqnarray*}
\mathcal{G}(u)=-M^{-1}\partial_{x}g(u),
\end{eqnarray*}
where $M$ is given by 
\begin{eqnarray}
M=1+\alpha\partial_{x}^{2}+\beta\partial_{x}^{4},\label{GM01b}
\end{eqnarray} 
which, from (\ref{linearw}), is invertible. From Duhamel formula, the solution of (\ref{GM01}) can be written as
\begin{eqnarray*}
u=\mathcal{S}(t)u_{0}+\int_{0}^{t}\mathcal{S}(t-\tau)\mathcal{G}(u)d\tau,
\end{eqnarray*}
where $\mathcal{S}(t)$ denotes the group generated by $\mathcal{L}$. We consider the mapping $\widetilde{u}\mapsto u$ with
\begin{eqnarray}
u=\mathcal{S}(t)u_{0}+\int_{0}^{t}\mathcal{S}(t-\tau)\mathcal{G}(\widetilde{u})d\tau.\label{GM02}
\end{eqnarray}
Note that due (\ref{linearw}), $\mathcal{S}(t)$ is a unitary group on $H^{s}$. On the other hand, let $T, R>0$ and $\widetilde{u}_{1},\widetilde{u}_{2}$ be in a closed ball of radius $R$ centered at $0$ in $X^{s}$. Since $g$ is locally Lipschitz and from the arguments used in e.~g. \cite{BonaChS2004}, we may find $C=C(R)>0$ such that
\begin{eqnarray}
\|\mathcal{G}(\widetilde{u}_{1})(\tau)-\mathcal{G}(\widetilde{u}_{2})(\tau)\|_{X^{s}}\leq C\|\widetilde{u}_{1}(\tau)-\widetilde{u}_{2}(\tau)\|_{X^{s}},\label{GM03}
\end{eqnarray}
for any $0\leq \tau\leq T$. Then, (\ref{GM03}) and the hypothesis $g(0)=0$ imply the existence of some small enough $T>0$ such that (\ref{GM02}) is a contraction of the closed ball into itself. The result follows from the application of the Contraction Mapping Theorem.
\end{proof}

Some additional properties of (\ref{GM1}) with $g$ satisfying (P1) or (P2), can also be mentioned. Note first that in the case of (P2) and if $G'(u)=g(u)$, then (\ref{GM1}) admits
\begin{eqnarray*}
V(u)=\frac{1}{2}\int_{\mathbb{R}}(u^{2}-\alpha u_{x}^{2}+\beta u_{xx}^{2})dx,\label{GM12}
\end{eqnarray*}
as an invariant in time quantity by smooth enough solutions $u$ which decay, along with its derivatives in space up to second order, to zero as $|x|\rightarrow\infty$. Furthermore, if (\ref{linearw}) holds, then
(\ref{GM1}) admits a Hamiltonian structure
\begin{eqnarray*}
u_{t}=\mathcal{J}\delta H(u),
\end{eqnarray*}
with $\mathcal{J}=-\partial_{x}M^{-1}$, $M$ as in (\ref{GM01b}), and the Hamiltonian function given by
\begin{eqnarray*}
H(u)=\int_{\mathbb{R}}\left(\frac{1}{2}\left(\epsilon u^{2}-\eta u_{x}^{2}+\gamma u_{xx}^{2}\right)+G(u)\right)dx.\label{GM13}
\end{eqnarray*}
The present paper is focused on the existence of solitary wave solutions.
Solitary wave solutions of the equation \eqref{GM1} are smooth functions $u$ of the form $u(x,t)=\varphi(X)$ where $X=x-c_s t$, $c_s \neq 0$ and where $\varphi$ satisfies
\begin{equation}
(\gamma -\beta c_s)  \varphi^{\prime\prime\prime\prime\prime}+ (\eta -\alpha c_s)  \varphi^{\prime\prime\prime}  +(\epsilon-c_s) \varphi^{\prime} +
 [ g(\varphi)]^{\prime} =0.  \label{solitary}
\end{equation}
Here $^\prime$ denotes the derivative with respect to $X$. Integrating the equation \eqref{solitary} and assuming that $\varphi$ and derivatives vanish as $|X|\rightarrow\infty$ then 
\begin{equation}
(\gamma -\beta c_s)  \varphi^{\prime\prime\prime\prime} + (\eta -\alpha c_s) \varphi^{\prime\prime}  +(\epsilon-c_s) \varphi +  g(\varphi) =0.  \label{GM14}
\end{equation}
Assuming that $\gamma-\beta c_{s}\neq 0$, then (\ref{GM14}) can be alternatively written as
\begin{eqnarray}
\varphi^{\prime\prime\prime\prime} -b\varphi^{\prime\prime}  +a \varphi =  f(\varphi),  \label{GM14b}
\end{eqnarray}
where $f=g/(\beta c_{s}-\gamma)$ and
\begin{eqnarray}
a=a(c_{s})=\frac{c_{s}-\epsilon}{\beta c_{s}-\gamma},\quad
b=b(c_{s})=\frac{\alpha c_{s}-\eta}{\gamma-\beta c_{s}}.\label{GM14c}
\end{eqnarray}
%
In what follows we will make a general study of existence of solutions of (\ref{GM14}). Under the assumption that $\varphi\rightarrow 0$ as $|X|\rightarrow\infty$, the corresponding solution will be called Classical Solitary Wave (CSW). However, we may consider other solutions of  (\ref{GM14}), with a different asymptotic behaviour, and our study will also consider two types of them: Generalized Solitary Waves (GSW), \cite{Lombardi1,Lombardi2,Lombardi3}, and Periodic Traveling Waves (PTW), \cite{Angulo}.

\section{Existence of solitary wave solutions via Normal Form Theory}
\label{sec3}
One of the approaches, considered here, to study the existence of the solutions for \eqref{GM14b}, \cite{Champ, ChampS, ChampT, IK}, is based on the application of the 
Normal Form Theory to the equivalent first-order differential system for 
 $U=(U_1,U_2,U_3,U_4)^T:=(\varphi,\varphi^{\prime}, \varphi^{\prime\prime}, \varphi^{\prime\prime\prime})$, given by
\begin{equation}\label{GM21a}
U^{\prime}=V(U,c_s):=L(c_s)U+R(U,c_s),
\end{equation}
where
\begin{equation}\label{GM21b}
L(c_s)=
\begin{pmatrix}
0 & 1 & 0 & 0  \\
0 & 0 & 1 & 0 \\
0 & 0 & 0 & 1 \\
\displaystyle\frac{\epsilon-c_s}{\beta c_s -\gamma} & 0 &\displaystyle\frac{\alpha c_s -\eta}{\gamma-\beta c_s} & 0 \\
\end{pmatrix},
\hspace*{20pt}
R(U,c_s)=
\begin{pmatrix}
0 \\
0 \\
0 \\
\displaystyle\frac{g(U)}{\beta c_s- \gamma} 
\end{pmatrix},
\end{equation}
The system \eqref{GM21a}, \eqref{GM21b} admits the solution $U=0$ if $g(0)=0$. We also assume that $g'(U)=0$ so that $\partial_{U}R(0)=0$.
Furthermore, note that if $g$ satisfies (P1) then (\ref{GM14b}) is reversible under the transformation
\begin{eqnarray}
&&t\mapsto -t\nonumber\\
&&(\varphi,\varphi',\varphi'',\varphi''')^{T}\mapsto (\varphi,-\varphi',\varphi'',-\varphi''')^{T},\label{GM14d}
\end{eqnarray}
Then, 
for $g$ satisfying (P1) or (P2), the vector field $V$ is reversible, in the sense that for all $U,c_s$
\begin{equation}
  SV(U,c_s)=-V(SU,c_s),\quad S={\rm diag}(1,-1,1,-1).\label{rever}
\end{equation}

In order to study the homoclinic solutions of (\ref{GM14b}), we first consider the linearization of \eqref{GM21a}, \eqref{GM21b} 
at the equilibrium point $U=0$ yields to the linear system $U'=LU$. The characteristic equation is indeed
\begin{equation}\label{GM22}
 \lambda^4 - b \lambda^2 + a=0,
\end{equation}
where $a=a(c_{s}), b=b(c_{s})$ are given by (\ref{GM14c}).
The distribution of the roots of \eqref{GM22} in the
$(b,a)$-plane, given in Figure \ref{GMfig1}, reproduces the arguments of \cite{Champ}. 
\begin{figure}[htbp]
\centering
{\includegraphics[width=1\textwidth]{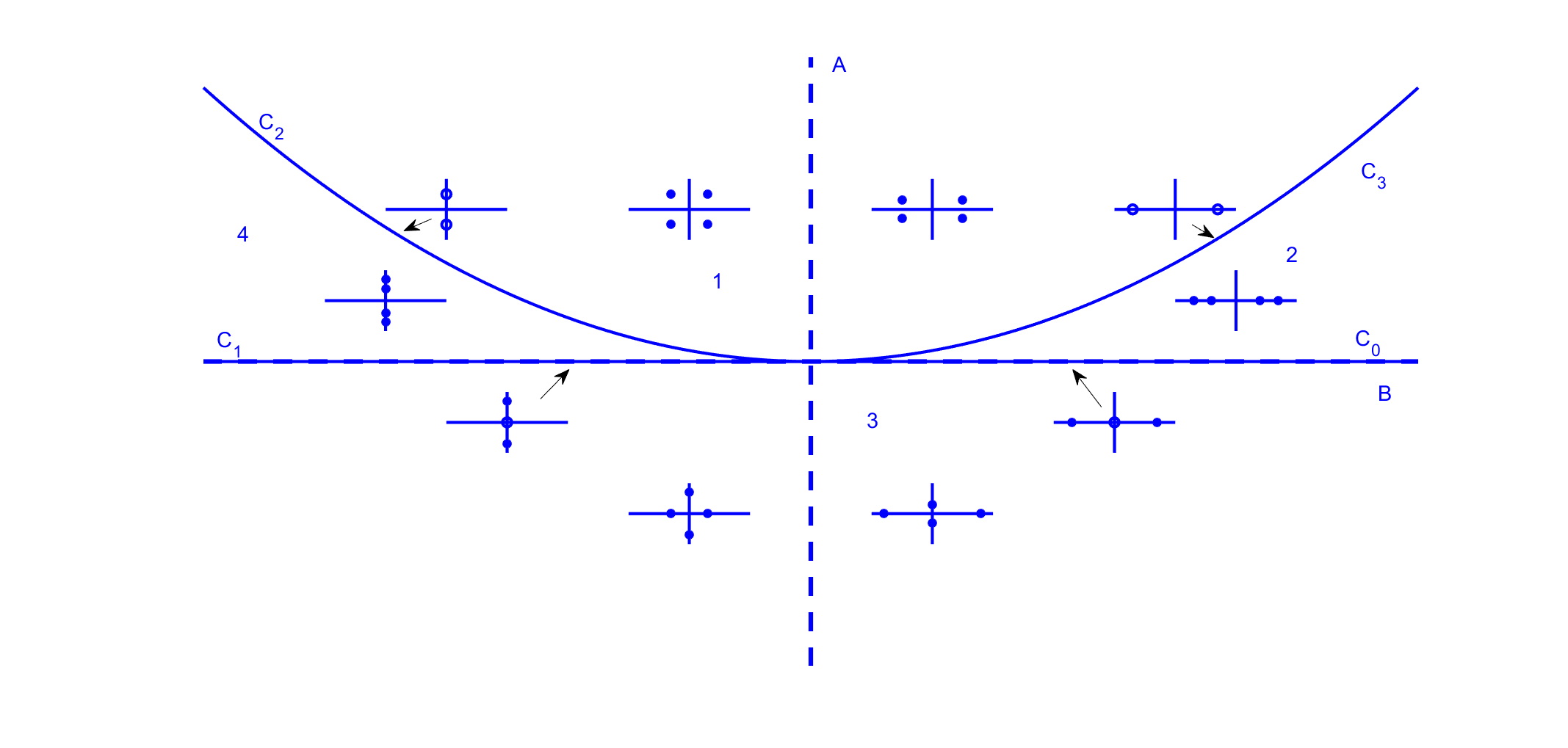}}
\caption{Linearization at the origin of (\ref{GM21a}) (as in Figure 1 of \cite{Champ}): Regions $1$ to $4$ in the $(b,a)$-plane, delimited by the bifurcation curves $\mathbb{C}_{0}$ to $\mathbb{C}_{3}$, and schematic representation of the position in the complex plane of the roots of (\ref{GM22}) for each curve and region. (Dot: simple root, larger dot: double root.)}
\label{GMfig1}
\end{figure}
Thus the linearized  dynamics
is determined  by the four regions $1$ to $4$, separated by the bifurcation curves
\begin{eqnarray*}
  \mathcal{C}_{0}(c_{s}) &=& \{(b,a) | a=0, b>0 \}=\{c_{s}=\epsilon, \frac{\alpha c_{s}-\eta}{\gamma-\beta c_{s}}>0\},   \\
  \mathcal{C}_{1}(c_{s}) &=& \{(b,a) | a=0, b<0 \}=\{c_{s}=\epsilon, \frac{\alpha c_{s}-\eta}{\gamma-\beta c_{s}}<0\},\\
  \mathcal{C}_{2}(c_{s}) &=& \{(b,a) | a>0, b=-2\sqrt{a} \}\\
&=&\{\frac{c_{s}-\epsilon}{\beta c_{s}-\gamma}>0, \frac{\alpha c_{s}-\eta}{\gamma-\beta c_{s}}
=-2\sqrt{\frac{c_{s}-\epsilon}{\beta c_{s}-\gamma}}\}, \\
  \mathcal{C}_{3}(c_{s}) &=&\{(b,a) | a>0, b=2\sqrt{a} \}\\
&=&\{\frac{c_{s}-\epsilon}{\beta c_{s}-\gamma}>0, \frac{\alpha c_{s}-\eta}{\gamma-\beta c_{s}}=2\sqrt{\frac{c_{s}-\epsilon}{\beta c_{s}-\gamma}}\}.
\end{eqnarray*}
In terms of the spectrum of $L(c_{s})$, $\mathbb{C}_{0}$ is characterized by the presence of two zero eigenvalues and other two, nonzero and real; in $\mathbb{C}_{1}$, there are two zero eigenvalues and other two, imaginary; in $\mathbb{C}_{2}$, the spectrum consists of a double complex conjugate pair of imaginary eigenvalues $\pm i\sqrt{b/2}$, while in $\mathbb{C}_{3}$ we have two double real eigenvalues $\pm\sqrt{b/2}$. The four curves meet at the origin $(b,a)=(0,0)$, where $L(c_{s})$ has only the zero eigenvalue with geometric multiplicity equals one, cf. Figure \ref{GMfig1}.

The theory of normal forms consists of finding a change of variables to transform, locally near the equilibrium, the system into a simpler one (the normal form) up to some fixed order and from which the dynamics of the original is better analyzed. The local transformation is polynomial and can be computed from the resolution of a sequence of linear problems. The dynamics of the resulting normal form is determined by the linear part. The theory, first introduced by Poincar\'e and Birkhoff, was developed by Arnold, \cite{Arnold}, and has many applications, see e.~g. \cite{HaragusI2011} and references therein. Its particular use to determine the existence of solitary wave solutions is also widely considered in the literature (cf. e.~g. \cite{Champ,ChampS,ChampT,DDS1,IK} and references therein). On the other hand, the search for solitary waves as homoclinic trajectories sometimes  requires the introduction of bifurcation parameters in order to analyze the emergence of structures via the bifurcation theory and the application of the corresponding version of the normal form theorem. In addition, of particular relevance is the influence of symmetries and reversible transformations, which are inherited by the normal form, see \cite{HaragusI2011} for details.

One of the applications of the normal form concerns the analysis of the systems obtained from a center manifold reduction. For the case at hand, we will consider the approach developed in \cite{Iooss95}, based on determining a normal form of the reversible vector field given by \eqref{GM21a}, \eqref{GM21b} in a general way and then discussing its principal part near the bifurcation curves $\mathcal{C}_{j}, j=0,1,2,$ with specific bifurcation parameter $\mu$ for each curve. The case of $\mathcal{C}_{3}$ will be explored computationally will the help of the references, cf. \cite{Champ,ChampT,Deveney1976,Belyakov}. The conclusions will be illustrated by their application to several types of Rosenau equations.

The corresponding normal form theorem (cf. e.~g. \cite{HaragusI2011}) establishes, in the case of \eqref{GM21a}, \eqref{GM21b} and for any positive integer $m\geq 2$, the existence of neighborhoods $V_{1}$ of $U=0$ and $V_{2}$ of $\mu=0$ such that for any $\mu\in V_{2}$ there is a polynomial $\Phi(\mu,\cdot):\mathbb{R}^{4}\rightarrow\mathbb{R}^{4}$ of degree $m$ satisfying:
\begin{itemize}
\item[(i)] The coefficients of $\Phi$ are smooth functions of $\mu$ and
\begin{eqnarray*}
\Phi(0,0)=0,\quad \partial_{U}\Phi(0,0)=0.
\end{eqnarray*}
\item[(ii)] For $V\in V_{1}$, the change of variable
\begin{eqnarray*}
U=V+\Phi(\mu,V),\label{GM25}
\end{eqnarray*}
transforms (\ref{GM21a}) into the normal form
\begin{eqnarray*}
V'=LV+N(\mu,V)+\rho(\mu,V),\label{GM26}
\end{eqnarray*}
where
\begin{itemize}
\item[(a)] For any $\mu\in V_{2}$, $N(\mu,\cdot):\mathbb{R}^{4}\rightarrow\mathbb{R}^{4}$ is a polynomial of degree $m$, with coefficients which are smooth functions of $\mu$ and $N(0,0)=\partial_{V}N(0,0)=0$.
\item[(b)] The equality
\begin{eqnarray*}
N(\mu,e^{tL^{*}}V)=e^{tL^{*}}N(\mu,V),
\end{eqnarray*}
holds for all $(t,V)\in \mathbb{R}\times \mathbb{R}^{4}$ and $\mu\in V_{2}$, and where $L^{*}$ denotes the adjoint of $L$. This is equivalent, \cite{HaragusI2011}, to the condition
\begin{eqnarray*}
\partial_{V}N(\mu,V)L^{*}V=L^{*}N(\mu,V),\quad V\in\mathbb{R}^{4}, \mu\in V_{2}.
\end{eqnarray*}
\item[(c)] $\rho$ is smooth in $(\mu,V)\in V_{2}\times V_{1}$ and
\begin{eqnarray*}
\rho(\mu,V)=o\left(||V||^{m}\right),
\end{eqnarray*}
for all $\mu\in V_{2}$ and where $||\cdot ||$ denotes the Euclidean norm in $\mathbb{R}^{4}$.
\item[(d)] The polynomials $\Phi(\mu,\cdot), N(\mu,\cdot)$ satisfy, for all $\mu\in V_{2}$
\begin{eqnarray*}
S\Phi(\mu,V)=\Phi(\mu,SV),\quad SN(\mu,V)=-N(\mu,SV),\quad V\in\mathbb{R}^{4},
\end{eqnarray*}
where $S$ is given in (\ref{rever}).
\end{itemize}
\end{itemize}
We note that (\ref{GM21a}), (\ref{GM21b}) can alternatively be written as
\begin{equation}\label{GM27a}
U^{\prime}=L_{0}U+\widetilde{R}(U,c_s),
\end{equation}
where $\widetilde{R}(U,c_s)=R_{11}(U,c_{s})+R(U,c_{s})$,
\begin{equation}\label{GM27b}
L_{0}=
\begin{pmatrix}
0 & 1 & 0 & 0  \\
0 & 0 & 1 & 0 \\
0 & 0 & 0 & 1 \\
0 & 0 &0 & 0 \\
\end{pmatrix},
\hspace*{20pt}
R_{11}(U,c_s)=
\begin{pmatrix}
0 \\
0 \\
0\\
-a(c_{s})U_{1}+b(c_{s})U_{3} 
\end{pmatrix},
\end{equation}
and observe that the approach made in \cite{Iooss95} can be used here to compute a general expression of a reversible normal form for (\ref{GM27a}), (\ref{GM27b}). We notice that the approach in \cite{Iooss95} initially assumes a regular, nonlinear function $\widetilde{R}$ at least quadratic near the origin. However, the treatment of the linear term $R_{11}$ in (\ref{GM27b}) (which can be considered quadratic as $a$ and $b$ will involve some bifurcation parameters) can be included in the construction of the normal form from the arguments used in \cite{Choudhury2007}, based on the preservation of the eigenvalue structure. All this leads to a normal form for (\ref{GM27a}) as, \cite{Iooss95}
\begin{eqnarray}
U'&=&L_{0}U+P_{4}(U_{1},V_{2},V_{4})\begin{pmatrix}0\\0\\0\\1\end{pmatrix}+
P_{2}(U_{1},V_{2},V_{4})\begin{pmatrix}0\\U_{1}\\U_{2}\\U_{3}\end{pmatrix}\nonumber\\
&&+P_{0}(V_{2},V_{4})\begin{pmatrix}0\\V_{2}\\W_{2}\\X_{2}\end{pmatrix}+
P_{1}(U_{1},V_{2},V_{4})\begin{pmatrix}V_{3}\\W_{3}\\X_{3}\\Y_{3}\end{pmatrix}\nonumber\\
&&+P_{3}(U_{1},V_{2},V_{4})\begin{pmatrix}0\\0\\V_{3}\\W_{3}\end{pmatrix},\label{GM28}
\end{eqnarray}
where
\begin{eqnarray*}
&&V_{2}=U_{2}^{2}-2U_{1}U_{3},\quad V_{3}=U_{2}^{3}-3U_{1}U_{2}U_{3}+3U_{1}^{2}U_{4},\\
&&V_{4}=3U_{2}^{2}U_{3}^{2}-6U_{2}^{3}U_{4}-8U_{1}U_{3}^{3}+18U_{1}U_{2}U_{3}U_{4}-9U_{1}^{2}U_{4}^{2},\\
&&W_{2}=-3U_{1}U_{4}+U_{2}U_{3},\quad W_{3}=3U_{1}U_{2}U_{4}-4U_{1}U_{3}^{2}+U_{2}^{2}U_{3},\\
&&X_{2}=-3U_{2}U_{4}+2U_{3}^{2},\quad X_{3}=-3U_{1}U_{3}U_{4}+3U_{2}^{2}U_{1}-U_{2}U_{3}^{2},\\
&&Y_{3}=3U_{2}U_{3}U_{4}-\frac{4}{3}U_{3}^{3}-3U_{1}U_{4}^{2},
\end{eqnarray*}
and $P_{j}, j=0,1,2,3,4$, are polynomials in their arguments with
\begin{eqnarray*}
P_{4}(U_{1},V_{2},V_{4})&=&\mu_{2}U_{1}+\nu_{1}U_{1}^{2}+\nu_{2}V_{2}+O(||U||^{3}),\\
P_{2}(U_{1},V_{2},V_{4})&=&\mu_{1}+\nu_{3}U_{1}+(||U||^{2}),\\
P_{0}(V_{2},V_{4})&=&\nu_{4}+O(||U||^{2}),
\end{eqnarray*}
where $\nu_{j}, j=1,2,3,4,$ are constants and
\begin{eqnarray}
\mu_{1}=\frac{b}{3},\quad \mu_{2}=\left(\frac{b}{3}\right)^{2}-a=\mu_{1}^{2}-a,\label{GM29}
\end{eqnarray}
with $a,b$ given by (\ref{GM14c}). In order to apply here the discussion, established in \cite{Iooss95}, on the normal form near each bifurcation curve $\mathcal{C}_{j}, j=0,1,2,$ it is also convenient to write the bifurcation curves in the $(\mu_{1},\mu_{2})$ plane, as
\begin{eqnarray*}
  \mathcal{C}_{0} &=& \{(\mu_{1},\mu_{2}) | \mu_{1}>0, \mu_{2}=\mu_{1}^{2}\},   \\
  \mathcal{C}_{1} &=& \{(\mu_{1},\mu_{2}) | \mu_{1}<0, \mu_{2}=\mu_{1}^{2}\},\\
  \mathcal{C}_{2} &=&  \{(\mu_{1},\mu_{2}) | \mu_{1}<0, \mu_{2}=-\frac{5}{4}\mu_{1}^{2}\}, \\
  \mathcal{C}_{3}&=& \{(\mu_{1},\mu_{2}) | \mu_{1}>0, \mu_{2}=-\frac{5}{4}\mu_{1}^{2}\}.
\end{eqnarray*}
(Cf. Figure 1 of \cite{Iooss95}.)

Note that the linearization of (\ref{GM28}) at the origin leads to the system
\begin{eqnarray*}
U'=L(\mu_{1},\mu_{2})U,
\end{eqnarray*}
where
\begin{equation}\label{GM29e}
L(\mu_{1},\mu_{2})=
\begin{pmatrix}
0 & 1 & 0 & 0  \\
\mu_{1} & 0 & 1 & 0 \\
0 & \mu_{1} & 0 & 1 \\
\mu_{2} & 0 &\mu_{1} & 0 \\
\end{pmatrix},
\end{equation}
which, using (\ref{GM29}), has the same characteristic polynomial as $L$ in (\ref{GM21b}), preserving then the eigenvalue structure, cf. \cite{Choudhury2007}.

Following the discussion in \cite{Iooss95}, the principal part of the normal form (\ref{GM28}) will be described near the curves $\mathcal{C}_{j}, j=0,1,2$, in order to show the existence of different homoclinic structures leading to different types of solutions of (\ref{GM14b}). The results will be specified for typical choices of the nonlinear term $g$ and illustrated numerically.

\subsection{Near $\mathcal{C}_{0}$}
We consider $\mu=-a$ as bifurcation parameter, in such a way that $\mu_{2}=\mu_{1}^{2}+\mu$, $\mathcal{C}_{0}$ is characterized by the conditions $\mu=0, \mu_{1}>0$, and the points near $\mathcal{C}_{0}$ depend on the (small) values of $\mu$. If $L$ denotes the matrix (\ref{GM29e}) at $\mu=0$, then a basis of the generalized eigenspace ${\rm Ker}L^{2}={\rm span}(\xi_{0},\xi_{1})$ with $L\xi_{1}=\xi_{0}$ is given explicitly by, \cite{Iooss95,Choudhury2007}
\begin{eqnarray*}
\xi_{0}=(1,0,-\mu_{1},0)^{T},\quad \xi_{1}=(0,1,0,-2\mu_{1})^{T},
\end{eqnarray*}
with, in addition, $S\xi_{0}=\xi_{0}, S\xi_{1}=-\xi_{1}$. Using the Center Manifold Theorem, \cite{Wiggins90,HaragusI2011}, the dynamics of bounded solutions near $\mathcal{C}_{0}$ can be studied on the two-dimensional center manifold. From the change of variables
\begin{eqnarray*}
U=A\xi_{0}+B\xi_{1}+\Phi(\mu,A,B),
\end{eqnarray*}
with suitable $\Phi$, the principal (linear and quadratic) part of the normal form (\ref{GM28}) takes the form, \cite{Iooss95,HaragusI2011}
\begin{eqnarray}
\frac{dA}{dX}&=&B,\nonumber\\
\frac{dB}{dX}&=&-\frac{1}{3\mu_{1}}\left(\mu A+(\nu_{1}+2\mu_{1}(\nu_{2}-\nu_{3})A^{2}\right).\label{GM210a}
\end{eqnarray}
(The variables $A,B,X$ can be rescaled in order to reformulate (\ref{GM210a}) as a nonsingular system as $\mu_{1}\rightarrow 0$, \cite{Iooss95}.) When $g$ in (\ref{GM21b}) is quadratic in the components of $U$ then (\ref{GM210a}) can be reduced to, \cite{Choudhury2007}
\begin{eqnarray*}
\frac{dA}{dX}&=&B,\nonumber\\
\frac{dB}{dX}&=&-\frac{\mu}{3\mu_{1}}A-\frac{2}{3}A^{2},\label{GM210b}
\end{eqnarray*}
which admits a homoclinic to zero solution, in the form of a Classical Solitary Wave 
\begin{eqnarray*}
A(X)=\frac{9a}{4b}{\rm sech}^{2}\left(\frac{1}{2}\sqrt{\frac{a}{b}}X\right),
\end{eqnarray*}
when $\mu<0$ (that is $a>0$, region 2 in Figure \ref{GMfig1}). The persistence of the homoclinic solution in the original system (\ref{GM21a}), which corresponds to a CSW solution of (\ref{GM1}) can be proved following \cite{IK,Champ}. These conclusions are applied to different families of equations of Rosenau-type in Tables \ref{GMtav1}-\ref{GMtav4}, where the range of speed $c_{s}>0$ ensuring the existence of CSW's is specified in each case when $g$ is quadratic in the components of $U$. (The existence of CSW solutions of the Rosenau equation will be considered in section \ref{sec4}, as well as the generalized case $g(u)=\frac{u^{p+1}}{p+1}, p\geq 1$.) The justification of Tables \ref{GMtav1}-\ref{GMtav4} is in Appendix \ref{GMapp1}, where the curves $\mathcal{C}_{j}, j=0,1,2,3,$ and the regions of Figure \ref{GMfig1} are described for each equation of Rosenau type considered in this paper. The generation of some of the CSW profiles are illustrated in Figure \ref{GMfig3}. (All the figures are generated from the experiments performed using the Petviashvili method, \cite{Petv1976}; the numerical procedure is described in Appendix \ref{GMapp2}.)

Two properties of the profiles can be observed from Figure \ref{GMfig3}. Note first that the wave is smooth, even, and decreases (if it is of elevation) or increases (if it is of depression) fast in both directions away from its maximum (or minimum) point at $X=0$. Furthermore, the corresponding phase portrait and the way how the homoclinic orbit approaches zero at $\pm\infty$ suggest that the profile goes to zero exponentially as $|X|\rightarrow\infty$. This behaviour can be theoretically checked from the application of some standard results (cf. e.~g. \cite{BonaLi}) in the case of nonlinear terms of the form $g(u)=\frac{u^{p+1}}{p+1}, p\geq 1$.

\begin{figure}[htbp]
\centering
\subfigure[]
{\includegraphics[width=6.27cm,height=5cm]{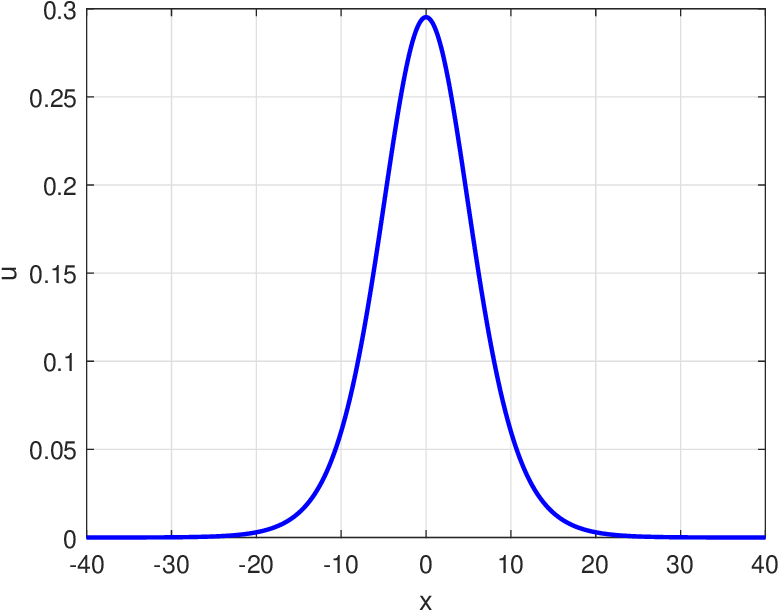}}
\subfigure[]
{\includegraphics[width=6.27cm,height=5cm]{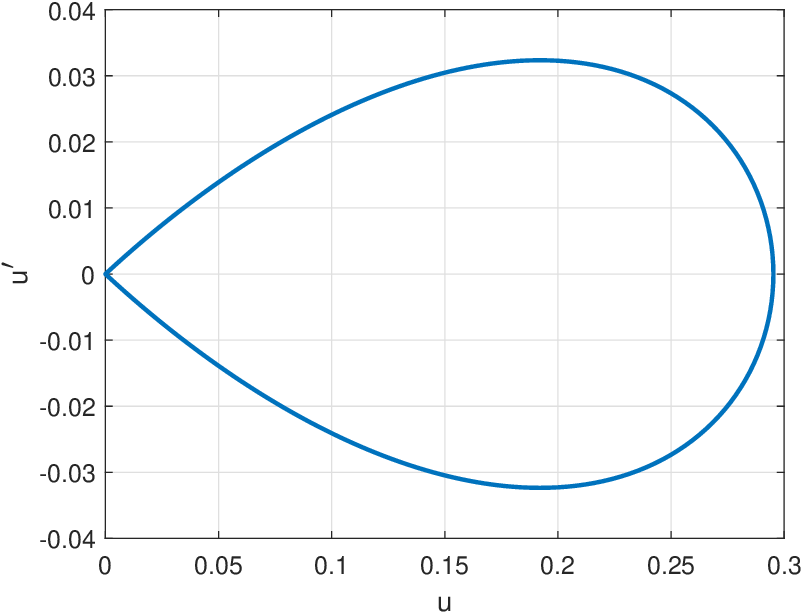}}
\subfigure[]
{\includegraphics[width=6.27cm,height=5cm]{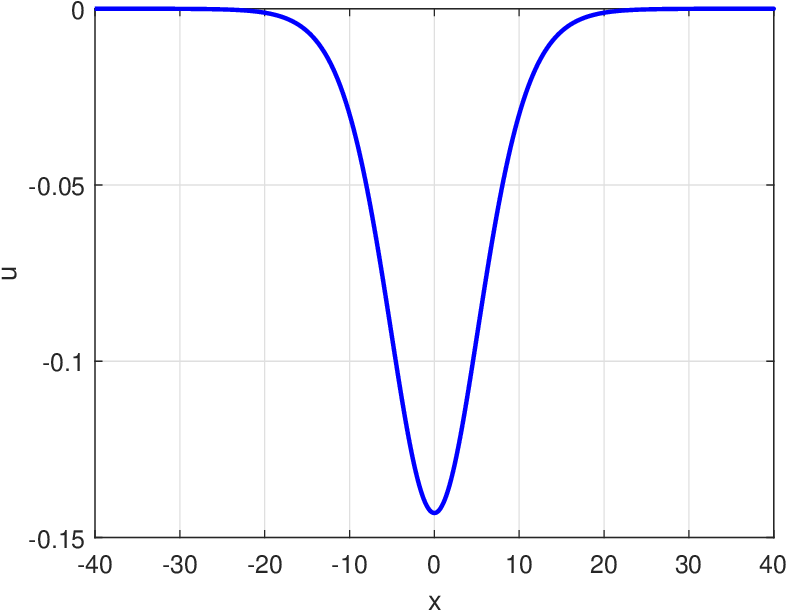}}
\subfigure[]
{\includegraphics[width=6.27cm,height=5cm]{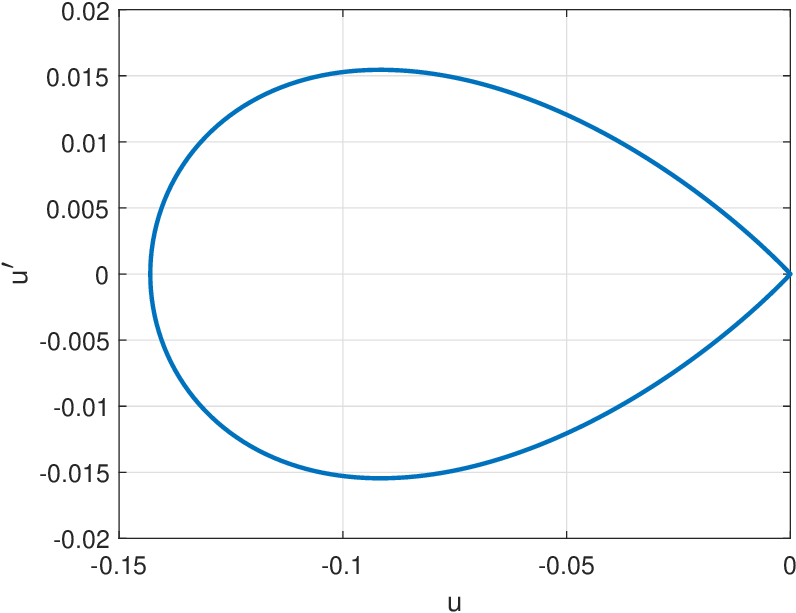}}
\subfigure[]
{\includegraphics[width=6.27cm,height=5cm]{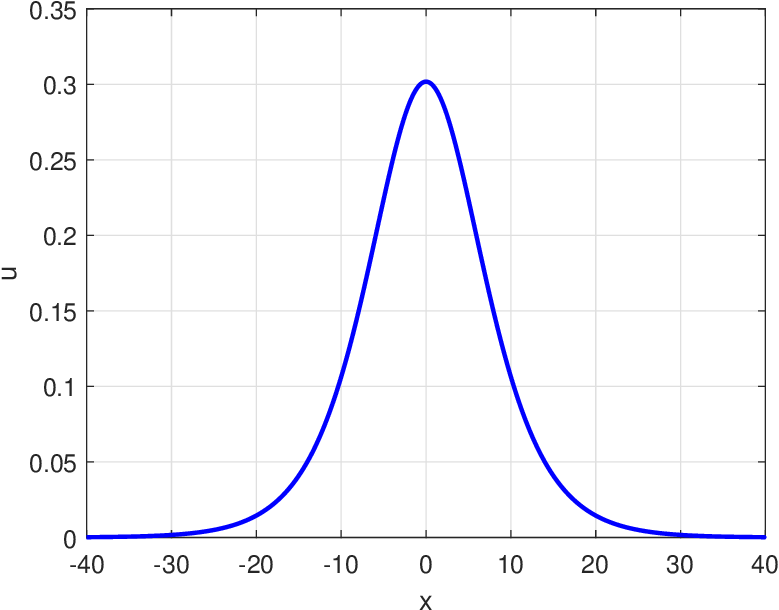}}
\subfigure[]
{\includegraphics[width=6.27cm,height=5cm]{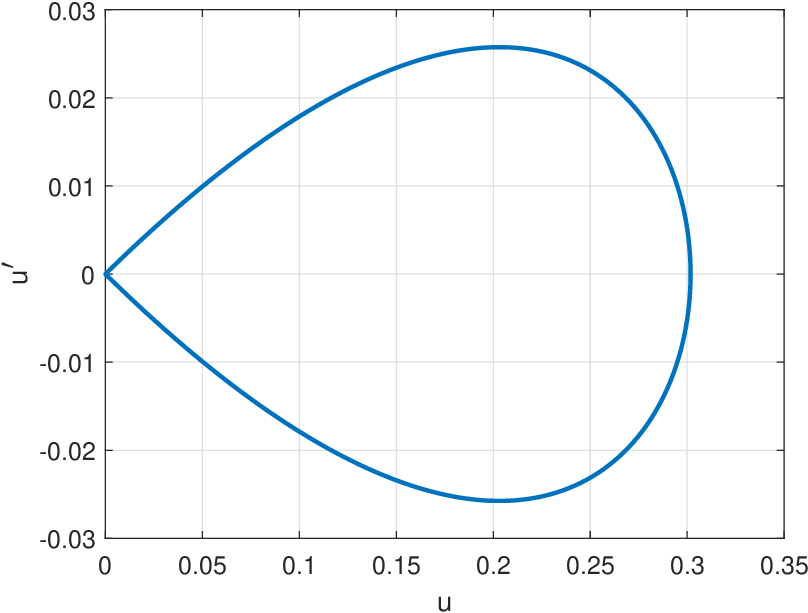}}
\caption{CSW generation, $u$ profiles and phase portraits. (a), (b) Rosenau-RLW equation with $\alpha=-1, \epsilon=\beta=1, g(u)=u^{2}/2; c_{s}=1.1$; (c), (d) Rosenau-Kawahara equation with $\eta=-1/2, \epsilon=\beta=1, \gamma=2, g(u)=u^{2}/2, c_{s}=y_{-}+\epsilon+0.01\approx 0.951$; (e), (f) Rosenau-RLW-Kawahara equation with $\alpha=\gamma=-1, \epsilon=\beta=1, \eta=1, g(u)=u^{2}/2, c_{s}=z_{+}+\epsilon-0.565\approx 1.102$. The values of $y_{-}$ and $z_{+}$ are given in Tables \ref{GMtav3} and \ref{GMtav4} resp.}
\label{GMfig3}
\end{figure}

\begin{remark}
When $\mu>0$ is small (region 3 of Figure \ref{GMfig1}, close to $\mathcal{C}_{0}$) NFT predicts the existence of periodic solutions of the reduced system, as in \cite{IK}. They correspond to periodic traveling wave (PTW) solutions of (\ref{GM1}): The description of the region for several families of Rosenau equations is given in Appendix \ref{GMapp1} and some approximate PTW are shown in Figure \ref{GMfig4}.
\end{remark}

\begin{figure}[htbp]
\centering
\subfigure[]
{\includegraphics[width=6.27cm,height=5cm]{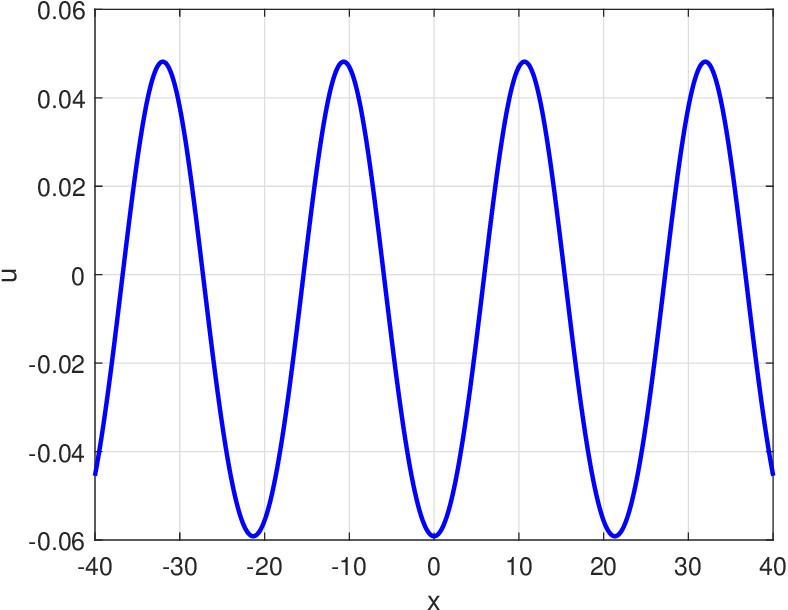}}
\subfigure[]
{\includegraphics[width=6.27cm,height=5cm]{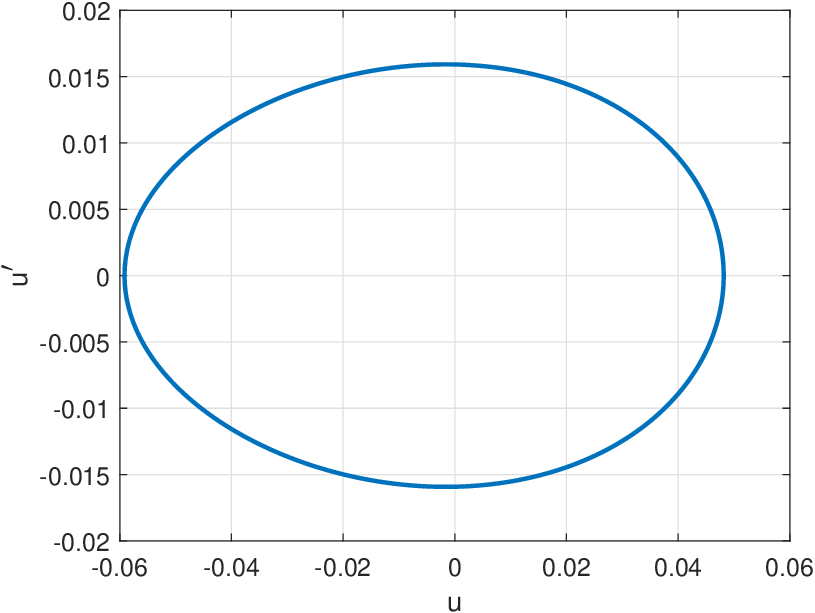}}
\subfigure[]
{\includegraphics[width=6.27cm,height=5cm]{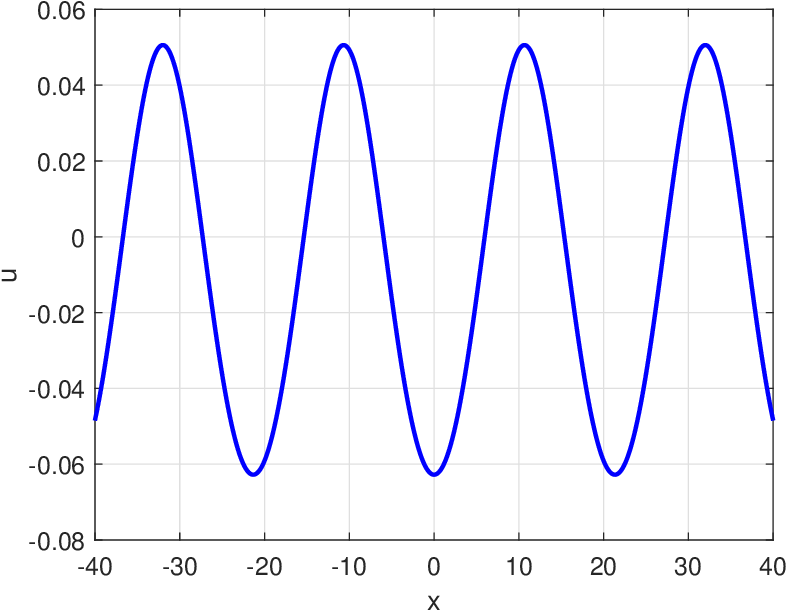}}
\subfigure[]
{\includegraphics[width=6.27cm,height=5cm]{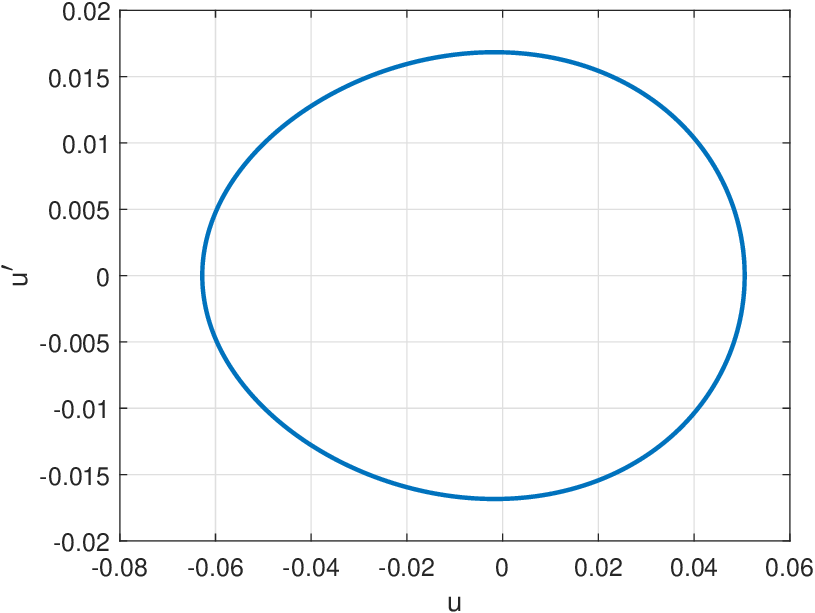}}
\subfigure[]
{\includegraphics[width=6.27cm,height=5cm]{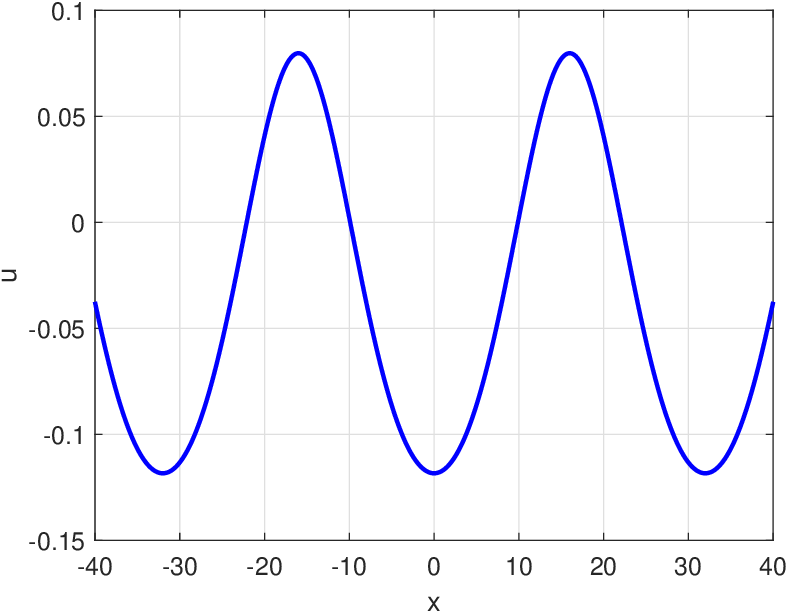}}
\subfigure[]
{\includegraphics[width=6.27cm,height=5cm]{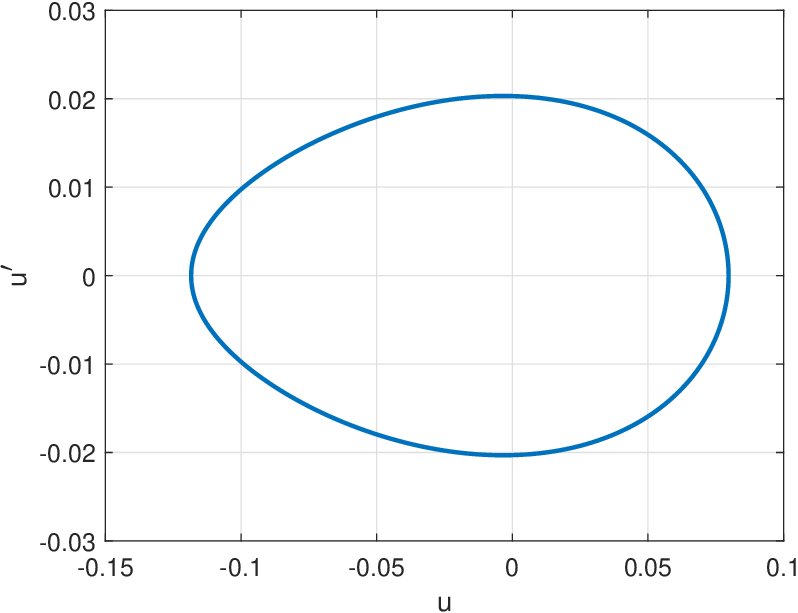}}
\caption{PTW generation, $u$ profiles and phase portraits. (a), (b) Rosenau-KdV equation with $\eta=1, \epsilon=\beta=1, g(u)=u^{2}/2; c_{s}=0.9$; (c), (d) Rosenau-Kawahara equation with $\eta=1, \epsilon=1, \beta=2, \gamma=1, g(u)=u^{2}/2, c_{s}=0.9$; (e), (f) Rosenau-RLW-Kawahara equation with $\alpha=\gamma=-1, \epsilon=\beta=1, \eta=1, g(u)=u^{2}/2, c_{s}=0.9$.}
\label{GMfig4}
\end{figure}

\subsection{Near $\mathcal{C}_{1}$}
We consider the same parameters defined in the previous section. In $\mathcal{C}_{1}$ (where $\mu_{1}<0$) the spectrum of $L$ consists of two imaginary simple eigenvalues $\pm i\sqrt{-3\mu_{1}}$ and $0$, which is double. A basis of the generalized eigenspace may consist of the vectors $\xi_{0}, \xi_{1}$, defined above, along with $\zeta_{0},\overline{\zeta_{0}}$, where
\begin{eqnarray*}
\zeta_{0}=(1, i\sqrt{-3\mu_{1}}, 2\mu_{1}, i\mu_{1}\sqrt{-3\mu_{1}})^{T}.
\end{eqnarray*}
For $\mu>0$, the double eigenvalue splits into two real eigenvalues (region 3 of Figure \ref{GMfig1}), \cite{Iooss95}. The change of variables
\begin{eqnarray*}
U=A\xi_{0}+B\xi_{1}+C\zeta_{0}+\overline{C}\overline{\zeta_{0}}+\Phi(\mu,A,B,C,\overline{C}),
\end{eqnarray*}
for some polynomial $\Phi$ of degree at least two, can be chosen to write (\ref{GM28}) in a normal form, \cite{IK}
\begin{eqnarray}
\frac{dA}{dX}&=&B,\nonumber\\
\frac{dB}{dX}&=&-\frac{1}{3\mu_{1}}\left(\mu A+(\nu_{1}+2\mu_{1}(\nu_{2}-\nu_{3})A^{2}+\nu_{5}|C|^{2}\right)+\cdots,\nonumber\\
\frac{dC}{dX}&=&i\sqrt{-3\mu_{1}}C\left(1+\frac{\mu}{18\mu_{1}^{2}}+\nu_{6}A+\cdots\right),\label{GM211}
\end{eqnarray}
where
\begin{eqnarray*}
\nu_{5}&=&2\nu_{1}+2\mu_{1}(\nu_{3}-7\nu_{2})+12\mu_{1}^{2}\nu_{4},\\
\nu_{6}&=&\frac{1}{9\mu_{1}^{2}}(\nu_{1}+\mu_{1}(4\nu_{3}-\nu_{2})-12\mu_{1}^{2}\nu_{4}),
\end{eqnarray*}
(which, as before, can be reformulated in order to be nonsingular as $\mu_{1}\rightarrow 0$). For a quadratic nonlinearity $g$, (\ref{GM211}) has the form, \cite{Choudhury2007}
\begin{eqnarray*}
\frac{dA}{dX}&=&B,\nonumber\\
\frac{dB}{dX}&=&-\frac{\mu}{3\mu_{1}}A-\frac{2}{3}A^{2}-\frac{4}{3}|C|^{2},\nonumber\\
\frac{dC}{dX}&=&i\sqrt{-3\mu_{1}}C\left(1+\frac{\mu}{18\mu_{1}^{2}}-\frac{1}{9\mu_{1}}\right).\label{GM212}
\end{eqnarray*}
As mentioned in e.~g. \cite{IK, Choudhury2007,Lombardi1,Lombardi2,Lombardi3}, in region 3 close to $\mathcal{C}_{1}$, there are orbits which are homoclinic to periodic orbits as $|X|\rightarrow\infty$, leading to smooth Generalized Solitary Wave solutions of (\ref{GM1}). Some of the ripples may have exponentially small amplitude and on isolated curves the oscillations may vanish and an orbit homoclinic to zero (CSW, also called embedded soliton) is formed. For the families of Rosenau equations mentioned in the introduction, the range of speed and parameters ensuring the emergence of GSW's is specified in Tables \ref{GMtav1}-\ref{GMtav4} (cf. Appendix \ref{GMapp1}) and some of approximate profiles are shown in Figure \ref{GMfig5}. In addition, close to $\mathcal{C}_{1}$, but in region 4, \cite{IK}, PTW's are also formed, cf. Figure \ref{GMfig6}.

\begin{figure}[htbp]
\centering
\subfigure[]
{\includegraphics[width=6.27cm,height=5cm]{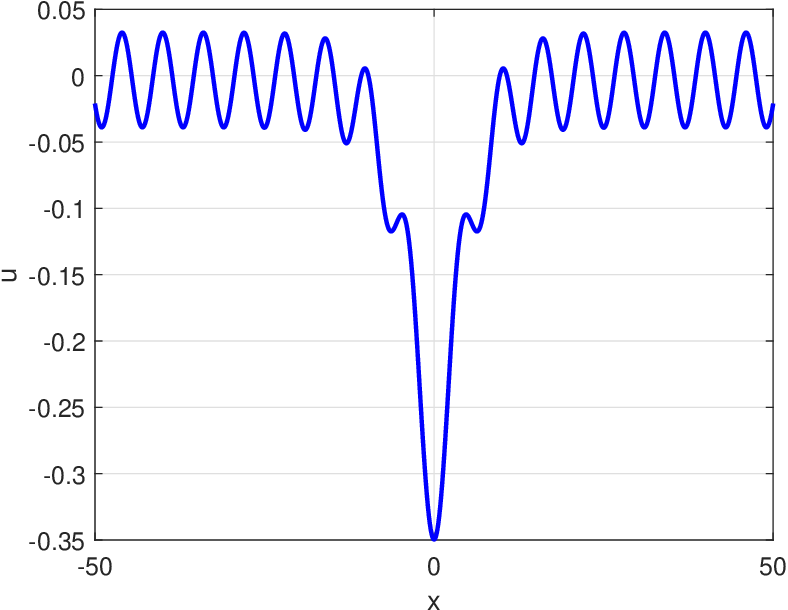}}
\subfigure[]
{\includegraphics[width=6.27cm,height=5cm]{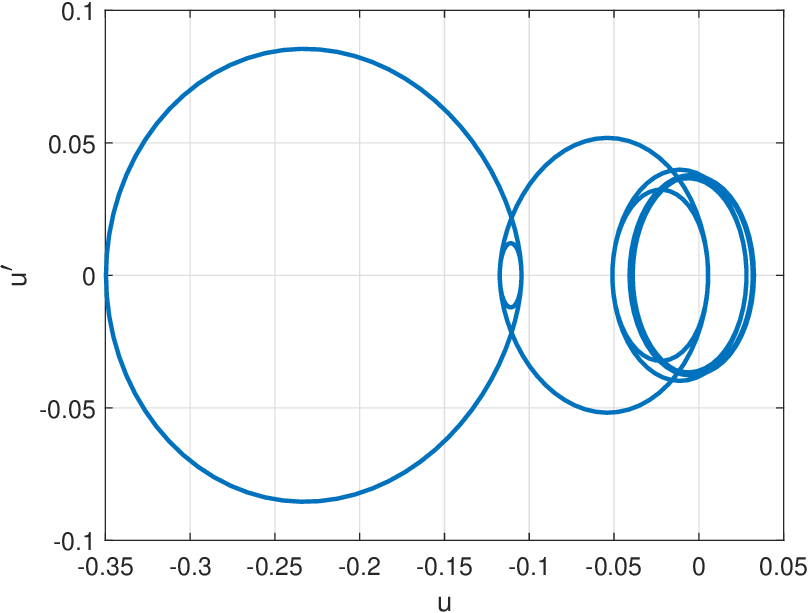}}
\subfigure[]
{\includegraphics[width=6.27cm,height=5cm]{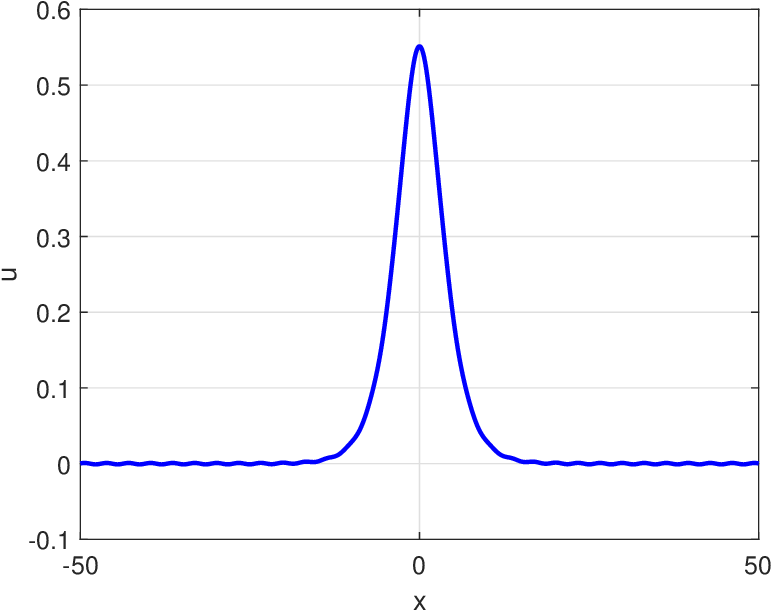}}
\subfigure[]
{\includegraphics[width=6.27cm,height=5cm]{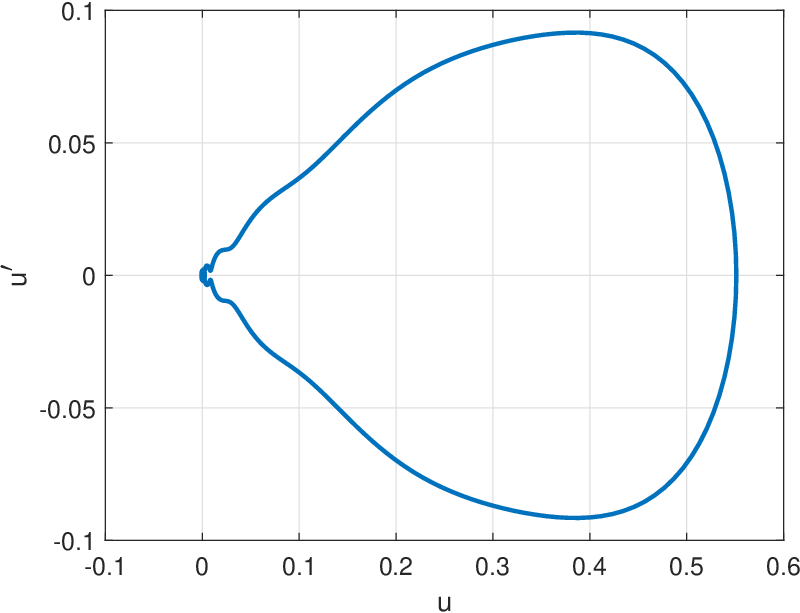}}
\subfigure[]
{\includegraphics[width=6.27cm,height=5cm]{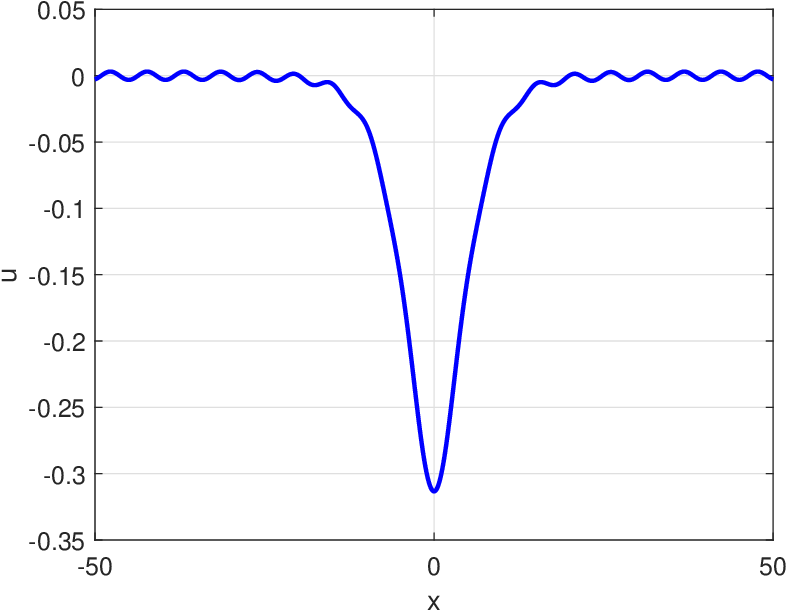}}
\subfigure[]
{\includegraphics[width=6.27cm,height=5cm]{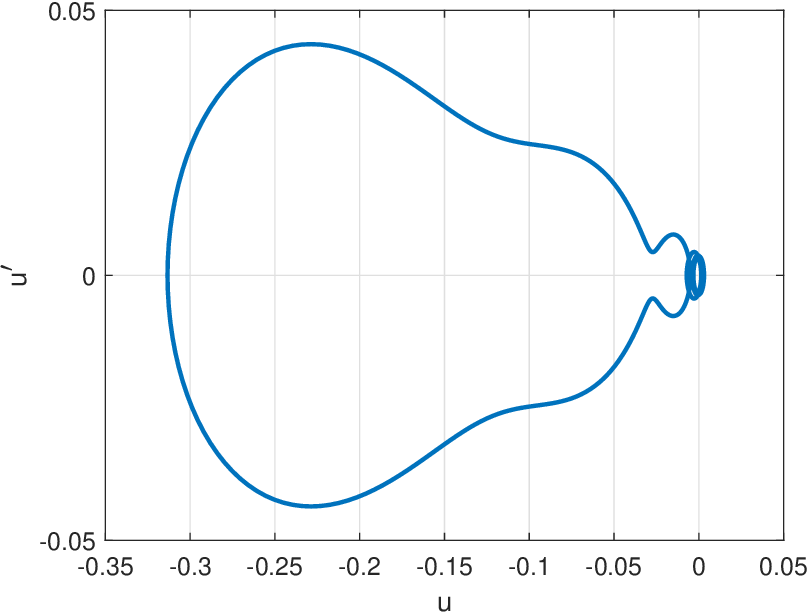}}
\caption{GSW generation, $u$ profiles and phase portraits. (a), (b) Rosenau-RLW equation with $\alpha=1, \epsilon=\beta=1, g(u)=u^{2}/2; c_{s}=0.9$; (c), (d) Rosenau-Kawahara equation with $\eta=1, \epsilon=0.25, \beta=4, \gamma=2, g(u)=u^{2}/2, c_{s}=0.43$; (e), (f) Rosenau-Kawahara equation with $\eta=-1, \epsilon=1, \beta=2, \gamma=1, g(u)=u^{2}/2, c_{s}=0.9$.}
\label{GMfig5}
\end{figure}

\begin{figure}[htbp]
\centering
\subfigure[]
{\includegraphics[width=6.27cm,height=5cm]{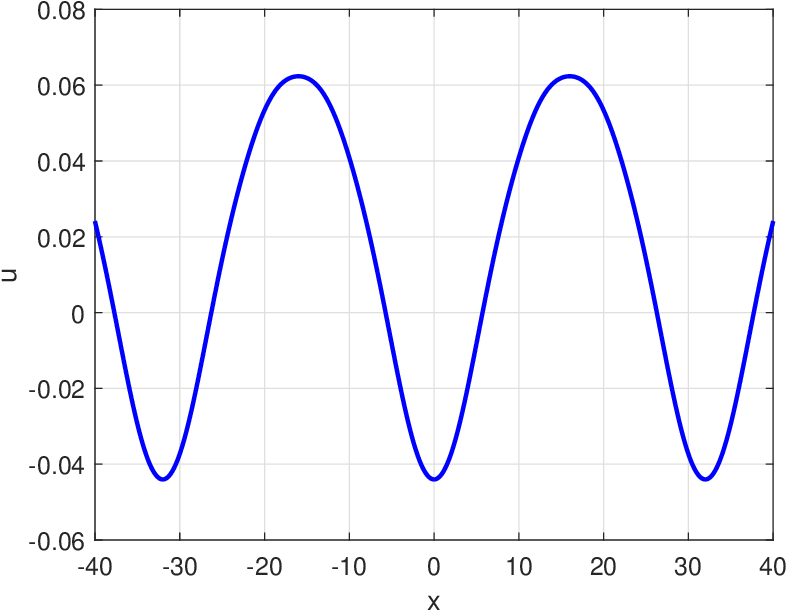}}
\subfigure[]
{\includegraphics[width=6.27cm,height=5cm]{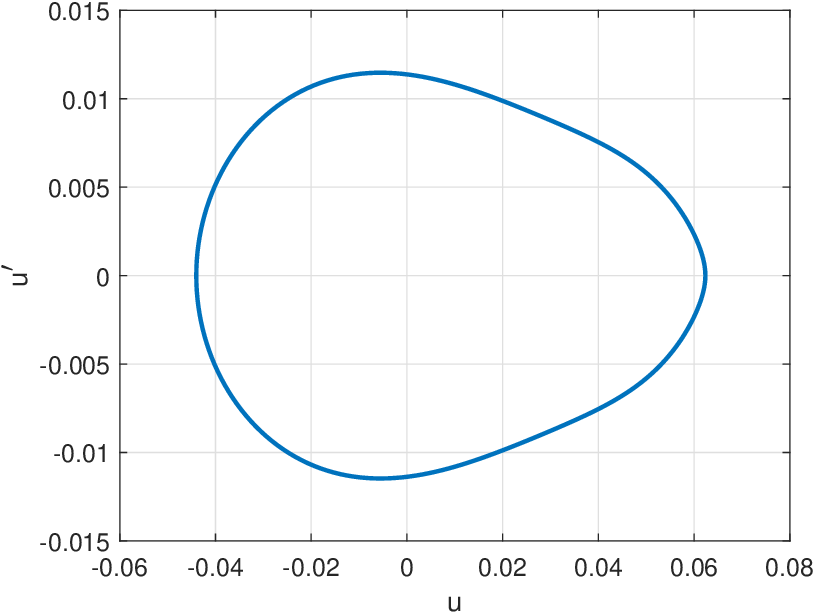}}
\subfigure[]
{\includegraphics[width=6.27cm,height=5cm]{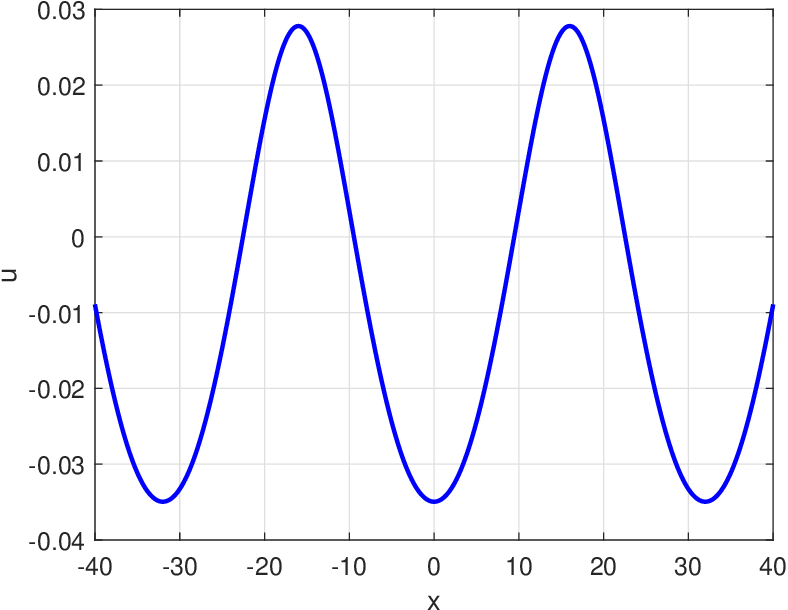}}
\subfigure[]
{\includegraphics[width=6.27cm,height=5cm]{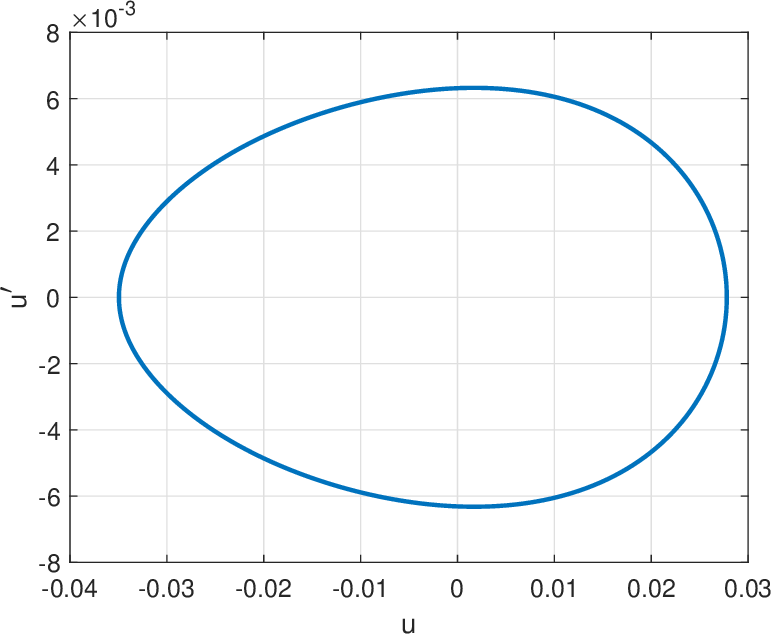}}
\subfigure[]
{\includegraphics[width=6.27cm,height=5cm]{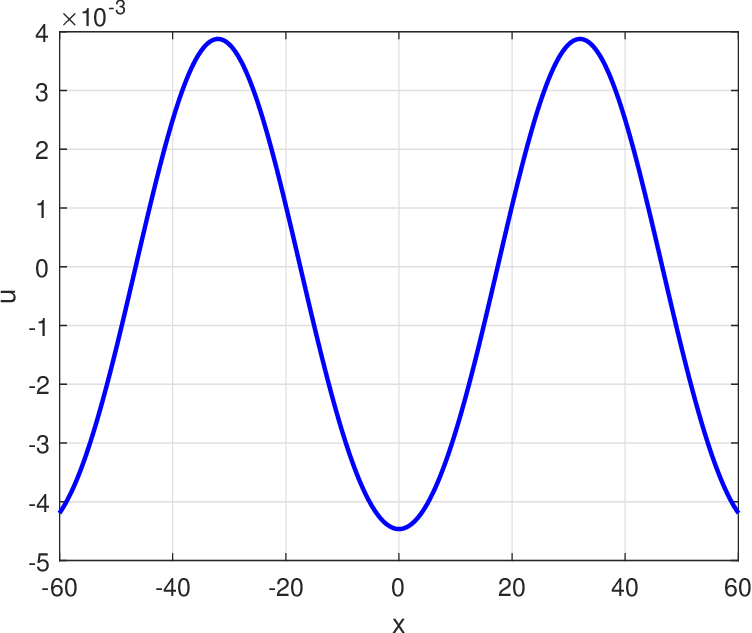}}
\subfigure[]
{\includegraphics[width=6.27cm,height=5cm]{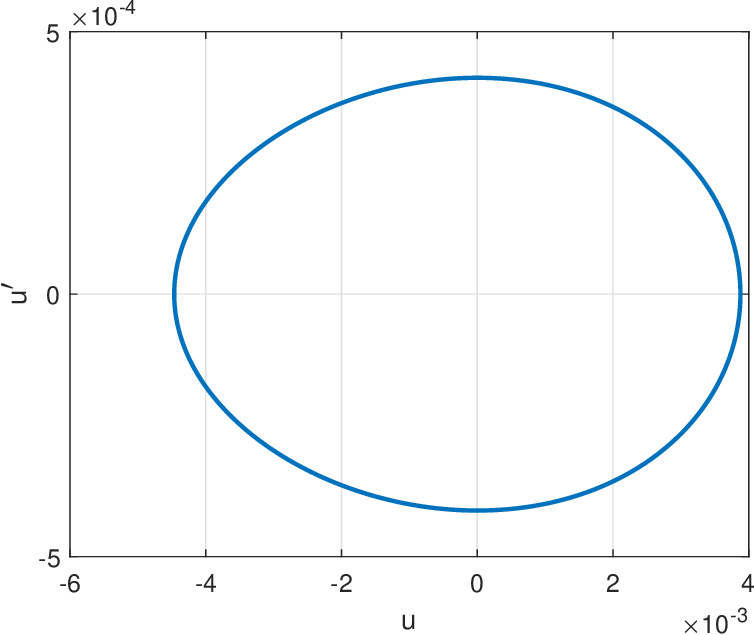}}
\caption{PTW generation, $u$ profiles and phase portraits. (a), (b) Rosenau-KdV equation with $\eta=-1, \epsilon=\beta=1, g(u)=u^{2}/2; c_{s}=1.05$; (c), (d) Rosenau-Kawahara equation with $\eta=1, \epsilon=0.25, \beta=4, \gamma=2, g(u)=u^{2}/2, c_{s}\approx 0.208$; (e), (f) Rosenau-Kawahara equation with $\alpha=0, \gamma=1, \epsilon=1/8, \beta=8, \eta=1, g(u)=u^{2}/2, c_{s}\approx 0.115$.}
\label{GMfig6}
\end{figure}

\subsection{Near $\mathcal{C}_{2}$}
In this case, according to the definition of $\mathcal{C}_{2}$ in the $(\mu_{1},\mu_{2})$ plane, we consider a bifurcation parameter $\mu$ such that
\begin{eqnarray*}
\mu_{2}=-\frac{5}{4}\mu_{1}^{2}-\mu,\quad \mu_{1}<0.
\end{eqnarray*}
Recall that the spectrum of $L$ in $\mathcal{C}_{2}$ consists of two double imaginary eigenvalues 
\begin{eqnarray*}
\lambda_{\pm}=\pm i\sqrt{\frac{-3}{2}\mu_{1}}.
\end{eqnarray*}
When $\mu>0$ is small (region 1 of Figure \ref{GMfig1}) they become two complex simple pairs of eigenvalues, symmetric with respect to the axis, while if $\mu<0$ (region 4) they become two pairs of imaginary eigenvalues which are simple. The generalized eigenspace when $\mu=0$ can be generated by the eigenvalues $\zeta_{0},\overline{\zeta_{0}}$ and the generalized eigenvectors $\zeta_{1}, \overline{\zeta_{1}}$, where, \cite{Iooss95}
\begin{eqnarray*}
\zeta_{0}=(1, \lambda_{+}, \frac{\mu_{1}}{2}, -\mu_{1}\frac{\lambda_{+}}{2})^{T},\quad
\zeta_{1}=(0,1,2\lambda_{+},\frac{5}{2}\mu_{1})^{T}.
\end{eqnarray*}
Following \cite{Iooss95,IoossA,IoossP,Champ} we can change the variable
\begin{eqnarray*}
U=A\zeta_{0}+B\zeta_{1}+\overline{A}\overline{\zeta_{0}}+\overline{B}\overline{\zeta_{1}}+\Phi(\mu,A,B,\overline{A},\overline{B}),
\end{eqnarray*}
in such a way that the following normal form for (\ref{GM28}) can be derived
\begin{eqnarray*}
\frac{dA}{dX}&=&\lambda_{+}A+B-\frac{i\mu A}{6\mu_{1}\sqrt{-6\mu_{1}}}+\cdots,\\
\frac{dB}{dX}&=&\lambda_{+}B-\frac{i\mu B}{6\mu_{1}\sqrt{-6\mu_{1}}}-\frac{\mu}{6\mu_{1}}A+\frac{76\nu_{1}^{2}}{243\mu_{1}^{3}}A|A|^{2}+\cdots
\end{eqnarray*}
After rescaling, the study made in \cite{IoossP} reveals the existence of homoclinic to zero orbits when $\mu>0$ is small (region 1 in Figure \ref{GMfig1}) corresponding to CSW but with nonmonotone decay; they are different from the ones obtained close to $\mathcal{C}_{0}$, (which are positive or negative) as well as the possible formation of orbits which are homoclinic to periodic orbits (corresponding to GSW's) and periodic orbits (corresponding to PTW's) when $\mu<0$ (region 4). The description given in Appendix \ref{GMapp1} gives the range of speeds and parameters corresponding to these two situations for each Rosenau-type equation considered in the present paper. The information corresponding to nonmonotone CSW's is displayed in Tables \ref{GMtav1}-\ref{GMtav4} and some  approximate profiles are shown in Figure \ref{GMfig7}.
\begin{figure}[htbp]
\centering
\subfigure[]
{\includegraphics[width=6.27cm,height=5cm]{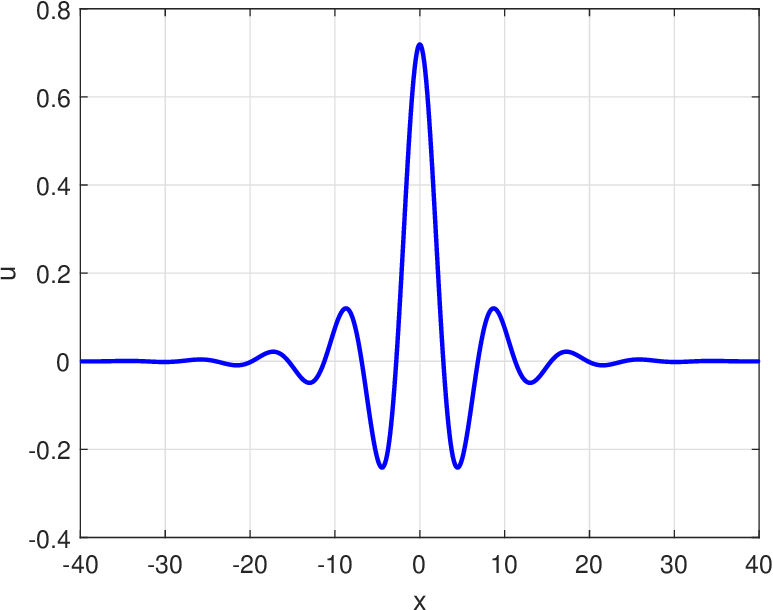}}
\subfigure[]
{\includegraphics[width=6.27cm,height=5cm]{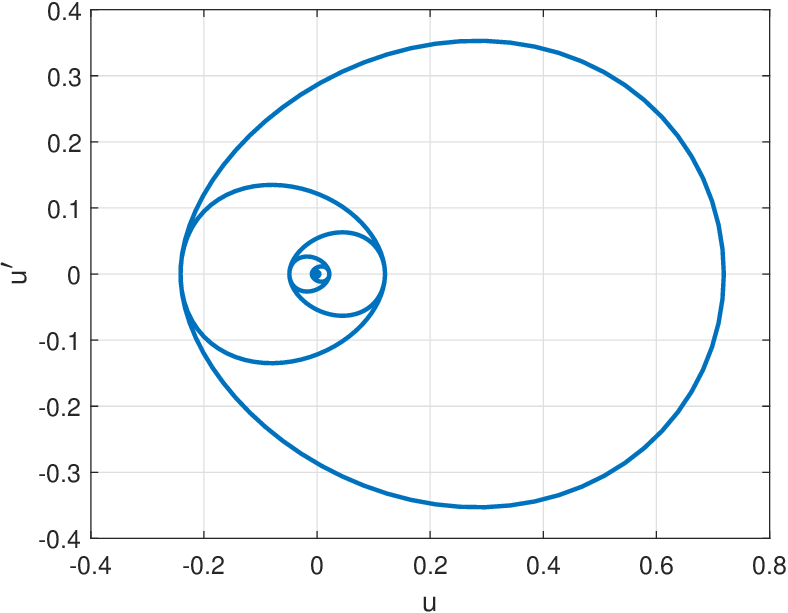}}
\subfigure[]
{\includegraphics[width=6.27cm,height=5cm]{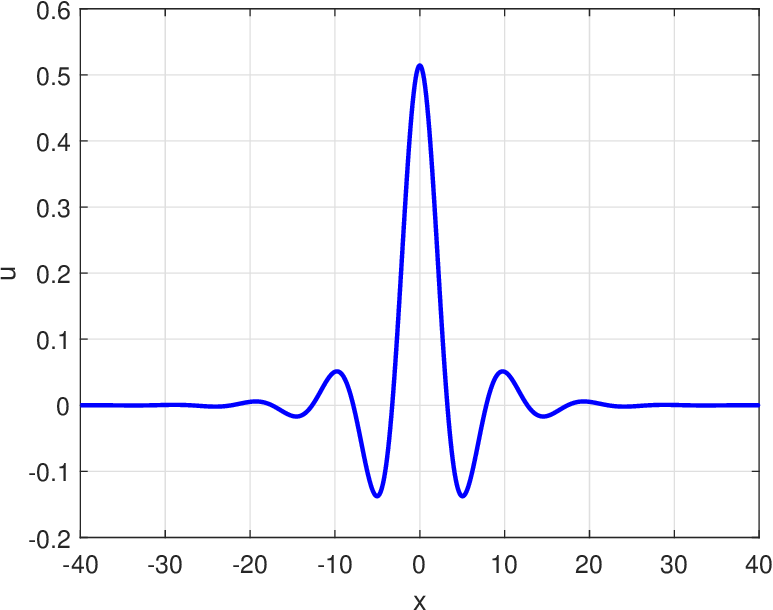}}
\subfigure[]
{\includegraphics[width=6.27cm,height=5cm]{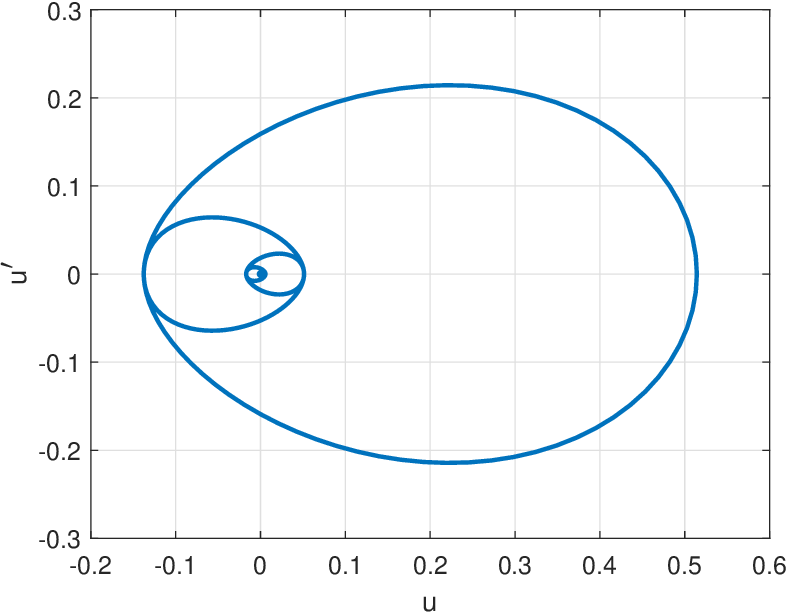}}
\subfigure[]
{\includegraphics[width=6.27cm,height=5cm]{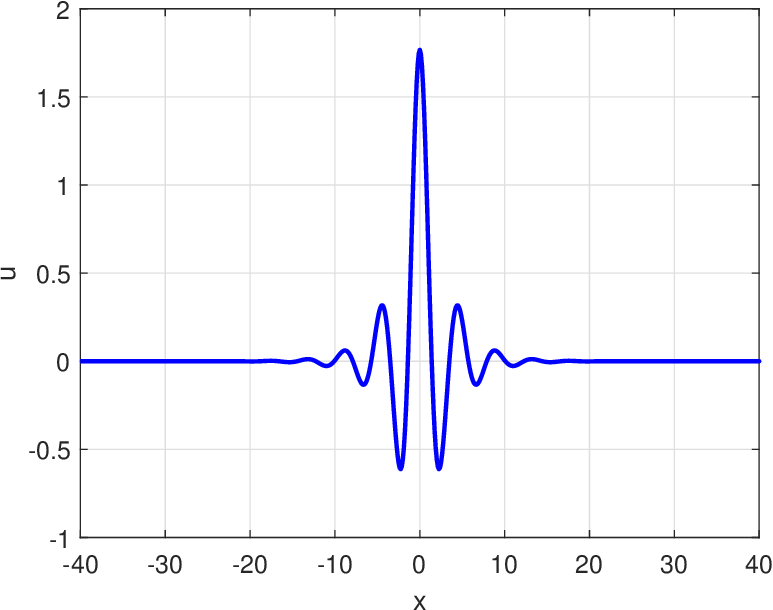}}
\subfigure[]
{\includegraphics[width=6.27cm,height=5cm]{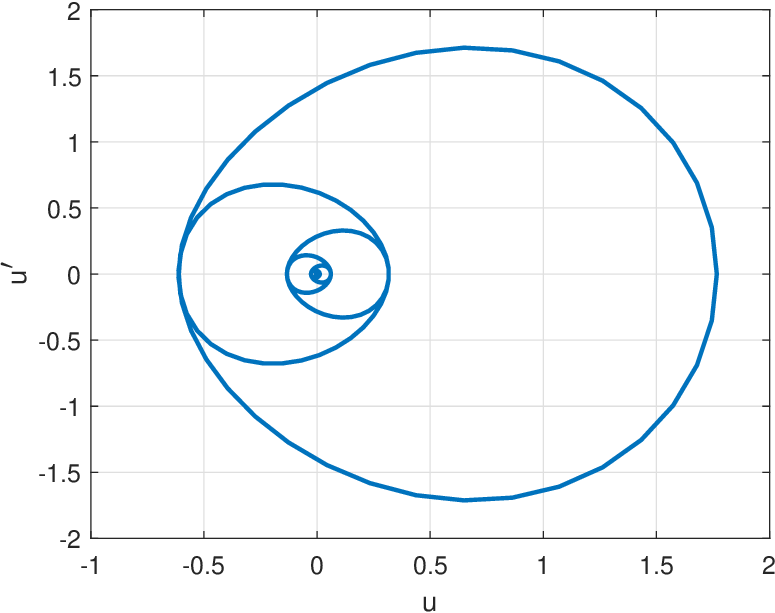}}
\caption{Generation of nonmonotone CSW's, $u$ profiles and phase portraits. (a), (b) Rosenau-RLW equation with $\alpha=1, \epsilon=\beta=1, g(u)=u^{2}/2; c_{s}=1.5$; (c), (d) Rosenau-KdV equation with $\eta=-1, \epsilon=\beta=1, g(u)=u^{2}/2, c_{s}\approx 1.3071$; (e), (f) Rosenau-Kawahara equation  with $\gamma=2, \epsilon=\beta=1, \eta=-1, g(u)=u^{2}/2, c_{s}\approx 2.2590$.}
\label{GMfig7}
\end{figure}
The form of both the approximate profiles and phase portraits reveals the nonmonotone behaviour of the waves and their oscillatory (exponential) asymptotic decay to zero as $|X|\rightarrow\infty$.

\subsection{Near $\mathcal{C}_{3}$}
In this case, the numerical experiments with $g(u)=u^{2}/2$ suggest that the bifurcation from region 2 to region 1 in Figure \ref{GMfig1} generates new homoclinic orbits in different ways: multimodal homoclinic orbits, \cite{ChampT}, CSW's and CSW's with nonmonotone decay, and PTW's, cf. \cite{DDS1,Deveney1976,Belyakov}. The bifurcation from CSW's to nonomonotone CSW's case is illustrated by some examples in Figures \ref{GMfig8} and \ref{GMfig9}.
\begin{figure}[htbp]
\centering
\subfigure[]
{\includegraphics[width=6.27cm,height=5cm]{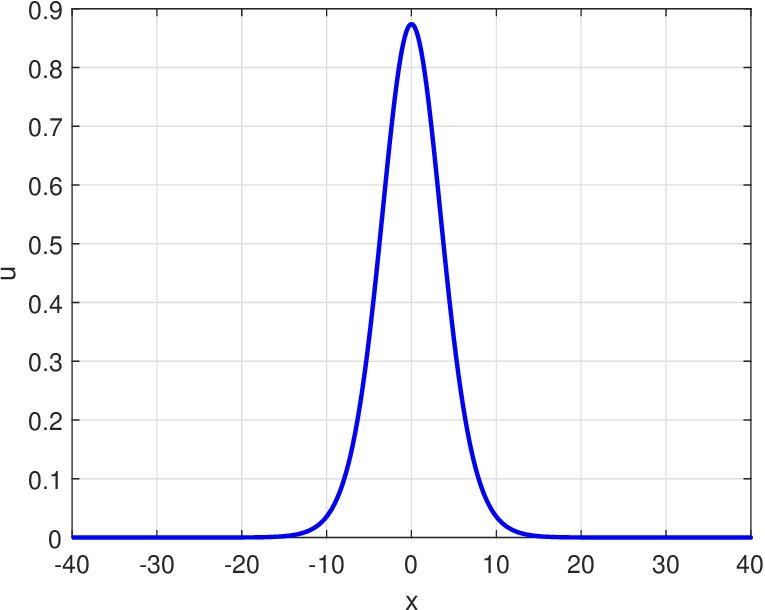}}
\subfigure[]
{\includegraphics[width=6.27cm,height=5cm]{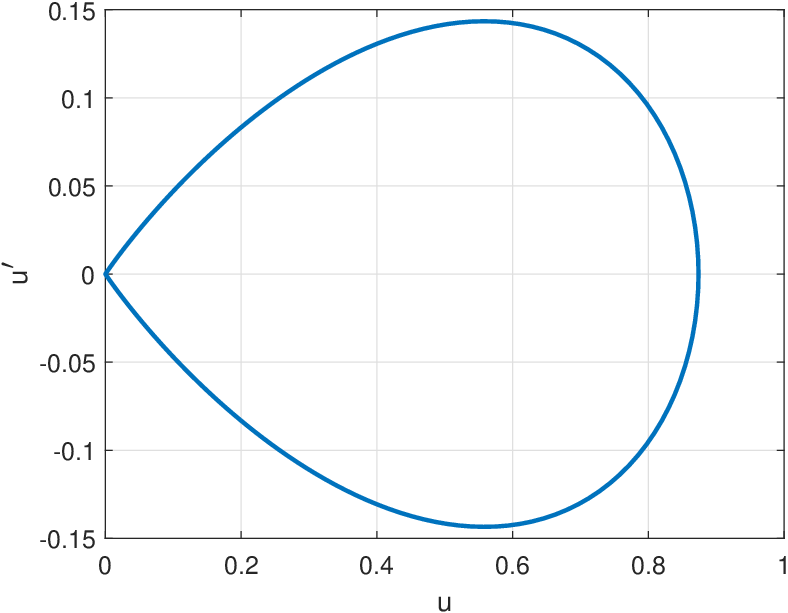}}
\subfigure[]
{\includegraphics[width=6.27cm,height=5cm]{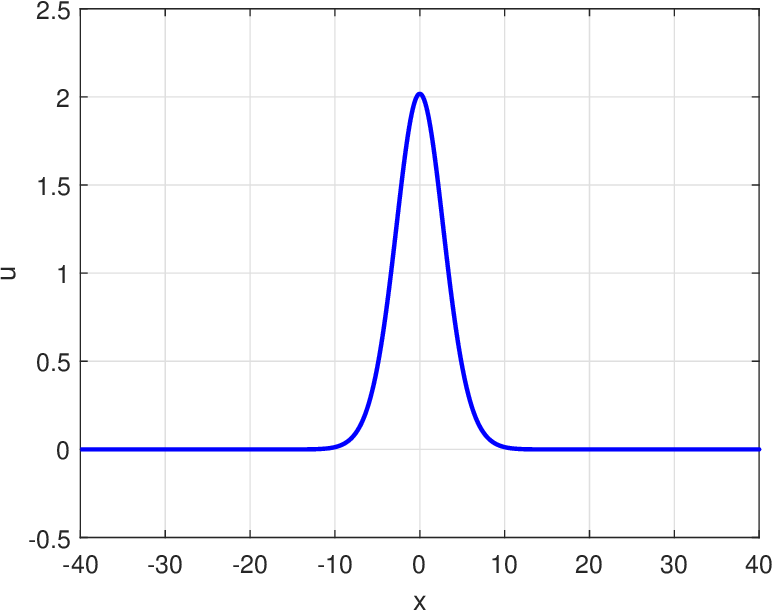}}
\subfigure[]
{\includegraphics[width=6.27cm,height=5cm]{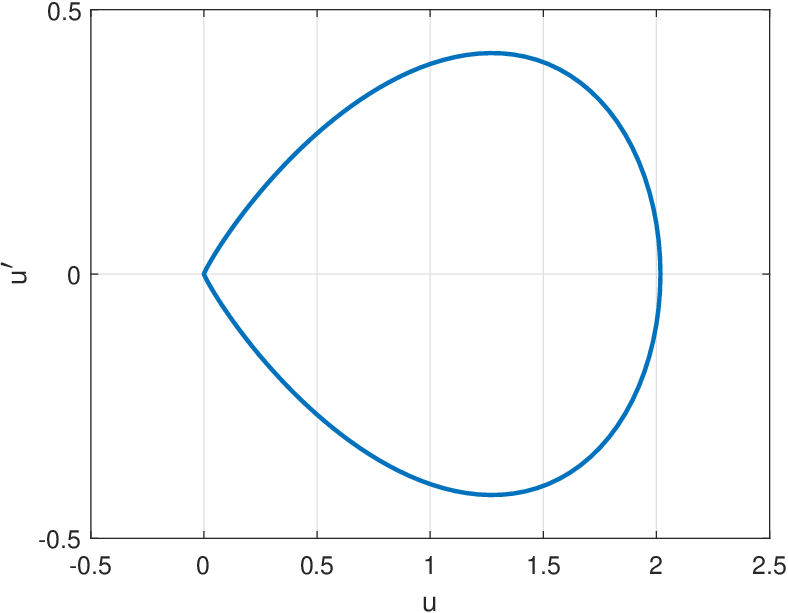}}
\subfigure[]
{\includegraphics[width=6.27cm,height=5cm]{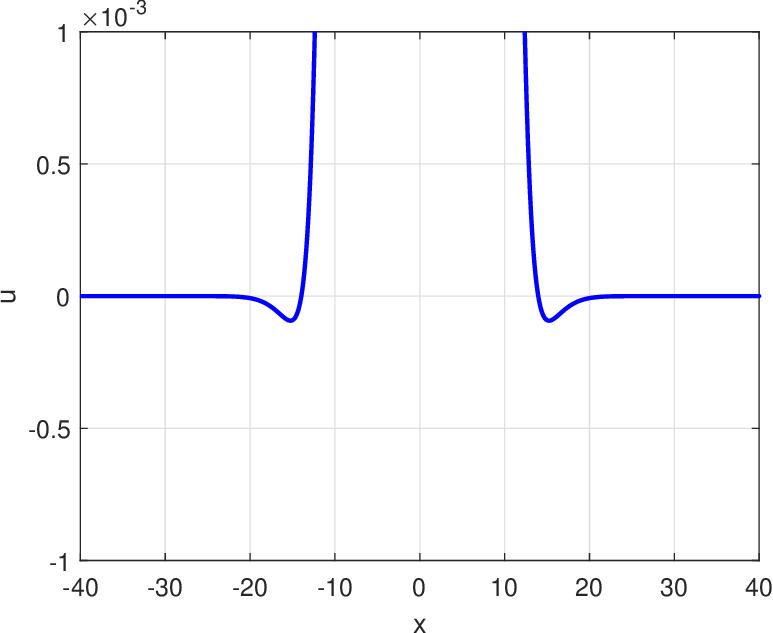}}
\subfigure[]
{\includegraphics[width=6.27cm,height=5cm]{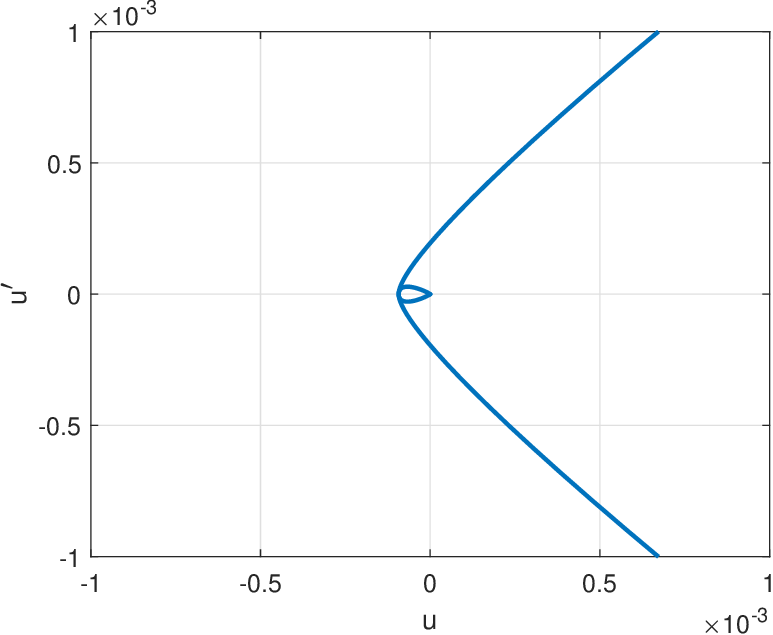}}
\caption{Bifurcation from region 2 to region 1 close to $\mathcal{C}_{3}$, $u$ profiles and phase portraits. (a), (b) Rosenau-RLW equation with $\alpha=-1, \epsilon=\beta=1, g(u)=u^{2}/2; c_{s}=1.3$; (c), (d) Rosenau-RLW equation with $\alpha=-1, \epsilon=\beta=1, g(u)=u^{2}/2; c_{s}=1.7$; (e) and (f) are magnifications of (c) and (d) resp.}
\label{GMfig8}
\end{figure}
\begin{figure}[htbp]
\centering
\subfigure[]
{\includegraphics[width=6.27cm,height=5cm]{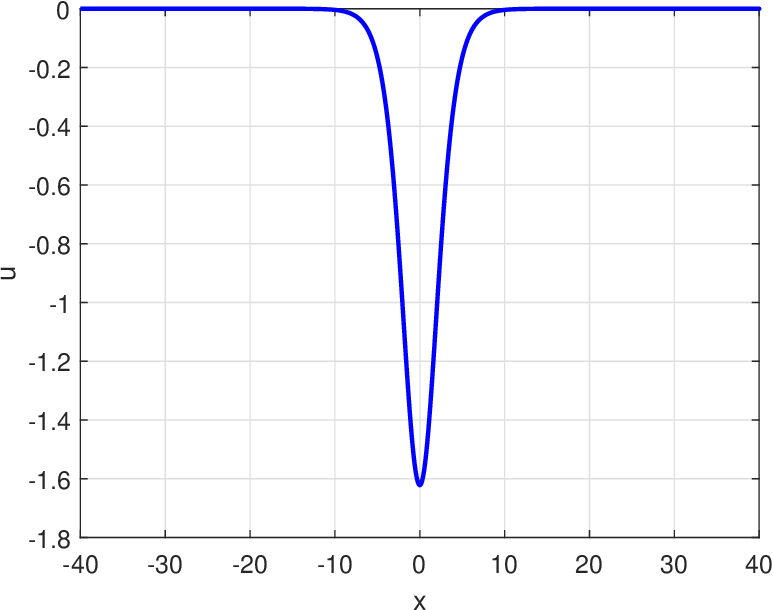}}
\subfigure[]
{\includegraphics[width=6.27cm,height=5cm]{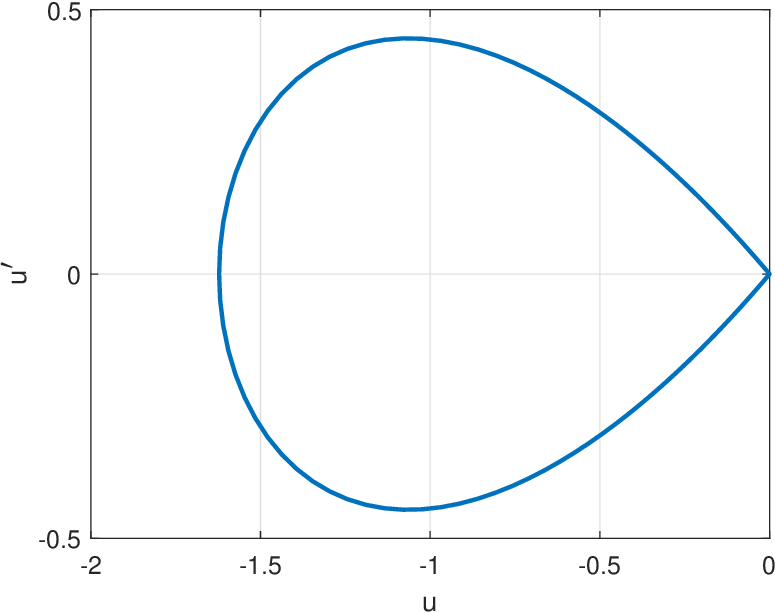}}
\subfigure[]
{\includegraphics[width=6.27cm,height=5cm]{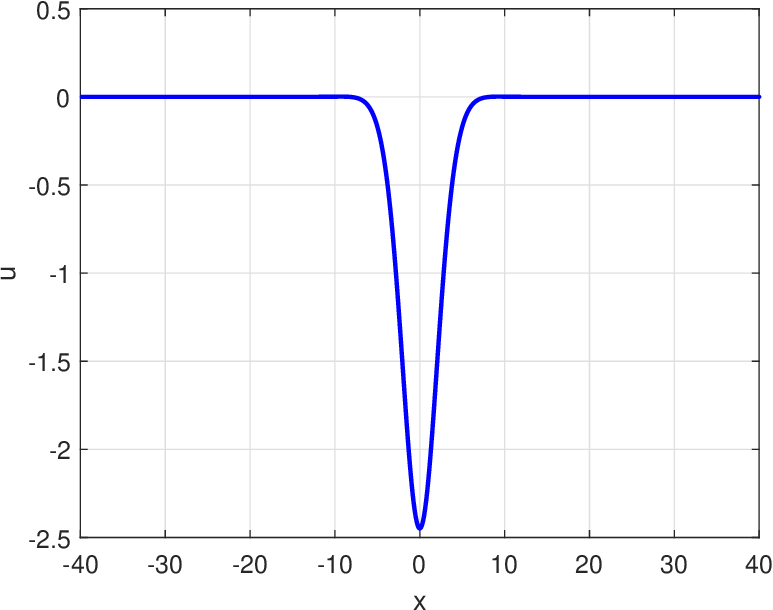}}
\subfigure[]
{\includegraphics[width=6.27cm,height=5cm]{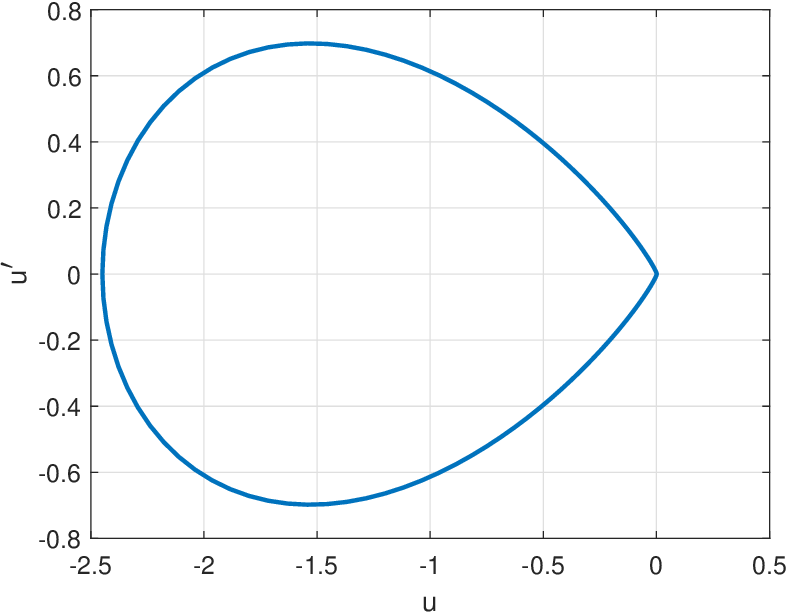}}
\subfigure[]
{\includegraphics[width=6.27cm,height=5cm]{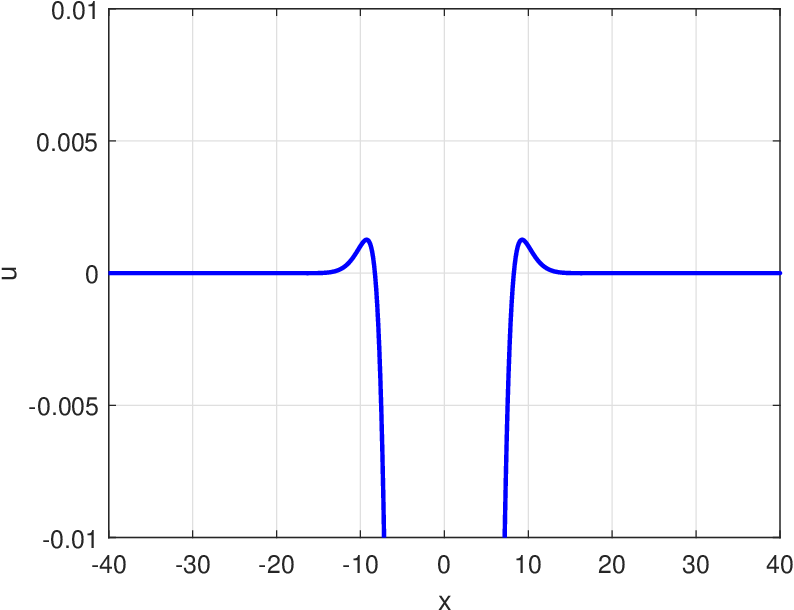}}
\subfigure[]
{\includegraphics[width=6.27cm,height=5cm]{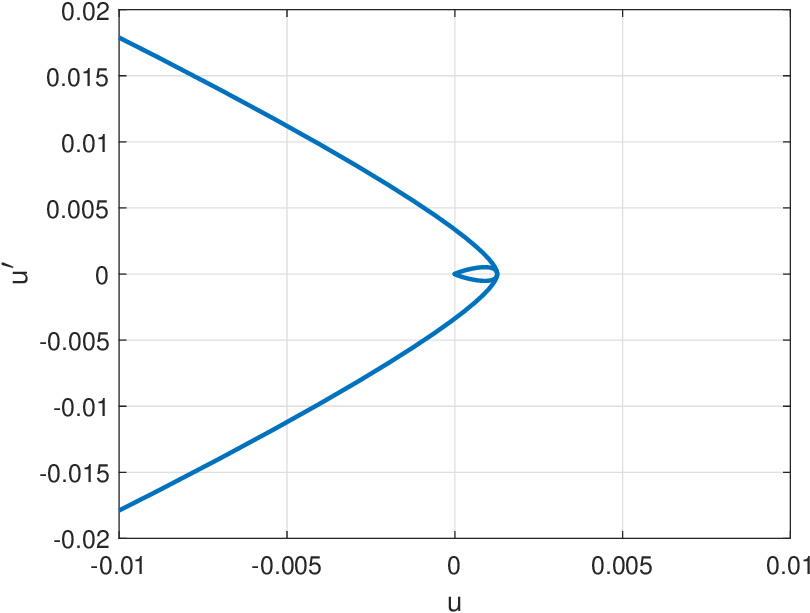}}
\caption{Bifurcation from region 2 to region 1 close to $\mathcal{C}_{3}$, $u$ profiles and phase portraits. 
(a), (b)
Rosenau-Kawahara equation with $\eta=-1, \epsilon=1, \beta=2, \gamma=1, g(u)=u^{2}/2, c_{s}\approx 0.4538$; (c), (d) Rosenau-Kawahara equation with $\gamma=1, \epsilon=1, \beta=2, \eta=-1, g(u)=u^{2}/2, c_{s}\approx 0.1438$; (e) and (f) are magnifications of (c) and (d) resp.}
\label{GMfig9}
\end{figure}

\begin{table}[htbp]
\begin{center}
\begin{tabular}{|c|c||c|c|}
    \hline
\multicolumn{2}{|c|}{Rosenau-RLW}\\
\multicolumn{2}{|c|}{($\gamma=\eta=0, \epsilon,\beta>0$)}\\
\hline
Admissible parameters&Type of solitary wave\\
\hline
 $\alpha<0, 0<c_{s}-\epsilon<\frac{\epsilon\alpha^{2}}{4\beta-\alpha^{2}}$&CSW\\
$c_{s}-\epsilon$ small&\\
    \hline
$\alpha>0, c_{s}-\epsilon<0$&GSW\\
$|c_{s}-\epsilon|$ small&\\
    \hline
$\alpha>0, c_{s}-\epsilon-\frac{\epsilon\alpha^{2}}{4\beta-\alpha^{2}}>0$ small&NMCSW\\
    \hline
    \hline
\end{tabular}
\end{center}
\caption{Admissible parameters and type of solitary wave for the Rosenau-RLW equation, according to the Normal Form Theory (cf. Appendix \ref{GMapp1}). NMCSW$=$Nonmonotone CSW.}
\label{GMtav1}
\end{table}

\begin{table}[htbp]
\begin{center}
\begin{tabular}{|c|c||c|c|}
    \hline
\multicolumn{2}{|c|}{Rosenau-KdV}\\
\multicolumn{2}{|c|}{($\alpha=\gamma=0, \epsilon,\beta>0$)}\\
\hline
Admissible parameters&Type of solitary wave\\
\hline
 $\eta>0, \epsilon<c_{s}<x_{+}$&CSW\\
$c_{s}-\epsilon$ small&\\
    \hline
$\eta<0, c_{s}-\epsilon<0$&GSW\\
$|c_{s}-\epsilon|$ small&\\
    \hline
$\eta<0, c_{s}-\epsilon-x_{+}>0$ small&NMCSW\\
    \hline
    \hline
\end{tabular}
\end{center}
\caption{Admissible parameters and type of solitary wave for the Rosenau-KdV equation, according to the Normal Form Theory, where $x_{+}=\frac{1}{2}\left(-\epsilon+\sqrt{\epsilon^{2}+\frac{\eta^{2}}{\beta}}\right)$ (cf. Appendix \ref{GMapp1}).}
\label{GMtav2}
\end{table}

\begin{table}[htbp]
\begin{center}
\begin{tabular}{|c|c||c|c|}
    \hline
\multicolumn{3}{|c|}{Rosenau-Kawahara}\\
\multicolumn{3}{|c|}{($\alpha=0, \epsilon,\beta>0$)}\\
\hline
$\rho=\epsilon\beta-\gamma$&Admissible parameters&Type of solitary wave\\
\hline
$\rho<0$& $\eta<0, y_{-}<c_{s}-\epsilon<0$, $|c_{s}-\epsilon|$ small&CSW\\
&or&\\
&$\eta>0, c_{s}-\epsilon>-\rho/\beta$&\\
    \hline
&$\eta>0, 0< c_{s}-\epsilon<-\rho/\beta$&GSW\\
&$c_{s}-\epsilon$ small&\\
    \hline
&$\eta>0, y_{-}-(c_{s}-\epsilon)>0$ small&NMCSW\\
&or&\\
&$\eta<0, c_{s}-\epsilon-y_{+}>0$ small&\\
    \hline
    \hline
$\rho>0$& $\eta>0, 0<c_{s}-\epsilon<y_{+}$, $|c_{s}-\epsilon|$ small&CSW\\
&or&\\
&$\eta<0, c_{s}-\epsilon<-\rho/\beta$&\\
    \hline
&$\eta<0, -\rho/\beta< c_{s}-\epsilon<0$&GSW\\
&$|c_{s}-\epsilon|$ small&\\
    \hline
&$\eta>0, y_{-}-(c_{s}-\epsilon)>0$ small&NMCSW\\
&or&\\
&$\eta<0, c_{s}-\epsilon-y_{+}>0$ small&\\
    \hline
    \hline
$\rho=0$& $\beta$ large and $\eta>0, 0<c_{s}-\epsilon<y_{+}$ &CSW\\
&or&\\
& $\beta$ large and $\eta<0, y_{-}<c_{s}-\epsilon<0$ &\\
\hline
&$\eta>0, y_{-}-(c_{s}-\epsilon)>0$ small &NMCSW\\
&or&\\
& $\eta<0, (c_{s}-\epsilon)-y_{+}>0$ small&\\
\hline\hline
\end{tabular}
\end{center}
\caption{Admissible parameters and type of solitary wave for the Rosenau-Kawahara equation, according to the Normal Form Theory, where $y_{\pm}=\frac{1}{2}\left(-\frac{\rho}{\beta}\pm\sqrt{\left(\frac{\rho}{\beta}\right)^{2}+\frac{\eta^{2}}{\beta}}\right)$ (cf. Appendix \ref{GMapp1}).}
\label{GMtav3}
\end{table}

\begin{table}[htbp]
\begin{center}
\begin{tabular}{|c|c||c|c|}
    \hline
\multicolumn{2}{|c|}{Rosenau-RLW-Kawahara}\\
\multicolumn{2}{|c|}{($\alpha=\gamma=-1, \epsilon=\beta=1, \eta>0$)}\\
\hline
Admissible parameters&Type of solitary wave\\
\hline
 $0<c_{s}-\epsilon<z_{+}$&CSW\\
$c_{s}-\epsilon$ small&\\
    \hline
    \hline
\end{tabular}
\end{center}
\caption{Admissible parameters and type of solitary wave for the Rosenau-RLW-Kawahara equation, according to the Normal Form Theory (cf. Appendix \ref{GMapp1})), where $z_{+}=\frac{1}{2}\left(\frac{2(\eta-3)}{3}+\sqrt{\left(\frac{2(\eta-3)}{3}\right)^{2}+4\frac{(1+\eta)^{2}}{3}}\right)$.}
\label{GMtav4}
\end{table}

\section{Existence of CSW's with CC Theory}
\label{sec4}
For the case $g(u)=G'(u), g(0)=0$, with $G(u)$ homogeneous of some degree $q>2$ and such that 
\begin{eqnarray}
\int_{\mathbb{R}}G(u)dx>0,\label{GM30}
\end{eqnarray}
for some $u$, the existence of CSW solutions of (\ref{GM1}) can be analyzed using the theory of Concentration-Compactness, developed by Lions in \cite{Lions} and which is a classical procedure to study solitary-wave solutions in nonlinear dispersive equations, \cite{Angulo,Weinstein}. In the present case we look for solitary waves of (\ref{GM1}) as solutions of minimization problems of the form
\begin{eqnarray}
I_{\lambda}=\inf\{E(u): u\in H^{2}/ F(u)=\lambda\},\label{GM31}
\end{eqnarray}
for $\lambda>0$ where, 
\begin{eqnarray}
E(u)&=&\int_{-\mathbb{R}}\left({(c_{s}-\epsilon)}u^{2}-{(c_{s}\alpha-\eta)}u_{x}^{2}+{(c_{s}\beta-\gamma)}u_{xx}^{2}\right)dx,\label{GM32a}\\
F(u)&=&\int_{\mathbb{R}}G(u)dx.\label{GM32b}
\end{eqnarray}
The main hypothesis we assume here is the existence of positive constants $C_{1}, C_{2}$ such that for all $u\in H^{2}$
\begin{eqnarray}
C_{1}||u||_{2}^{2}\leq E(u)\leq C_{2}||u||_{2}^{2}.\label{GM33}
\end{eqnarray}
Note that the continuity condition (second inequality in (\ref{GM33})) easily follows from the definition of $E$. The coercivity condition (first inequality in  (\ref{GM33})) can be characterized as follows.
\begin{proposition}
\label{proA} Let $\rho=\beta\epsilon-\gamma, \delta=\alpha\epsilon-\eta$, and
\begin{eqnarray}
x_{\pm}=\frac{1}{4\beta-\alpha^{2}}\left({\alpha\delta-2\rho}\pm\sqrt{(\alpha\delta-2\rho)^{2}+\delta^{2}(4\beta-\alpha^{2})}\right).\label{A*1}
\end{eqnarray}
If 
\begin{eqnarray}
c_{s}-\epsilon<x_{-}\;{\rm or}\; c_{s}-\epsilon>x_{+},\label{A*2}
\end{eqnarray}
then (\ref{GM33}) holds for some positive constants $C_{j}, j=1,2$ and all $u\in H^{2}$.
\end{proposition}
\begin{proof}
As mentioned before, we just have to prove the first inequality in (\ref{GM33}). We can write (\ref{GM32a}) in the form
\begin{eqnarray}
E(u)&=&\int_{-\mathbb{R}}\left({(c_{s}-\epsilon)}u^{2}+{(c_{s}\alpha-\eta)}uu_{xx}+{(c_{s}\beta-\gamma)}u_{xx}^{2}\right)dx\nonumber\\
&=&\langle A\begin{pmatrix}u\\u_{xx}\end{pmatrix}, \begin{pmatrix}u\\u_{xx}\end{pmatrix}\rangle,\label{A*4}
\end{eqnarray}
where $\langle\cdot,\cdot\rangle$ denotes the $L^{2}$ inner product and
\begin{eqnarray*}
A=\begin{pmatrix}c_{s}-\epsilon&\frac{\alpha(c_{s}-\epsilon)+\delta}{2}\\
\frac{\alpha(c_{s}-\epsilon)+\delta}{2}&\beta(c_{s}-\epsilon)+\rho\end{pmatrix}.
\end{eqnarray*}
The eigenvalues $z_{\pm}$ of $A$ satisfy $z_{+}>0$ while $z_{-}>0$ when
\begin{eqnarray}
P(c_{s}-\epsilon)>0,\label{A*3}
\end{eqnarray}
with
\begin{eqnarray*}
P(x)=x^{2}+\frac{2(2\rho-\alpha\delta)}{4\beta-\alpha^{2}}x-\frac{\delta^{2}}{4\beta-\alpha^{2}}.
\end{eqnarray*}
The roots of $P$ are given by (\ref{A*1}) and they satisfy $x_{-}\leq 0\leq x_{+}$. Thus, (\ref{A*3}) holds when one of the conditions (\ref{A*2}) is satisfied. In such case, then $A$ is shown to be positive definite and (\ref{GM33}) follows from (\ref{A*4}).
\end{proof}
The characterization (\ref{A*2})
will be discussed below for the families of Rosenau equations considered in the present paper. This will give us a range of speeds ensuring the existence of CSW's in each case and that will be related to the corresponding results obtained in section \ref{sec3} from the NFT.

The following lemma uses to be the starting point for the application of the Concentration Compactness theory.
\begin{lemma}
\label{GMlemma31}
Under the hypotheses (\ref{GM30}) and (\ref{GM33}), the problems (\ref{GM31}) satisfy the following properties:
\begin{itemize}
\item[(i)] $I_{\lambda}>0$ for $\lambda>0$.
\item[(ii)] Every minimizing sequence for $I_{\lambda}, \lambda>0$, is bounded in $H^{2}$.
\item[(iii)] For all $\theta\in (0,\lambda)$
\begin{eqnarray}
I_{\lambda}<I_{\theta}+I_{\lambda-\theta}.\label{GM34}
\end{eqnarray}
\end{itemize}
\end{lemma}
\begin{proof}
Note first that since $G$ is homogeneous of degree $q>2$ then \cite{Levandosky1999}
\begin{eqnarray*}
|G(u)|\leq C|u|^{q},\label{GM35}
\end{eqnarray*}
for some constant $C$. Then, if $u$ minimizes (\ref{GM31}) for $\lambda>0$, from (\ref{GM32b}) and the Sobolev embedding of $H^{2}$ into $L^{\infty}$ it holds that
\begin{eqnarray}
\lambda&=&\int_{\mathbb{R}}G(u)dx\leq C\int_{\mathbb{R}}|u|^{q}dx\leq C||u||_{L^{\infty}}^{q-2}||u||_{L^{2}}^{2}\label{GM36a}\\
&\leq &C||u||_{2}^{q},\label{GM36b}
\end{eqnarray}
for some constant $C$. Using (\ref{GM33}) and (\ref{GM36b}) we have
\begin{eqnarray*}
I_{\lambda}=E(u)\geq C\lambda^{2/q}>0,
\end{eqnarray*}
for some constant $C$ and (i) follows. Property (ii) is also a consequence of (\ref{GM36b}) and coercivity hypothesis (\ref{GM33}).
On the other hand, since $E$ and $F$ are homogeneous of degrees $2$ and $q$ respectively, then
\begin{eqnarray*}
I_{\tau\lambda}=\tau^{\frac{2}{q}}I_{\lambda}.
\end{eqnarray*} 
Therefore, if $\theta\in (0,\lambda)$, with $\theta=\tau\lambda, \tau\in (0,1)$, then we have, for $q>2$
\begin{eqnarray*}
I_{\theta}+I_{\lambda-\theta}&=&I_{\tau\lambda}+I_{(1-\tau)\lambda}=\tau^{\frac{2}{q}}I_{\lambda}+(1-\tau)^{\frac{2}{q}}I_{\lambda}\\
&>&(\tau+(1-\tau))I_{\lambda},
\end{eqnarray*}
and (\ref{GM34}) follows.
\end{proof}
The application of Concentration-Compactness theory in order to prove the existence of CSW's of (\ref{GM1}) under the assumption (\ref{GM33}) is as follows.
 Let $\{u_{n}\}_{n}$ be a minimizing sequence for (\ref{GM31}) and consider the sequence of nonnegative functions
\begin{eqnarray*}
\rho_{n}(x)=|u_{n}(x)|^{2}+|u'_{n}(x)|^{2}+|u''_{n}(x)|^{2}.\label{fnls_2310}
\end{eqnarray*}
Then $\rho_{n}\in L^{1}$ and its $L^{1}$ norm satisfies
$\lambda_{n}=||\rho_{n}||_{L^{1}}=||u_{n}||_{2}^{2}.$ From Lemma \ref{GMlemma31}, $\lambda_{n}$ is bounded and from (\ref{GM33}), (\ref{GM36b})
$$\lambda_{n}>C\lambda^{\frac{2}{q}},$$ for some constant $C$.
Let $\sigma=\lim_{n\rightarrow\infty}\lambda_{n}>0$. Normalizing $\rho_{n}$ as
$\widetilde{\rho}_{n}(x)=\sigma\rho_{n}(\lambda_{n}x)$, to have
$$\widetilde{\lambda}_{n}=||\widetilde{\rho}_{n}||_{L^{1}}=\sigma,$$
(and dropping tildes from now on), then from Lemma 1.1 of \cite{Lions} there is a subsequence $\{\rho_{n_{k}}\}_{k\geq 1}$ satisfying one of the following three possibilities:
\begin{itemize}
\item[(1)] (Compactness.) There are $y_{k}\in\mathbb{R}$ such that $\rho_{n_{k}}(\cdot+y_{k})$ satisfies that for any $\widetilde{\epsilon}>0$ there exists $R=R(\widetilde{\epsilon})>0$ large enough such that
\begin{eqnarray*}
\int_{|x-y_{k}|\leq R}\rho_{n_{k}}(x)dx\geq \sigma-\widetilde{\epsilon}.
\end{eqnarray*}
\item[(2)] (Vanishing.) For any $R>0$
\begin{eqnarray}
\lim_{k\rightarrow\infty}\sup_{y\in\mathbb{R}}\int_{|x-y|\leq R}\rho_{n_{k}}(x)dx=0.\label{fnls_2311}
\end{eqnarray}
\item[(3)] (Dichotomy.) There is $\theta_{0}\in (0,\sigma)$ such that for any $\widetilde{\epsilon}>0$ there exists $k_{0}\geq 1$ and $\rho_{k,{1}}, \rho_{k,{2}}\in L^{1}$ with $\rho_{k_{1}}, \rho_{k_{2}}\geq 0$ such that for $k\geq k_{0}$
\begin{eqnarray}
&&\int_{\mathbb{R}}|\rho_{n_{k}}-(\rho_{k,1}+\rho_{k,2})|dx\leq \widetilde{\epsilon},\label{fnls_2312}\\
&&\left|\int_{\mathbb{R}}\rho_{k,1}dx-\theta_{0}\right|\leq\widetilde{\epsilon},\quad
\left|\int_{\mathbb{R}}\rho_{k,2}dx-(\sigma-\theta_{0})\right|\leq\widetilde{\epsilon},\nonumber
\end{eqnarray}
with
\begin{eqnarray*}
&&{\rm supp}\rho_{k,1}\cap{\rm supp}\rho_{k,2}=\emptyset,\\
&&{\rm dist}\left({\rm supp}\rho_{k,1},{\rm supp}\rho_{k,2}\right)\rightarrow+\infty,\; k\rightarrow\infty.
\end{eqnarray*}
Since the supports of $\rho_{k,1}$ and $\rho_{k,2}$ are disjoint, we may assume the existence of $R_{0}>0$ and sequences $\{y_{k}\}_{k}$ and $R_{k}\rightarrow\infty$ as $k\rightarrow\infty$ such that, \cite{AnguloS2020}
\begin{eqnarray}
&&{\rm supp}\rho_{k,1}\subset (y_{k}-R_{0},y_{k}+R_{0}),\nonumber\\
&&{\rm supp}\rho_{k,2}\subset (-\infty,y_{k}-2R_{k})\cup (y_{k}+2R_{k},\infty).\label{fnls_2312a}
\end{eqnarray}
\end{itemize}

Our goal is to prove that compactness holds by ruling out the other two possibilities. Note first that if vanishing property (\ref{fnls_2311}) hods, then
\begin{eqnarray*}
\lim_{k\rightarrow\infty}\sup_{y\in\mathbb{R}}\int_{y-R}^{y+R}\left(|u_{n_{k}}(x)|^{2}\right)dx=0.
\end{eqnarray*}
Since $\{u_{n}\}_{n}$ is a bounded sequence in $H^{2}$ and from (\ref{GM36a}) we have
\begin{eqnarray*}
F(u_{n_{k}})\leq C||u_{n_{k}}||_{L^{\infty}}^{q-2}||u_{n_{k}}||_{L^{2}}^{2}\leq C||u_{n_{k}}||_{L^{2}}^{2},
\end{eqnarray*}
for some constant $C$. This implies that
$$\lim_{k\rightarrow\infty}F(u_{n_{k}})=0,$$ which contradicts the fact that $\lambda>0$ and vanishing is not possible.

We now assume that dichotomy holds and consider cutoff functions $\varphi,\phi\in C^{\infty}(\mathbb{R}), 0\leq\varphi,\phi\leq 1$, with
\begin{eqnarray*}
&&\phi(x)=1,\quad |x|\leq 1,\quad \phi(x)=0,\quad |x|\geq 2,\\
&&\varphi(x)=1,\quad |x|\geq 2,\quad \varphi(x)=0,\quad |x|\leq 1.
\end{eqnarray*}
Let $R>R_{0}$ and define, for $ x\in\mathbb{R}$,
\begin{eqnarray*}
&&\phi_{k}(x)=\phi\left(\frac{x-y_{k}}{R}\right),\quad 
u_{k,1}=\phi_{k}(x)u_{n_{k}}(x),\\ 
&&\varphi_{k}(x)=\varphi\left(\frac{x-y_{k}}{R_{k}}\right),\quad u_{k,2}=\varphi_{k}(x)u_{n_{k}}(x),\\
&&w_{k}(x)=u_{n_{k}}(x)-u_{k,1}(x)-u_{k,2}(x).
\end{eqnarray*}
In the following lemma some auxiliary results are collected, \cite{AnguloS2020}.
\begin{lemma}
\label{GMlemma32}
Let $\widetilde{\epsilon}>0$. If dichotomy holds, then:
\begin{itemize}
\item[(1)] For $R$ large enough
\begin{eqnarray}
&&\int_{|x-y_{k}|\leq R_{0}}|\rho_{n_{k}}(x)-\rho_{k,1}(x)|dx\leq \widetilde{\epsilon}\label{GM37a}\\
&&\int_{|x-y_{k}|\geq 2R_{k}}|\rho_{n_{k}}(x)-\rho_{k,2}(x)|dx\leq \widetilde{\epsilon},\label{GM37b}\\
&&\int_{R_{0}\leq |x-y_{k}|\leq 2R_{k}}\rho_{n_{k}}(x)dx\leq \widetilde{\epsilon}.\label{GM37c}
\end{eqnarray}
\item[(2)] For $R,R_{k}$ large enough
\begin{eqnarray}
\left|||u_{k,1}||_{2}^{2}-\int_{\mathbb{R}}\rho_{k,1}dx\right|&=&O(\widetilde{\epsilon}),\label{GM37d}\\
\left|||u_{k,2}||_{2}^{2}-\int_{\mathbb{R}}\rho_{k,2}dx\right|&=&O(\widetilde{\epsilon}).\label{GM37e}
\end{eqnarray}
\item[(3)] For $R,R_{k}$ large enough
\begin{eqnarray}
||w_{k}||_{2}=O(\widetilde{\epsilon}).\label{GM37f}
\end{eqnarray}
\end{itemize}
\end{lemma}
\begin{proof}
Using (\ref{fnls_2312}) and (\ref{fnls_2312a}) we have 
\begin{eqnarray}
O(\widetilde{\epsilon})&=&\int_{\mathbb{R}}|\rho_{n_{k}}-(\rho_{k,1}+\rho_{k,2})|dx\nonumber\\
&=&\int_{|x-y_{k}|\leq R_{0}}|\rho_{n_{k}}(x)-\rho_{k,1}(x)|dx+\int_{|x-y_{k}|\geq 2R_{k}}|\rho_{n_{k}}(x)-\rho_{k,2}(x)|dx\nonumber\\
&&+\int_{R_{0}\leq |x-y_{k}|\leq 2R_{k}}\rho_{n_{k}}(x)dx,\nonumber
\end{eqnarray}
which implies (\ref{GM37a})-(\ref{GM37c}). Note now that
\begin{eqnarray*}
&&\phi'_{k}(x)=\frac{1}{R}\phi'\left(\frac{x-y_{k}}{R}\right),\quad \phi''_{k}(x)=\frac{1}{R^{2}}\phi''\left(\frac{x-y_{k}}{R}\right),\\
&&\varphi'_{k}(x)=\frac{1}{R_{k}}\varphi'\left(\frac{x-y_{k}}{R_{k}}\right),\quad \varphi''_{k}(x)=\frac{1}{R_{k}^{2}}\varphi''\left(\frac{x-y_{k}}{R_{k}}\right).
\end{eqnarray*}
Then, for some constant $C$ and $R$ large enough
\begin{eqnarray}
||u_{k,1}||_{2}^{2}&=&\int_{\mathbb{R}}\left(|u_{k,1}(x)|^{2}+|u'_{k,1}(x)|^{2}+|u''_{k,1}(x)|^{2}\right)dx\nonumber\\
&\leq &\int_{|x-y_{k}|\leq 2R}\left(|u_{n_{k}}(x)|^{2}+|u'_{n_{k}}(x)|^{2}+|u''_{n_{k}}(x)|^{2}\right)dx\nonumber\\
&&+\frac{C}{R}\int_{|x-y_{k}|\leq 2R}|u_{n_{k}}(x)|^{2}dx\nonumber\\
&=&\int_{|x-y_{k}|\leq R_{0}}\rho_{n_{k}}(x)dx-\int_{\mathbb{R}}\rho_{k,1}(x)dx+\int_{\mathbb{R}}\rho_{k,1}(x)dx\nonumber\\
&&+\frac{C}{R}\int_{|x-y_{k}|\leq 2R}|u_{n_{k}}(x)|^{2}dx+\int_{R_{0}\leq |x-y_{k}|\leq 2R}\rho_{n_{k}}(x)dx.\nonumber
\end{eqnarray}
Therefore, using (\ref{fnls_2312a}), we have
\begin{eqnarray*}
\left| ||u_{k,1}||_{2}^{2}-\int_{\mathbb{R}}\rho_{k,1}(x)dx\right|&\leq&\underbrace{\int_{|x-y_{k}|\leq R_{0}}|\rho_{n_{k}}(x)-\rho_{k,1}(x)|dx}_{I_{1}}\\
&&+\underbrace{\frac{C}{R}\int_{|x-y_{k}|\leq 2R}|u_{n_{k}}(x)|^{2}dx}_{I_{2}}\\
&&+\underbrace{\int_{R_{0}\leq |x-y_{k}|\leq 2R}\rho_{n_{k}}(x)dx}_{I_{3}}.
\end{eqnarray*}
For $R$ large enough, (\ref{GM37a}), and (\ref{GM37c}), the integrals $I_{j}, j=1,2,3$, are $O(\widetilde{\epsilon})$, leading to (\ref{GM37d}). Similarly,
\begin{eqnarray*}
\left| ||u_{k,2}||_{2}^{2}-\int_{\mathbb{R}}\rho_{k,2}(x)dx\right|&\leq&\underbrace{\int_{|x-y_{k}|\geq R_{k}}|\rho_{n_{k}}(x)-\rho_{k,2}(x)|dx}_{J_{1}}\\
&&+\underbrace{\frac{C}{R_{k}}\int_{|x-y_{k}|\leq 2R_{k}}|u_{n_{k}}(x)|^{2}dx}_{J_{2}}\\
&&+\underbrace{\int_{R_{k}\leq |x-y_{k}|\leq 2R_{k}}\rho_{n_{k}}(x)dx}_{J_{3}},
\end{eqnarray*}
and for $R$ large enough, (\ref{GM37b}), and (\ref{GM37c}), the integrals $J_{j}, j=1,2,3$, are $O(\widetilde{\epsilon})$ and (\ref{GM37e}) holds. Note finally that 
\begin{eqnarray*}
&&w_{k}=\chi_{k}u_{n_{k}}, \quad \chi_{k}=1-\phi_{k}-\varphi_{k},\\
&&{\rm supp}\chi_{k}\subset \{R\leq |x-y_{k}|\leq 2R_{k}\}.
\end{eqnarray*}
Then
\begin{eqnarray}
||w_{k}||_{2}^{2}\leq C\left(||\chi_{k}||_{L^{\infty}}^{2}+||\chi'_{k}||_{L^{\infty}}^{2}+||\chi''_{k}||_{L^{\infty}}^{2}\right)\int_{R\leq |x-y_{k}|\leq 2R_{k}}\rho_{n_{k}}(x)dx,\label{GM37h}
\end{eqnarray}
for some constant $C$. The parenthesis on the right-hand side of (\ref{GM37h}) is bounded while the integral is $O(\widetilde{\epsilon})$ because of (\ref{GM37c}), implying (\ref{GM37f}).
\end{proof}
\begin{lemma}
\label{GMlemma33}
Let $\widetilde{\epsilon}>0$.
If dichotomy holds and $R, R_{k}$ are large enough, then there are subsequences of $\{u_{k,1}\}_{k}, \{u_{k,2}\}_{k}$, denoted again in the same way, such that
\begin{eqnarray}
E(u_{n_{k}})&=&E(u_{k,1})+E(u_{k,2})+O(\widetilde{\epsilon}),\label{GM38a}\\
F(u_{n_{k}})&=&F(u_{k,1})+F(u_{k,2})+O(\widetilde{\epsilon}).\label{GM38b}
\end{eqnarray}
\end{lemma}
\begin{proof}
Since $u_{n_{k}}=w_{k}+u_{k,1}+u_{k,2}$ then we can write
$$E(u_{n_{k}})=E(w_{k})+E(u_{k,1})+E(u_{k,2})+N,$$ where
\begin{eqnarray*}
N&=&2A_{1}\int_{\mathbb{R}}w_{k}(u_{k,1}+u_{k,2})dx+2A_{2}\int_{\mathbb{R}}w'_{k}(u'_{k,1}+u'_{k,2})dx\\
&&+2A_{2}\int_{\mathbb{R}}w''_{k}(u''_{k,1}+u''_{k,2})dx,
\end{eqnarray*}
where
$$A_{1}={c_{s}-\epsilon},\quad A_{2}=-({c_{s}\alpha-\eta}),\quad A_{3}={c_{s}\beta-\gamma}.$$ Therefore, from Cauchy-Schwarz inequality we have
\begin{eqnarray*}
|N|\leq C||w_{k}||_{2}\left(||u_{k,1}||_{2}+||u_{k,2}||_{2}\right),
\end{eqnarray*}
for some constant $C$. Now, (\ref{GM37d}), (\ref{GM37e}), (\ref{fnls_2312}), and (\ref{GM37f}) imply that $N=O(\widetilde{\epsilon})$ and (\ref{GM38a}) holds. On the other hand, since $\{u_{n_{k}}\}_{k}$ is a minimizing sequence, using Lemma \ref{GMlemma31}(i) and (iii) along with the property $0\leq\phi\leq 1$, then $F(u_{k,1})$ is bounded, and there is a subsequence (denoted again by $u_{k,1}$) such that $F(u_{k,1})$ converges to some $\theta\in\mathbb{R}$. Thus, there is some $k_{0}>0$ such that for $k\geq k_{0}$
\begin{eqnarray*}
|F(u_{k,1})-\theta|\leq \widetilde{\epsilon}.
\end{eqnarray*}
Now, since $F(u_{n_{k}})=\lambda$, $\phi_{k}\varphi_{k}=0$, and $G$ is homogeneous of degree $q$, from (\ref{GM36b}) we have
\begin{eqnarray}
|F(u_{k,2})-(\lambda-\theta)|&\leq &\left|\int_{\mathbb{R}}\left(G(\varphi_{k}u_{n_{k}})-G(u_{n_{k}})+G(\phi_{k}u_{n_{k}})\right)dx\right|+\widetilde{\epsilon}\nonumber\\
&=&\left| \int_{\mathbb{R}}\left(\varphi_{k}^{q}+\phi_{k}^{q}-1\right)G(u_{n_{k}})dx\right|+\widetilde{\epsilon}\nonumber\\
&\leq &\int_{\mathbb{R}}\chi_{k}^{q}G(u_{n_{k}})dx+\widetilde{\epsilon}=F(w_{k})+\widetilde{\epsilon}\nonumber\\
&\leq &C||w_{k}||_{2}^{q}+\widetilde{\epsilon}=O(\widetilde{\epsilon}),\label{GM38c}
\end{eqnarray}
where in the last equality  (\ref{GM37f}) was used; then (\ref{GM38b}) follows.
\end{proof}
As a consequence of (\ref{GM38a}) we have
\begin{corolary}
Under the conditions of Lemma \ref{GMlemma33}
\begin{eqnarray}
I_{\lambda}&\geq &\lim_{k\rightarrow\infty}\inf E(u_{n_{k}})\nonumber\\
&\geq &\lim_{k\rightarrow\infty}\inf E(u_{k,1}) +\lim_{k\rightarrow\infty}\inf E(u_{k,2})+\widetilde{\epsilon}.\label{GM39}
\end{eqnarray}
\end{corolary}
In order to rule out dichotomy, we discuss the existence of the limit obtained in the proof of Lemma \ref{GMlemma33}
$$\theta=\lim_{k\rightarrow\infty}F(u_{k,1}).$$
We have the following possibilities, \cite{AnguloS2020}:
\begin{itemize}
\item[(1)] $\theta=0$. Then, using (\ref{GM38c}) we have, for $k$ large enough and $\widetilde{\epsilon}<\lambda/2$
\begin{eqnarray}
F(u_{k,2})>\lambda-\widetilde{\epsilon}>\frac{\lambda}{2}>0.\label{fnls_2320b}
\end{eqnarray}
We consider $d_{k}>0$ such that $\lambda=F(d_{k}h_{k,2})$; explicitly
$$d_{k}=\left(\frac{\lambda}{F(u_{k,2})}\right)^{\frac{1}{q}},$$ Then
\begin{eqnarray*}
|d_{k}-1|<\left|\left(\frac{\lambda}{F(u_{k,2})}\right)^{\frac{1}{q}}-1\right|<\frac{1}{F(u_{k,2})^{1/q}}\left|\lambda^{\frac{1}{q}}-F(h_{k,2})^{\frac{1}{q}}\right|.\label{fnls_2321}
\end{eqnarray*}
For $x>0$ let $r(x)=x^{\frac{1}{q}}$. Then
\begin{eqnarray*}
\left|\lambda^{\frac{1}{q}}-F(u_{k,2})^{\frac{1}{q}}\right|&=&\left|r(\lambda)-r(F(h_{k,2}))\right|\\
&=&\left|r'(\xi_{k})(\lambda-F(u_{k,2}))\right|,
\end{eqnarray*}
for some $\xi_{k}$ between $F(u_{k,2})$ and $\lambda$ and consequently, by (\ref{fnls_2320b}), between $\lambda/2$ and $\lambda$. Therefore
\begin{eqnarray*}
r'(\xi_{k})=\frac{1}{q}\left(\frac{1}{\xi_{k}}\right)^{\frac{1}{q}-1}<\frac{1}{q}\left(\frac{2}{\lambda}\right)^{\frac{1}{q}-1}.\label{fnls_2322}
\end{eqnarray*}
This yields, for some constant $C(\lambda)$
\begin{eqnarray*}
|d_{k}-1|<\left(\frac{2}{\lambda}\right)^{1/q}\left|\lambda-f(u_{k,2})\right|=O(\widetilde{\epsilon}).
\end{eqnarray*}
Then $\lim_{k\rightarrow\infty}d_{k}=1$ and
\begin{eqnarray}
I_{\lambda}\leq E(d_{k}u_{k,2})=d_{k}^{2}E(u_{k,2})=E(u_{k,2})+O(\widetilde{\epsilon}).\label{fnls_2323}
\end{eqnarray}
Therefore, from coercivity property of $E$, dichotomy, and (\ref{GM37d})
\begin{eqnarray*}
\lim_{k\rightarrow\infty}\inf E(u_{k,1})&\geq &C\lim_{k\rightarrow\infty}\inf ||u_{k,1}||_{2}^{2}\\
&\geq & C\lim_{k\rightarrow\infty}\inf ||\rho_{k,1}||_{L^{1}}+O(\widetilde{\epsilon})\geq C\theta_{0}+O(\widetilde{\epsilon}).
\end{eqnarray*}
Therefore, (\ref{GM39}) and (\ref{fnls_2323}) imply, for $\widetilde{\epsilon}>0$ arbitrarily small
\begin{eqnarray*}
I_{\lambda}\geq C\theta_{0}+I_{\lambda}+O(\widetilde{\epsilon}),
\end{eqnarray*}
that taking $\widetilde{\epsilon}\rightarrow 0$
 leads to $I_{\lambda}\geq C\theta_{0}+I_{\lambda}$, which is a the contradiction with $\theta_{0}\leq 0$.
\item[(2)] $\lambda>\theta>0$. Then, from (\ref{GM38a})
\begin{eqnarray*}
E(u_{n_{k}})&=&E(u_{k,1})+E(u_{k,2})+O(\widetilde{\epsilon})\\
&\geq & I_{F(u_{k,1})}+I_{F(u_{k,2})}+O(\widetilde{\epsilon})\\
&=&\left(F(u_{k,1})^{\frac{2}{q}}+F(u_{k,2})^{\frac{2}{q}}\right)I_{1}+O(\widetilde{\epsilon}).
\end{eqnarray*}
If $k\rightarrow\infty$ then
\begin{eqnarray*}
I_{\lambda}&\geq &\left(\theta^{\frac{2}{q}}+(\lambda-\theta)^{\frac{2}{q}}\right)I_{1}+O(\widetilde{\epsilon})=I_{\theta}+I_{\lambda-\theta}+O(\widetilde{\epsilon}),
\end{eqnarray*}
for $\widetilde{\epsilon}>0$ arbitrarily small. This contradicts Lemma \ref{GMlemma31}(ii).
\item[(3)] $\theta<0$. Then, from  (\ref{GM38c})
\begin{eqnarray*}
\lim_{k\rightarrow\infty}F(h_{k,2})=\lambda-\theta>\frac{\lambda}{2},
\end{eqnarray*}
and we apply the analysis of item (1) leading to contradiction.
\item[(4)] $\theta=\lambda$. Then, 
as in item (1), we may take $\widetilde{\epsilon}$ small enough so that $F(h_{k,1})>\lambda/2$ for $k$ large. The same arguments apply to have
\begin{eqnarray*}
I_{\lambda}\leq E(u_{k,1})+O(\widetilde{\epsilon}),
\end{eqnarray*}
and
\begin{eqnarray*}
\lim_{k\rightarrow\infty}\inf E(u_{k,2})&\geq &C\lim_{k\rightarrow\infty}\inf ||u_{k,2}||_{2}^{2}\\
&\geq & C\lim_{k\rightarrow\infty}\inf ||\rho_{k,2}||_{L^{1}}+O(\widetilde{\epsilon})\geq C(\sigma-\theta_{0})+O(\widetilde{\epsilon}).
\end{eqnarray*}
This leads, for $k\rightarrow\infty, \widetilde{\epsilon}\rightarrow 0$ 
\begin{eqnarray*}
I_{\lambda}\geq C(\sigma-\theta_{0})+I_{\lambda},
\end{eqnarray*}
implying the contradiction $\sigma-\theta_{0}\leq 0$. 
\item[(5)] $\theta>\lambda$. Then $F(h_{k,1})>0$ for $k$ large enough. We consider $e_{k}$ such that $F(e_{k}u_{k,1})=\lambda$. Explicitly
\begin{eqnarray*}
e_{k}=\left(\frac{\lambda}{F(u_{k,1})}\right)^{\frac{2}{q}}.
\end{eqnarray*}
Then
$$\lim_{k\rightarrow\infty}e_{k}=\left(\frac{\lambda}{\theta}\right)^{\frac{2}{q}}.$$ Therefore, for $k$ large enough, we have
\begin{eqnarray}
I_{\lambda}\leq E(e_{k}u_{k,1})=e_{k}^{2}E(u_{k,1})<E(u_{k,1}).\label{fnls_2324a}
\end{eqnarray}
From coercivity property of $E$, dichotomy, and (\ref{GM37e})
\begin{eqnarray}
\lim_{k\rightarrow\infty}\inf E(u_{k,2})&\geq &C\lim_{k\rightarrow\infty}\inf ||u_{k,2}||_{2}^{2}
\geq  C\lim_{k\rightarrow\infty}\inf ||\rho_{k,2}||_{L^{1}}+O(\widetilde{\epsilon})\nonumber\\
&\geq &C(\sigma-\theta_{0})+O(\widetilde{\epsilon}).\label{fnls_2324b}
\end{eqnarray}
Thus,  (\ref{GM39}), (\ref{fnls_2324a}), and (\ref{fnls_2324b}) imply
\begin{eqnarray*}
I_{\lambda}\geq I_{\lambda}+C(\sigma-\theta_{0})+O(\widetilde{\epsilon}).
\end{eqnarray*}
When $\widetilde{\epsilon}\rightarrow 0$, we have the contradiction $\sigma-\theta_{0}\leq 0$.
\end{itemize}
Since vanishing and dichotomy do not hold, then Lemma 1.1 of \cite{Lions} implies that necessarily compactness is satisfied. This means that there exists $\{y_{k}\}_{k}\subset\mathbb{R}$ such that for any $\widetilde{\epsilon}>0$ there is $R=R(\widetilde{\epsilon})>0$ large and $k_{0}>$ such that for $k\geq k_{0}$
\begin{eqnarray}
\int_{|x-y_{k}|\leq R}\rho_{n_{k}}dx\geq \sigma-\widetilde{\epsilon},\quad 
\int_{|x-y_{k}|\geq R}\rho_{n_{k}}dx=O(\widetilde{\epsilon}).\label{GM310}
\end{eqnarray}
From the arguments of Lemma \ref{GMlemma31} (i), applied to $P_{k}=\{|x-y_{k}|\geq R\}$,  we have
\begin{eqnarray*}
\int_{P_{k}}G(u_{n_{k}})dx&\leq & C||u_{n_{k}}||_{L^{\infty}}^{q-2}\int_{P_{k}}u_{n_{k}}^{2}dx\\
&\leq & C||u_{n_{k}}||_{2}^{q-2}\int_{P_{k}}\rho_{n_{k}}dx=O(\widetilde{\epsilon}),
\end{eqnarray*}
where in the last inequality Lemma \ref{GMlemma31}(iii) and (\ref{GM310}) were used. Therefore
\begin{eqnarray}
\left|\int_{|x-y_{k}|\leq R}G(u_{n_{k}})dx-\lambda\right|\leq \widetilde{\epsilon}.\label{GM311}
\end{eqnarray}
Let $\widetilde{u}_{n_{k}}(x)=u_{n_{k}}(x-y_{k})$. Then $\widetilde{u}_{n_{k}}$ is bounded in $H^{2}$ and therefore there exists a subsequence $\widetilde{u}_{n_{k}}$ which converges weakly in $H^{2}$ to some $\widetilde{u}$. From (\ref{GM311})
\begin{eqnarray*}
\lambda\geq \int_{-R}^{R}G(\widetilde{u}_{n_{k}})dx\geq \lambda-\widetilde{\epsilon}.
\end{eqnarray*}
The compact embedding of $H^{2}(-R,R)$ in $W^{1,p}(-R,R)\;  \forall p\geq 1$ and Lemma 2.2 of \cite{Levandosky1999} imply
\begin{eqnarray*}
\lambda\geq \int_{-R}^{R}G(\widetilde{u})dx\geq \lambda-\widetilde{\epsilon}.
\end{eqnarray*}
Taking $\widetilde{\epsilon}\rightarrow 0$ then $R=R(\widetilde{\epsilon})\rightarrow\infty$ and therefore
$F(\widetilde{u})=\lambda$. Furthermore, from lower semicontinuity and invariance by translations of $E$ we have
\begin{eqnarray*}
I_{\lambda}=\lim_{k\rightarrow\infty}\inf E(\widetilde{u}_{n_{k}})\geq E(\widetilde{u})\geq I_{\lambda}.
\end{eqnarray*}
Therefore $\widetilde{u}$ satisfies the variational problem
$$\delta E(\widetilde{u})=\kappa \delta F(\widetilde{u}),$$ with $\kappa$ a Lagrange multiplier. This means
\begin{eqnarray}
(\gamma-\beta c_{s})\widetilde{u}''''+(\eta-\alpha c_{s})\widetilde{u}''+(\epsilon-c_{s})u+\kappa g(\widetilde{u})=0.\label{GM312}
\end{eqnarray}
Using the Euler theorem to the homogeneous function $G$, then
$$G'(\widetilde{u})\widetilde{u}=qG(\widetilde{u}),$$ and multiplying (\ref{GM312}) by $\widetilde{u}$ and integrating yield
\begin{eqnarray*}
q\kappa \lambda=E(\widetilde{u})=I_{\lambda}\Rightarrow \kappa=\frac{I_{\lambda}}{q\lambda}>0.
\end{eqnarray*}
Defining
$$u=\kappa^{\frac{1}{q-1}}\widetilde{u},$$ then $u$ is a solution of (\ref{GM14}). This completes the proof of the first part of the following theorem.
\begin{theorem}
\label{GMtheor34} Under the assumptions (\ref{linearw}), (\ref{GM30}), (\ref{A*2}), there is a solution $u\in H^{u}$ of (\ref{GM14}). Furthermore if $g(u)\in H^{s-2}$ when $u\in H^{s}, s>0$, 
then
$u\in H^{\infty}$.
\end{theorem}
\begin{proof}
The last statement is proved from writing (\ref{GM14}) in the form
\begin{eqnarray}
Q\varphi=g(\varphi),\label{GM315}
\end{eqnarray}
where $Q$ is the linear operator with Fourier symbol
\begin{eqnarray*}
\widehat{Qu}(k)=\left((\beta c_{s}-\gamma)k^{4}-(\alpha c_{s}-\eta)k^{2}+(c_{s}-\epsilon)\right)\widehat{u}(k),\quad k\in\mathbb{R}.
\end{eqnarray*}
Note that from Proposition \ref{proA}, the determinant of the matrix $A$ in (\ref{A*4}) is positive, leading to
\begin{eqnarray*}
(\alpha c_{s}-\eta)^{2}<4(\beta c_{s}-\gamma)(c_{s}-\epsilon).\label{GM315a}
\end{eqnarray*}
This implies that $Q$ is invertible. The hypothesis on $g$ and 
(\ref{GM315}) yield $u\in H^{4}$. A bootstrap argument applied to (\ref{GM315}) proves the result.
\end{proof}

\subsection{Some examples}
In this section Theorem \ref{GMtheor34} will be applied in order to prove the existence of CSW solutions of several families of Rosenau-type equations (\ref{GM1}). The application mainly depends on the coercivity property (\ref{GM33}) of the corresponding functional (\ref{GM32a}). We will consider nonlinear terms of the form $g(u)=\frac{u^{p+1}}{p+1}, p\geq 1$, for which
$$G(u)=\frac{u^{p+2}}{(p+2)(p+1)},$$ and consequently $G$ is homogeneous of degree $q=p+2$. Thus
$$F(u)=\int_{\mathbb{R}}\frac{u^{p+2}}{(p+2)(p+1)}dx.$$
The corresponding conditions on the speed $c_{s}$ for coercivity property will be derived in each case from the characterization given in Proposition \ref{proA}. In all the cases, the CSW's predicted from CCT are non monotone and were partially anticipated by NFT.
\subsubsection{Rosenau equation}
Rosenau equation corresponds to (\ref{GM1}) with $\alpha=\eta=\gamma=0, \epsilon, \beta>0$. Then
\begin{eqnarray*}
E(u)=\int_{\mathbb{R}}\left({(c_{s}-\epsilon)}u^{2}+{(c_{s}\beta)}u_{xx}^{2}\right)dx.
\end{eqnarray*}
In this case the roots in (\ref{A*1}) are $x_{+}=0, x_{-}=-\epsilon$. Therefore, (\ref{A*2}) is satisfied for $c_{s}>0$ only when $c_{s}-\epsilon>0$. In this case the application of NFT is special, since from (\ref{GM14c}), 
$a=a(c_{s})=\frac{c_{s}-\epsilon}{\beta c_{s}},\quad
b=0,$ and the eigenvalues of the linearization are the roots of $\lambda^{4}+a=0$. The existence of (nonmonotone) CSW's for speeds $c_{s}>\epsilon$ with $c_{s}-\epsilon$ small can be established (region 1 close to $\mathcal{C}_{3}$ in Figure \ref{GMfig1}). The existence result obtained from Concentration-Compactness theory is valid for $c_{s}>\epsilon$ with no restriction on the size of $c_{s}-\epsilon$. The numerical generation of some of the profiles is illustrated in Figure \ref{GMfig10}. For $c_{s}<\epsilon$ (region 3 with $b=0$) PTW's were found experimentally.
\begin{figure}[htbp]
\centering
\subfigure[]
{\includegraphics[width=6.27cm,height=5cm]{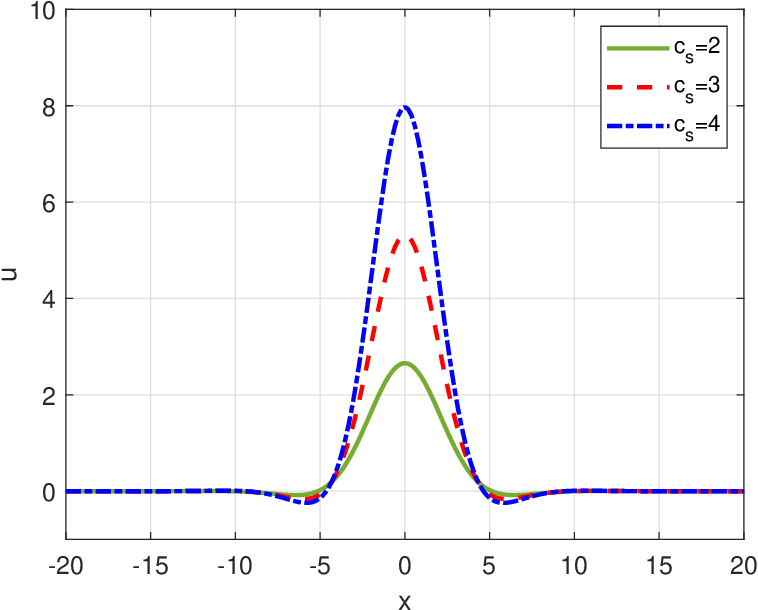}}
\subfigure[]
{\includegraphics[width=6.27cm,height=5cm]{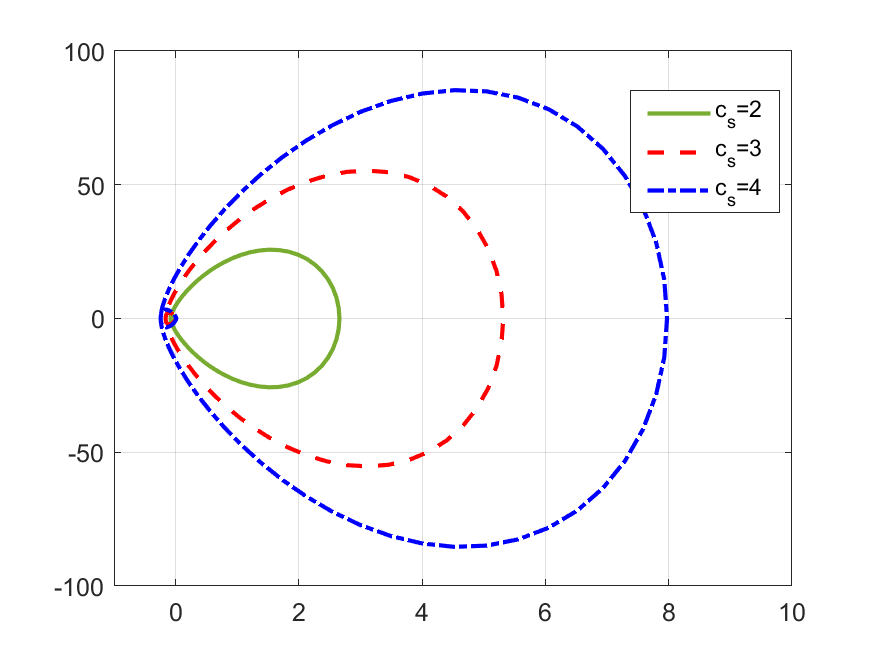}}
\caption{Generation of nonmonotone CSW's, $u$ profiles and phase portraits. Rosenau equation with $\epsilon=\beta=1, g(u)=u^{2}/2$ and for several speeds $c_{s}>\epsilon$.}
\label{GMfig10}
\end{figure}

It may be worth mentioning that when $g(u)$ is a sum of homogeneous functions, the application of Concentration-Compactness theory to (\ref{GM1}) can follow the steps developed in \cite{EsfahaniL2021}. For the Rosenau equation, the corresponding result would ensure the existence of solitary waves under the same restriction $c_{s}>\epsilon$ on the speeds. This is illustrated in Figure \ref{GMfig11} for $g(u)=u^{3}/3+u^{5}/5$.
\begin{figure}[htbp]
\centering
\subfigure[]
{\includegraphics[width=6.27cm,height=5cm]{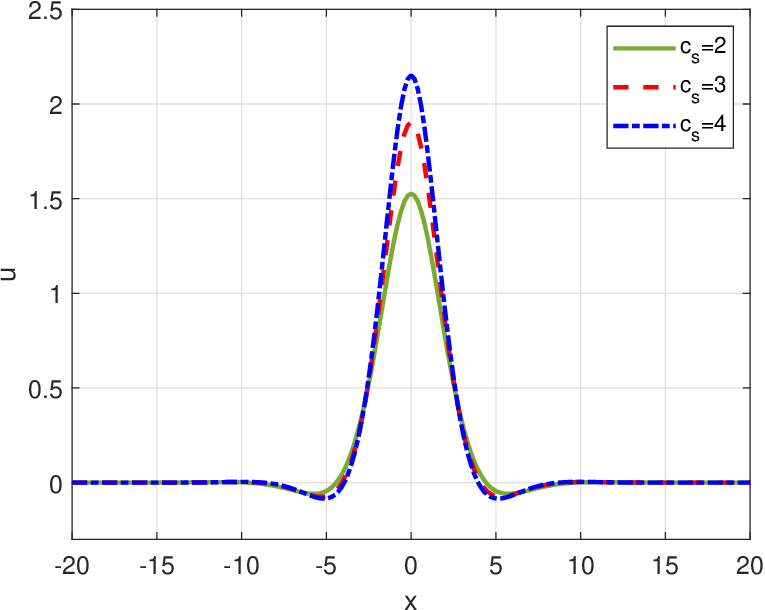}}
\subfigure[]
{\includegraphics[width=6.27cm,height=5cm]{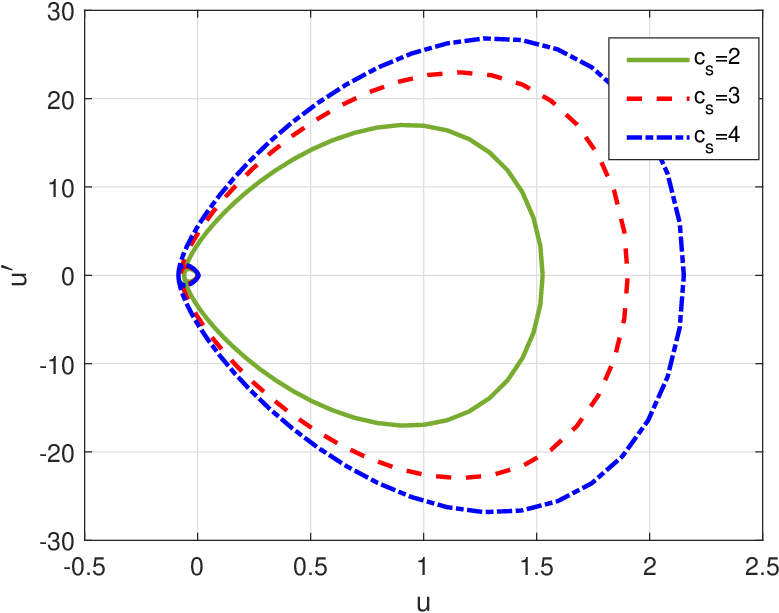}}
\caption{Generation of nonmonotone CSW's, $u$ profiles and phase portraits. Rosenau equation with $\epsilon=\beta=1, g(u)=u^{3}/3+u^{5}/5$ and for several speeds $c_{s}>\epsilon$.}
\label{GMfig11}
\end{figure}

In both Figures \ref{GMfig10} and \ref{GMfig11}, we observe that the amplitude of the waves is an increasing function of the speed and, from the phase portraits, the profiles decay to zero exponentially as $|X|\rightarrow\infty$ in an oscillatory way.
\subsubsection{Rosenau-RLW equation}
Rosenau-RLW equation corresponds to (\ref{GM1}) with $\eta=\gamma=0, \epsilon, \beta>0$. Then
\begin{eqnarray*}
E(u)=\int_{\mathbb{R}}\left((c_{s}-\epsilon)u^{2}-\alpha c_{s}u_{x}^{2}+{(c_{s}\beta)}u_{xx}^{2}\right)dx.
\end{eqnarray*}
In this case $$x_{+}=\frac{\alpha^{2}\epsilon}{4\beta-\alpha^{2}}, \quad x_{-}=-\epsilon,$$ and for speeds $c_{s}>0$ (\ref{A*2}) holds when 
$$c_{s}-\epsilon>\frac{\alpha^{2}\epsilon}{4\beta-\alpha^{2}},$$ which was partially anticipated by NFT (cf. Table \ref{GMtav1}) but now $\alpha\in\mathbb{R}$ and the proximity between the two values $c_{s}-\epsilon$ and $\frac{\alpha^{2}\epsilon}{4\beta-\alpha^{2}}$ is not required.
\subsubsection{Rosenau-KdV equation}
Rosenau-KdV equation corresponds to (\ref{GM1}) with $\alpha=\gamma=0, \epsilon, \beta>0$. Then
\begin{eqnarray*}
E(u)=\int_{\mathbb{R}}\left({(c_{s}-\epsilon)}u^{2}+\eta u_{x}^{2}+{(c_{s}\beta)}u_{xx}^{2}\right)dx.
\end{eqnarray*}
In this case $$x_{\pm}=\frac{1}{2}\left(-\epsilon\pm\sqrt{\epsilon^{2}+\frac{\eta^{2}}{\beta}}\right),$$ and (\ref{A*2}) holds when 
$c_{s}-\epsilon>x_{+}$ or $c_{s}-\epsilon<x_{-}$. Note that $x_{-}<-\epsilon$, so for $c_{s}>0$ only the first condition is possible, and this is the same as that in Table \ref{GMtav2} without the restrictions on the sign of $\eta$ and on the size of the difference between $c_{s}-\epsilon$ and $x_{+}$.
\subsubsection{Rosenau-Kawahara equation}
Rosenau-Kawahara equation corresponds to (\ref{GM1}) with $\alpha=0, \epsilon, \beta>0$. Then
\begin{eqnarray*}
E(u)=\int_{\mathbb{R}}\left({(c_{s}-\epsilon)}u^{2}+\eta u_{x}^{2}+{(c_{s}\beta-\gamma)}u_{xx}^{2}\right)dx.
\end{eqnarray*}
We may write $E$ in the form
\begin{eqnarray*}
E(u)=\int_{\mathbb{R}}\left({(c_{s}-\epsilon)}{2}u^{2}+\eta u_{x}^{2}+{((c_{s}-\epsilon)\beta+\rho)}u_{xx}^{2}\right)dx,
\end{eqnarray*}
with $\rho=\epsilon\beta-\gamma$. In this case
$$x_{\pm}=y_{\pm}=\frac{1}{2}\left(-\frac{\lambda}{\beta}+\sqrt{\left(\frac{\lambda}{\beta}\right)^{2}+\frac{\eta^{2}}{\beta}}\right),$$ where $y_{\pm}$ are defined in Appendix \ref{GMapp1}. Hence, as before, CCT gives the same conditions on the speed $c_{s}$ to ensure the existence of CSW as some given by NFT, cf. Table \ref{GMtav3}, for the case of nonmonotone CSW's, now independently of the sign of $\eta$ and $\rho$, and of the difference between $c_{s}-\epsilon$ and $x_{\pm}$.
%
%
%
\subsubsection{Rosenau-RLW-Kawahara equation}
Rosenau-RLW-Kawahara equation corresponds to (\ref{GM1}) with $\alpha=\gamma=-1, \epsilon=\beta=1, \eta>0$. Then
\begin{eqnarray*}
E(u)=\int_{\mathbb{R}}\left({(c_{s}-\epsilon)}u^{2}+(\eta+c_{s}) u_{x}^{2}+{(c_{s}\beta+1)}u_{xx}^{2}\right)dx.
\end{eqnarray*}
Now
\begin{eqnarray*}
x_{\pm}=z_{\pm}=\frac{1}{2}\left(\frac{2(\eta-3)}{3}\pm\sqrt{\left(\frac{2(\eta-3)}{3}\right)^{2}+\frac{4(1+\eta)^{2}}{3}}\right),
\end{eqnarray*}
(cf. Appendix \ref{GMapp1}). Since $x_{-}<-\epsilon$ then, for $c_{s}>0$ (\ref{A*2}) holds only when $c_{s}-\epsilon>x_{+}$, and the existence of nonmonotone CSW's can be ensured in the corresponding region 1 (right) of Figure \ref{GMfig1} but not necessarily close to the curve $\mathcal{C}_{3}$. By way of illustration, approximations of some of these profiles are shown in Figure \ref{GMfig12a}.

\begin{figure}[htbp]
\centering
\subfigure[]
{\includegraphics[width=6.27cm,height=5cm]{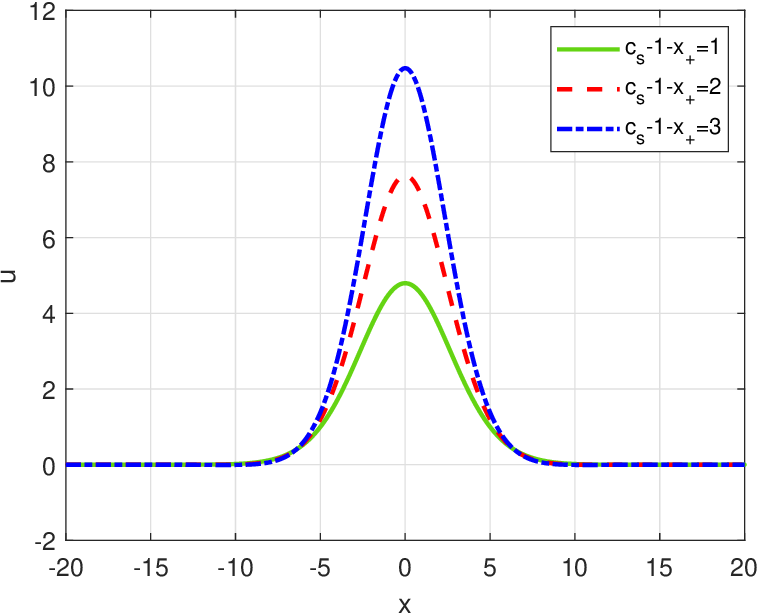}}
\subfigure[]
{\includegraphics[width=6.27cm,height=5cm]{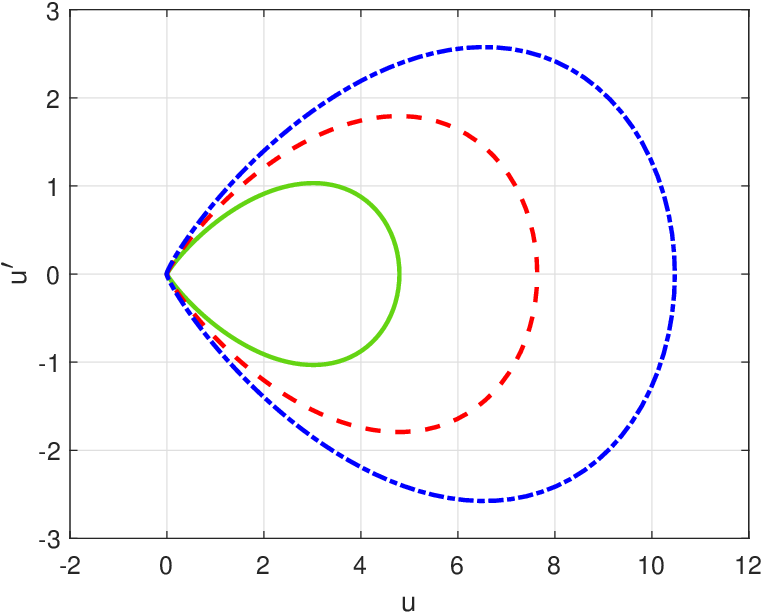}}
\subfigure[]
{\includegraphics[width=6.27cm,height=5cm]{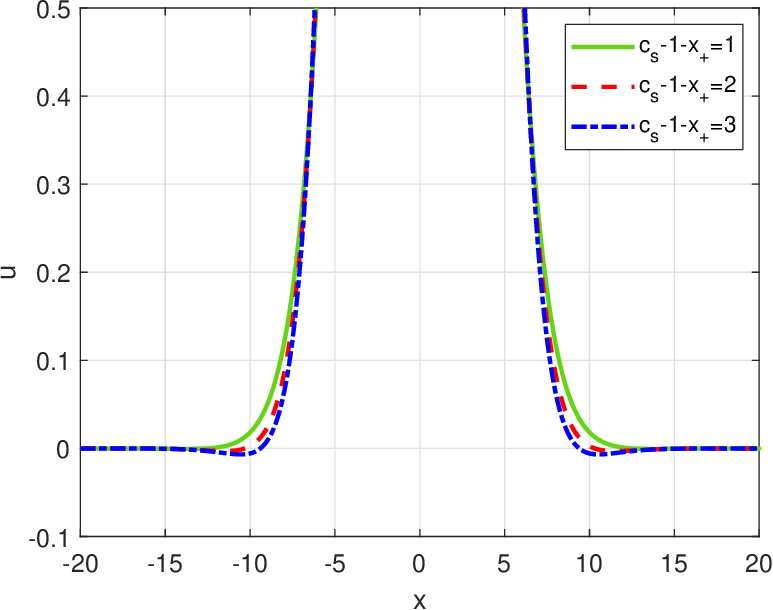}}
\subfigure[]
{\includegraphics[width=6.27cm,height=5cm]{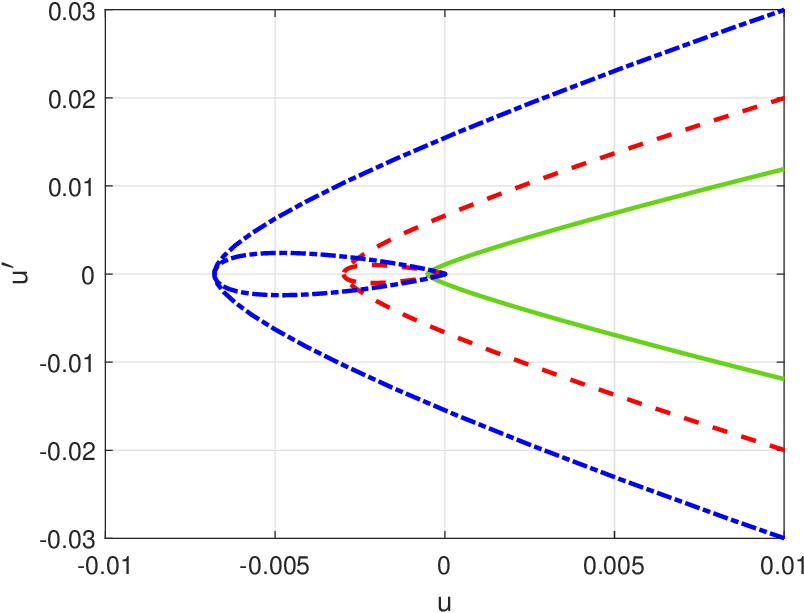}}
\caption{Generation of NMCSW's, $u$ profiles and phase portraits. Rosenau-RLW-Kawahara equation equation with $\eta=1, g(u)=u^{2}/2$ and several speeds $c_{s}$; (c) and (d) are magnifications of (a) and (b) resp.}
\label{GMfig12a}
\end{figure}

We note that for 
nonlinear terms $g$ of the form
$$g(u)=g(u,u_{x},u_{xx}),$$ the Concentration-Compactness theory for (\ref{GM1}) can also be analyzed using the approach in \cite{EsfahaniL2021}. Some of the approximate profiles for this case are given in Figure \ref{GMfig12}.

\begin{figure}[htbp]
\centering
\subfigure[]
{\includegraphics[width=6.27cm,height=5cm]{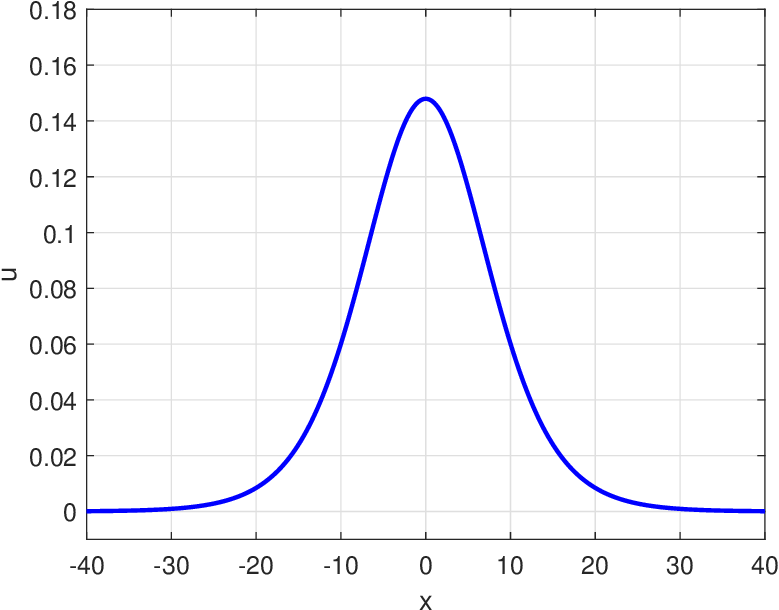}}
\subfigure[]
{\includegraphics[width=6.27cm,height=5cm]{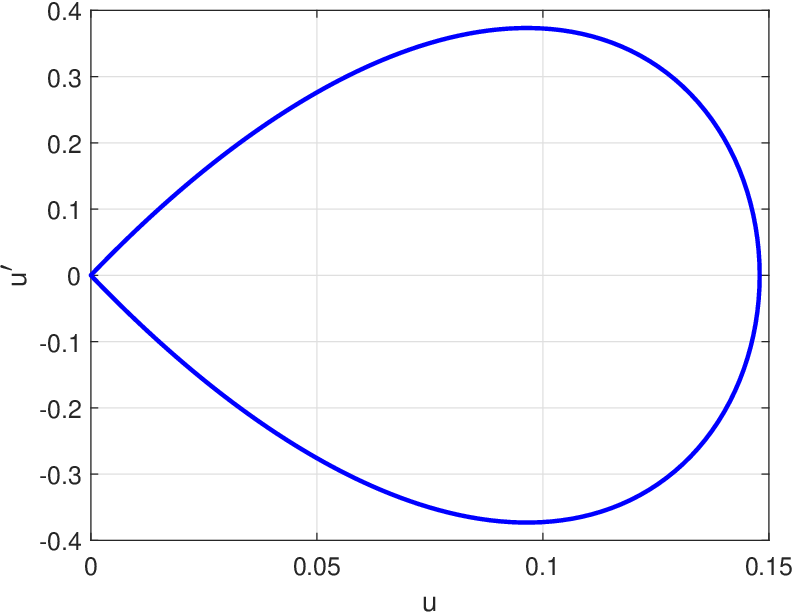}}
\subfigure[]
{\includegraphics[width=6.27cm,height=5cm]{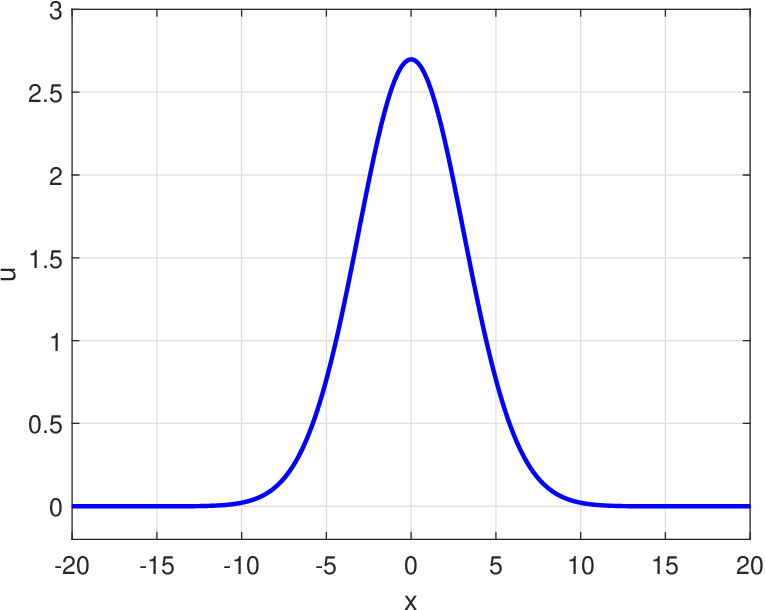}}
\subfigure[]
{\includegraphics[width=6.27cm,height=5cm]{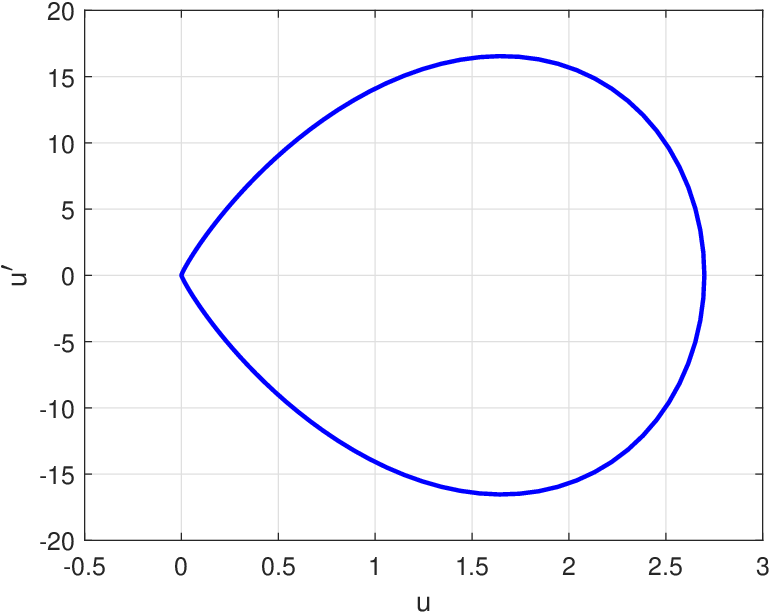}}
\caption{Generation of CSW's, $u$ profiles and phase portraits. Rosenau-RLW-Kawahara equation equation with $\eta=1, g(u)=u^{2}+u_{x}^{2}+uu_{xx}$ and speeds (a), (b) $c_{s}=1.1$; (c), (d) $c_{s}=2.9$.}
\label{GMfig12}
\end{figure}
\section{Concluding remarks}
\label{sec5}
The present paper is focused on the existence of solitary-wave solutions of equations of Rosenau type of the form (\ref{GM1}), which involve different generalizations of the Rosenau equation derived in \cite{Rosenau1988}. After the study, in section \ref{sec2}, of some mathematical properties of the corresponding ivp, such as well-posedness, conserved quantities and Hamiltonian formulation, we apply two standard theories in order to obtain existence results of solitary-wave solutions. In section \ref{sec3}, for two types of nonlinearities, the equation for the solitary waves (\ref{GM14}) is written as a reversible first-order system (\ref{GM27a}) and where the Normal Form Theory (NFT) is applied. The derivation and analysis of a normal form close to suitable bifurcation curves reveal the existence of classical solitary waves (CSW's) of two types (positive$/$negative and with nonmonotone decay) as well as generalized solitary waves (GSW's). The general conclusions are then applied to several families of Rosenau-type equations, with the purpose of identifying the range of speeds which, depending on the parameters of the equation under study, ensures the existence of each type of solitary wave. A more detailed description is made in Appendix \ref{GMapp1}.

Section \ref{sec4} is devoted to the use of Concentration-Compactness Theory (CCT) with the aim at obtaining additional results on the existence of CSW's and which may extend those from NFT. The corresponding theorem establishes the formation of such waves under a general condition on the speed and ensuring the coercivity property of the minimized functional, a key feature for the successful application of CCT. The resulting restriction on the speed is analyzed for the above mentioned families of Rosenau equations. The main conclusions are that the corresponding results of existence do extend those obtained with NFT for CSW's with nonmonotone decay, in the sense of allowing a wider range of values of velocities which in the case of NFT is somehow restricted because of the local character of this theory.

In both sections \ref{sec3} and \ref{sec4}, the existence results are illustrated with the numerical generation of approximate CSW and GSW solutions, as well as some additional periodic traveling wave (PTW) solutions. The numerical procedure, described and checked in Appendix \ref{GMapp2}, consists of discretizing the fourth-order equation (\ref{GM14}) for the solitary-wave profiles on a long enough interval and periodic boundary conditions with a Fourier collocation scheme, and solving the algebraic system for the discrete Fourier coefficients of the approximation (in the Fourier space) using the Petviashvili iteration. The performance of the method, checked with several examples in Appendix \ref{GMapp2}, guarantees an accuracy of the computations which enables to suggest some additional conclusions on the behaviour of the waves. The main properties concern the CSW's, that seem to decay to zero exponentially (in a monotone or oscillating way, depending on the type of CSW) and whose amplitude seems to be an increasing function of the speed.

We believe that the study performed in the present paper gives a complete enough picture about the existence of solitary-wave solutions of Rosenau-type equations. This will also serve as motivation for a second part, developed in a future research, concerning the dynamics of the waves. It is our purpose to design efficient discretizations of the equations that may suggest, from a computational point of view, some conclusions on different stability issues.

\appendix
\section{NFT for several families of Rosenau equations}
\label{GMapp1}
In this appendix the regions of Figure \ref{GMfig1} are described for several families of Rosenau type equations, according to the parameters and the speed $c_{s}>0$ of the solitary wave. Linear well-posedness ($\alpha^{2}<4\beta$) is assumed.
\subsection{Rosenau-RLW equation ($\eta=\gamma=0, \epsilon, \beta>0$)}
In this case
$$a=a(c_{s})=\frac{c_{s}-\epsilon}{\beta c_{s}},\quad b=b(c_{s})=-\frac{\alpha}{\beta}.$$
\subsubsection{Case $\alpha\leq 0$}
Since $b\geq 0$, then only regions 1, 2, and 3 are possible.
\begin{itemize}
\item Region 1: $a>0, b^{2}<4a, b\geq 0$; then $c_{s}-\epsilon>\frac{\epsilon\alpha^{2}}{4\beta-\alpha^{2}}$. 
\begin{itemize}
\item Close to $\mathcal{C}_{3}\Rightarrow c_{s}-\epsilon-\frac{\epsilon\alpha^{2}}{4\beta-\alpha^{2}}$ is small.
\end{itemize}
\item Region 2: $a>0, b^{2}>4a, b\geq 0$; then $0<c_{s}-\epsilon<\frac{\epsilon\alpha^{2}}{4\beta-\alpha^{2}}$. 
\begin{itemize}
\item Close to $\mathcal{C}_{3}\Rightarrow |c_{s}-\epsilon-\frac{\epsilon\alpha^{2}}{4\beta-\alpha^{2}}|$ is small.
\item Close to $\mathcal{C}_{0}\Rightarrow c_{s}-\epsilon$ is small.
\end{itemize}
\item Region 3: $a<0, b\geq 0$; then $c_{s}-\epsilon<0$. 
\begin{itemize}
\item Close to $\mathcal{C}_{0}\Rightarrow |c_{s}-\epsilon|$ is small.
\end{itemize}
\end{itemize}
\subsubsection{Case $\alpha> 0$}
Since $b< 0$, then only regions 1, 3, and 4 are possible.
\begin{itemize}
\item Region 1: $a>0, b^{2}<4a, b< 0$; then $c_{s}-\epsilon>\frac{\epsilon\alpha^{2}}{4\beta-\alpha^{2}}$. 
\begin{itemize}
\item Close to $\mathcal{C}_{2}\Rightarrow c_{s}-\epsilon-\frac{\epsilon\alpha^{2}}{4\beta-\alpha^{2}}$ is small.
\end{itemize}
\item Region 3: $a<0, b< 0$; then $c_{s}-\epsilon<0$. 
\begin{itemize}
\item Close to $\mathcal{C}_{1}\Rightarrow |c_{s}-\epsilon|$ is small.
\end{itemize}
\item Region 4: $a>0, b^{2}>4a, b<0$; then $0<c_{s}-\epsilon<\frac{\epsilon\alpha^{2}}{4\beta-\alpha^{2}}$. 
\begin{itemize}
\item Close to $\mathcal{C}_{2}\Rightarrow |c_{s}-\epsilon-\frac{\epsilon\alpha^{2}}{4\beta-\alpha^{2}}|$ is small.
\item Close to $\mathcal{C}_{1}\Rightarrow c_{s}-\epsilon$ is small.
\end{itemize}
\end{itemize}
\subsection{Rosenau-KdV equation ($\alpha=\gamma=0, \epsilon, \beta>0$)}
In this case
$$a=a(c_{s})=\frac{c_{s}-\epsilon}{\beta c_{s}},\quad b=b(c_{s})=\frac{\eta}{\beta c_{s}}.$$
Here the discussion depends on the sign of $Q(c_{s}-\epsilon)$ where
\begin{eqnarray*}
Q(x)=x^{2}+\epsilon x-\frac{\eta^{2}}{4\beta}=(x-x_{+})(x-x_{-}),\quad x_{\pm}=\frac{1}{2}\left(-\epsilon\pm\sqrt{\epsilon^{2}+\frac{\eta^{2}}{\beta}}\right),
\end{eqnarray*}
(where $x_{-}<0<x_{+}$) in the sense that
$$b^{2}<4a \; ({\rm resp.} \; b^{2}>4a)\Leftrightarrow Q(c_{s}-\epsilon)>0\; ({\rm resp.}\; Q(c_{s}-\epsilon)<0).$$
Then we have the following situations:
\subsubsection{Case $\eta\geq 0$}
Since $b\geq0$, then only regions 1, 2 and 3 are possible.
\begin{itemize}
\item Region 1: $a>0, b^{2}<4a, b\geq  0$; then $c_{s}-\epsilon>x_{+}$. 
\begin{itemize}
\item Close to $\mathcal{C}_{2}\Rightarrow c_{s}-\epsilon-x_{+}$ is small.
\end{itemize}
\item Region 2: $a>0, b^{2}>4a, b\geq 0$; then $0<c_{s}-\epsilon<x_{+}$. 
\begin{itemize}
\item Close to $\mathcal{C}_{3}\Rightarrow |c_{s}-\epsilon-x_{+}|$ is small.
\item Close to $\mathcal{C}_{0}\Rightarrow c_{s}-\epsilon$ is small.
\end{itemize}
\item Region 3: $a<0, b\geq 0$; then $x_{-}<c_{s}-\epsilon<0$. 
\begin{itemize}
\item Close to $\mathcal{C}_{0}\Rightarrow |c_{s}-\epsilon|$ is small.
\end{itemize}
\end{itemize}
\subsubsection{Case $\eta < 0$}
Since $b< 0$, then only regions 1, 3, and 4 are possible.
\begin{itemize}
\item Region 1: $a>0, b^{2}<4a, b< 0$; then $c_{s}-\epsilon>x_{+}$. 
\begin{itemize}
\item Close to $\mathcal{C}_{2}\Rightarrow c_{s}-\epsilon-x_{+}$ is small.
\end{itemize}
\item Region 3: $a<0, b< 0$; then $x_{-}<c_{s}-\epsilon<0$. 
\begin{itemize}
\item Close to $\mathcal{C}_{1}\Rightarrow |c_{s}-\epsilon|$ is small.
\end{itemize}
\item Region 4: $a>0, b^{2}>4a, b<0$; then $0<c_{s}-\epsilon<x_{+}$. 
\begin{itemize}
\item Close to $\mathcal{C}_{2}\Rightarrow |c_{s}-\epsilon-x_{+}|$ is small.
\item Close to $\mathcal{C}_{1}\Rightarrow c_{s}-\epsilon$ is small.
\end{itemize}
\end{itemize}
\subsection{Rosenau-Kawahara equation ($\alpha=0, \epsilon, \beta>0$)}
In this case
$$a=a(c_{s})=\frac{c_{s}-\epsilon}{\beta c_{s}-\gamma},\quad b=b(c_{s})=\frac{\eta}{\beta c_{s}-\gamma}.$$
Here the discussion depends on the sign of $R(c_{s}-\epsilon)$ where $\rho=\epsilon\beta-\gamma$ and
\begin{eqnarray*}
R(x)=x^{2}+\frac{\rho}{\beta} x-\frac{\eta^{2}}{4\beta}=(x-y_{+})(x-y_{-}),\quad y_{\pm}=\frac{1}{2}\left(-\frac{\rho}{\beta}\pm\sqrt{\left(\frac{\lambda}{\beta}\right)^{2}+\frac{\eta^{2}}{\beta}}\right),
\end{eqnarray*}
(where $y_{-}<0<y_{+}$) in the sense that
$$b^{2}<4a\; ({\rm resp.}\; b^{2}>4a)\Leftrightarrow R(c_{s}-\epsilon)>0 \;({\rm resp.}\;  Q(c_{s}-\epsilon)<0).$$ Then the parameters $a$ and $b$ are rewritten in the form
$$a=a(c_{s})=\frac{c_{s}-\epsilon}{\beta (c_{s}-\epsilon)+\rho},\quad b=b(c_{s})=\frac{\eta}{\beta (c_{s}-\epsilon)+\rho},$$ and the discussion is as follows.
\subsubsection{Case $\rho< 0$}
\begin{itemize}
\item Region 1: We have two possibilities:
$$c_{s}-\epsilon<y_{-}\; {\rm or}\; c_{s}-\epsilon>y_{+}.$$
\begin{itemize}
\item Close to $\mathcal{C}_{2}$ holds when $\eta>0, y_{-}-(c_{s}-\epsilon)$ is small or when
$\eta<0, (c_{s}-\epsilon)-y_{+}$ is small.
\item Close to $\mathcal{C}_{3}$ holds when $\eta<0, y_{-}-(c_{s}-\epsilon)$ is small or when
$\eta>0, (c_{s}-\epsilon)-y_{+}$ is small.
\end{itemize}
\item Region 2: Since $b>0$, we have two possibilities:
$$\eta<0, y_{-}<c_{s}-\epsilon<0\; {\rm or}\; \eta>0, -\rho/\beta<c_{s}-\epsilon<y_{+}.$$
\begin{itemize}
\item Close to $\mathcal{C}_{0}$ holds when $a=a(c_{s})=\frac{c_{s}-\epsilon}{\beta (c_{s}-\epsilon)+\rho}$ is small.
\item Close to $\mathcal{C}_{3}$ holds when $\eta<0, |y_{-}-(c_{s}-\epsilon)|$ is small or when
$\eta>0, |(c_{s}-\epsilon)-y_{+}|$ is small.
\end{itemize}
\item Region 3: Since $a<0$, we have just one case:
$$0<c_{s}-\epsilon<-\rho/\beta.$$
\begin{itemize}
\item Close to $\mathcal{C}_{0}$ holds when additionally $|a|$ is small and $\eta<0$.
\item Close to $\mathcal{C}_{1}$ holds when additionally $|a|$ is small and $\eta>0$.
\end{itemize}
\item Region 4: Since $a<0$, we have two possibilities:
$$\eta>0, y_{-}<c_{s}-\epsilon<0\; {\rm or}\; \eta<0, -\rho/\beta<c_{s}-\epsilon<y_{+}.$$
\begin{itemize}
\item Close to $\mathcal{C}_{1}$ holds when $a=a(c_{s})=\frac{c_{s}-\epsilon}{\beta (c_{s}-\epsilon)+\rho}$ is small.
\item Close to $\mathcal{C}_{2}$ holds when $\eta>0, |y_{-}-(c_{s}-\epsilon)|$ is small or when
$\eta<0, |(c_{s}-\epsilon)-y_{+}|$ is small.
\end{itemize}
\end{itemize}

\subsubsection{Case $\rho> 0$}
\begin{itemize}
\item Region 1: We have two possibilities:
$$c_{s}-\epsilon<y_{-}\; {\rm or}\; c_{s}-\epsilon>y_{+}.$$
\begin{itemize}
\item Close to $\mathcal{C}_{2}$ holds when $\eta>0, y_{-}-(c_{s}-\epsilon)$ is small or when
$\eta<0, (c_{s}-\epsilon)-y_{+}$ is small.
\item Close to $\mathcal{C}_{3}$ holds when $\eta<0, y_{-}-(c_{s}-\epsilon)$ is small or when
$\eta>0, (c_{s}-\epsilon)-y_{+}$ is small.
\end{itemize}
\item Region 2: Since $b>0$, we have two possibilities:
$$\eta<0, y_{-}<c_{s}-\epsilon<-\rho/\beta\; {\rm or}\; \eta>0, 0<c_{s}-\epsilon<y_{+}.$$
\begin{itemize}
\item Close to $\mathcal{C}_{0}$ holds when $a=a(c_{s})=\frac{c_{s}-\epsilon}{\beta (c_{s}-\epsilon)+\rho}$ is small.
\item Close to $\mathcal{C}_{3}$ holds when $\eta<0, |y_{-}-(c_{s}-\epsilon)|$ is small or when
$\eta>0, |(c_{s}-\epsilon)-y_{+}|$ is small.
\end{itemize}
\item Region 3: Since $a<0$, we have just one case:
$$-\rho/\beta<c_{s}-\epsilon<0.$$
\begin{itemize}
\item Close to $\mathcal{C}_{0}$ holds when additionally $|a|$ is small and $\eta>0$.
\item Close to $\mathcal{C}_{1}$ holds when additionally $|a|$ is small and $\eta<0$.
\end{itemize}
\item Region 4: Since $a<0$, we have two possibilities:
$$\eta>0, y_{-}<c_{s}-\epsilon<-\rho/\beta\; {\rm or}\; \eta<0, 0<c_{s}-\epsilon<y_{+}.$$
\begin{itemize}
\item Close to $\mathcal{C}_{1}$ holds when $a=a(c_{s})=\frac{c_{s}-\epsilon}{\beta (c_{s}-\epsilon)+\rho}$ is small.
\item Close to $\mathcal{C}_{2}$ holds when $\eta>0, |y_{-}-(c_{s}-\epsilon)|$ is small or when
$\eta<0, |(c_{s}-\epsilon)-y_{+}|$ is small.
\end{itemize}
\end{itemize}
\subsubsection{Case $\rho= 0$}
Then $a=1/\beta>0$ (there is no region 3) and $y_{\pm}=\pm\eta/4\beta$.
\begin{itemize}
\item Region 1: We have two possibilities:
$$c_{s}-\epsilon<y_{-}\; {\rm or}\; c_{s}-\epsilon>y_{+}.$$
\begin{itemize}
\item Close to $\mathcal{C}_{2}$ holds when $\eta>0, y_{-}-(c_{s}-\epsilon)$ is small or when
$\eta<0, (c_{s}-\epsilon)-y_{+}$ is small.
\item Close to $\mathcal{C}_{3}$ holds when $\eta<0, y_{-}-(c_{s}-\epsilon)$ is small or when
$\eta>0, (c_{s}-\epsilon)-y_{+}$ is small.
\end{itemize}
\item Region 2: Since $b>0$, we have two possibilities:
$$\eta<0, y_{-}<c_{s}-\epsilon<0\; {\rm or}\; \eta>0, 0<c_{s}-\epsilon<y_{+}.$$
\begin{itemize}
\item Close to $\mathcal{C}_{0}$ holds when $a=\frac{1}{\beta }$ is small ($\beta$ large).
\item Close to $\mathcal{C}_{3}$ holds when $\eta<0, |y_{-}-(c_{s}-\epsilon)|$ is small or when
$\eta>0, |(c_{s}-\epsilon)-y_{+}|$ is small.
\end{itemize}
\item Region 4: Since $a<0$, we have two possibilities:
$$\eta>0, y_{-}<c_{s}-\epsilon<0\; {\rm or}\; \eta<0, 0<c_{s}-\epsilon<y_{+}.$$
\begin{itemize}
\item Close to $\mathcal{C}_{1}$ holds when  $a=\frac{1}{\beta }$ is small ($\beta$ large).
\item Close to $\mathcal{C}_{2}$ holds when $\eta>0, |y_{-}-(c_{s}-\epsilon)|$ is small or when
$\eta<0, |(c_{s}-\epsilon)-y_{+}|$ is small.
\end{itemize}
\end{itemize}

\subsection{Rosenau-RLW-Kawahara equation}
Here we discuss the case $\epsilon=\beta=1, \alpha=\gamma=-1$, and $\eta>0$. Then
$$a=a(c_{s})=\frac{c_{s}-1}{c_{s}+1},\quad b=b(c_{s})=\frac{c_{s}+\eta}{c_{s}+1}.$$
Since $b>0$ we just have regions 1 (right), 2, and 3 (right).
Here the discussion depends on the sign of $S(c_{s}-\epsilon)$ where 
\begin{eqnarray*}
S(x)&=&x^{2}+\frac{2(3-\eta)}{3} x-\frac{(1+\eta)^{2}}{3}=(x-z_{+})(x-z_{-}),\\
 z_{\pm}&=&\frac{1}{2}\left(\frac{2(\eta-3)}{3}\pm\sqrt{\left(\frac{2(\eta-3)}{3}\right)^{2}+\frac{4(1+\eta)^{2}}{3}}\right),
\end{eqnarray*}
(where $z_{-}<0<z_{+}$) in the sense that
$$b^{2}<4a \;({\rm resp.}\; b^{2}>4a)\Leftrightarrow S(c_{s}-\epsilon)>0 \;({\rm resp.}\; S(c_{s}-\epsilon)<0).$$
\begin{itemize}
\item Region 1 (right): $c_{s}-\epsilon>z_{+}$.
\begin{itemize}
\item Close to $\mathcal{C}_{3}$ holds when 
$(c_{s}-\epsilon)-z_{+}$ is small.
\end{itemize}
\item Region 2: $0<c_{s}-\epsilon<z_{+}.$
\begin{itemize}
\item Close to $\mathcal{C}_{0}$ holds when $c_{s}-\epsilon$ is small.
\item Close to $\mathcal{C}_{3}$ holds when
$|(c_{s}-\epsilon)-z_{+}|$ is small.
\end{itemize}
\item Region 3 (right): Since $a<0$, we have just 
$$c_{s}-\epsilon<0.$$
\begin{itemize}
\item Close to $\mathcal{C}_{0}$ holds when additionally $|c_{s}-\epsilon|$ is small.
\end{itemize}
\end{itemize}
\section{Numerical generation of  solitary waves}
\label{GMapp2}
In this appendix the numerical method to generate approximations to the solitary wave profiles is described and checked computationally. See e.~g. \cite{pelinovskys,Demirci2022,AlvarezD2014,AlvarezD2015,DDS1} for details.
%
\subsection{The Petviashvili method}
Applying   the Fourier transform to the equation \eqref{GM14} yields
\begin{equation}
\left[(\gamma -c_{s}\beta)k^{4}- (\eta -c_{s}\alpha)k^{2}+(\epsilon-c) \right]   \widehat{\varphi}(k)=
-\widehat{g(\varphi)}(k),\quad k\in\mathbb{R}.\label{B1}
\end{equation}
The equation (\ref{B1}) can be solved iteratively with
the Petviashvili method, \cite{Petv1976}, which formally generates a sequence of approximations $\varphi_{n}, n=0,1,\ldots,$ from the iteration in the Fourier space
\begin{equation}
   \widehat{\varphi}_{n+1}(k)= -(M_n)^{\nu} \frac{\widehat{g(\varphi)}(k)}{(\gamma -c_{s}\beta)k^{4}- (\eta -c_{s}\alpha)k^{2}+(\epsilon-c_{s}) } \label{scheme}
\end{equation}
where $M_{n}$ denotes the stabilizing factor
\begin{equation*}
  M_{n}= - \frac{\int_{\mathbb{R}} \left[(\gamma -c_{s}\beta)k^{4}- (\eta -c_{s}\alpha)k^{2}+(\epsilon-c_{s}) \right]   [\widehat{\varphi}_{n}(k)]^2 dk }{\int_{\mathbb{R}}\displaystyle \widehat{g(\varphi)}(k) \widehat{\varphi}_{n}(k)dk },
\end{equation*}
for some parameter $\nu$. The iteration was analyzed in detail (including convergence issues)  in e.~g.  \cite{pelinovskys} for a general class of nonlinear dispersive equations (see also \cite{AlvarezD2014,AlvarezD2015}). 
In the case of equations of the form (\ref{GM1}), the Petviashvili iteration was applied to Rosenau equation in \cite{Erbay2020}. 

The overall iterative process can be  controlled in several ways: by the $L^{2}$ error between two consecutive iterations
\begin{equation*}
  Error(n)=\|\varphi_n-\varphi_{n-1}\|_{2},\quad n=0,1,\ldots, \label{error1}
\end{equation*}
by the error in   the stabilization factor (which in case of convergence must tend to one)
\begin{equation}
|1-M_n|, \quad n=0,1,\ldots, \label{error2}
\end{equation}
and the residual error
\begin{equation}
{RES(n)}= \|{\mathcal{R}} \varphi_n\|_{2}, \quad n=0,1,\ldots,\label{error3}
\end{equation}
where
\begin{equation}
{\mathcal{R}}\varphi= S\varphi+g(\varphi),
\nonumber
\end{equation}
where $S$ is the linear operator with Fourier symbol
\begin{eqnarray*}
\widehat{S\varphi}(k)=\left[(\gamma -c_{s}\beta)k^{4}- (\eta -c_{s}\alpha)k^{2}+(\epsilon-c_{s}) \right]   \widehat{\varphi}(k),\quad k\in\mathbb{R}.
\end{eqnarray*}
The practical application of the Petviashvili approach considers a long enough interval where (\ref{GM14}) is posed, along with periodic boundary conditions. The equation is discretized with a Fourier collocation procedure based on a uniform grid of collocation points: the profile $\varphi$ is then approximated by the corresponding trigonometric interpolant; equation (\ref{B1}) for the corresponding discrete Fourier coefficients is then solved iteratively with (\ref{scheme}) via de Discrete Fourier Transform, see e.~g. \cite{DougalisDM2015} for details.  The performance of the method can be improved by including some kind of acceleration with extrapolation techniques, cf. \cite{AlvarezD2015}. 
\subsection{Some examples of accuracy}
In order to check the accuracy of the computations developed in the present paper, we test here the performance of the Petviashvili iteration for some of the Rosenau-type equations by comparing the numerical approximation of the profiles with exact solutions, known for specific values of the parameters or the speed. For the experiments below, we took an interval
$-100 \le x \le 100$ and $N=1024$ spatial collocation points. 
\subsubsection{Rosenau-RLW equation}
Choosing $\gamma=\eta=0$, the equation \eqref{GM1} reduces to Rosenau-RLW equation 
 \begin{eqnarray}
u_{t}+\epsilon u_{x}+\alpha u_{xxt}+\beta u_{xxxxt}+(g(u))_{x}=0.  \label{Rosenau-RLW}
\end{eqnarray}
If $\alpha <0, c_{s} > \epsilon$ and
\begin{equation*}
\beta=\frac{36 c_{s}\alpha^2}{169 (c_{s}-\epsilon)},
\end{equation*}
then the exact solitary wave solution of the equation \eqref{Rosenau-RLW} is given by
\begin{equation}
u(x,t)=\frac{35}{12} (c_{s}-\epsilon) \mbox{sech}^4 \left( \sqrt{\frac{13(\epsilon -c_{s})}{144 c \alpha}} (x-c_{s}t) \right),\label{KRLWsol}
\end{equation}
for the quadratic nonlinearity $g(u)=\frac{u^2}{2}$ in \cite{EsfahaniP2014}. The approximate profile obtained with the Petviashvili iteration with $c_{s}=5$ is shown in Figure \ref{App2fig1}(a). The $L^{\infty}$ error  between the exact and numerical solution was computed as $8.9 \times 10^{-15}$. In Figure \ref{App2fig1}(b) the accuracy of the iteration is checked by computing the errors (\ref{error2}), (\ref{error3}), measured as functions of the number of iterations.
\begin{figure}[htbp]
\centering
\subfigure[]
{\includegraphics[width=6.27cm,height=5cm]{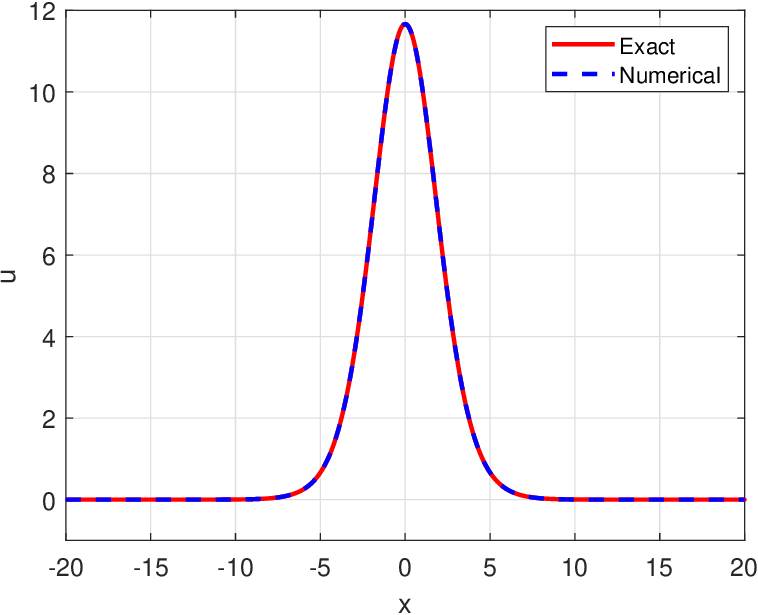}}
\subfigure[]
{\includegraphics[width=6.27cm,height=5cm]{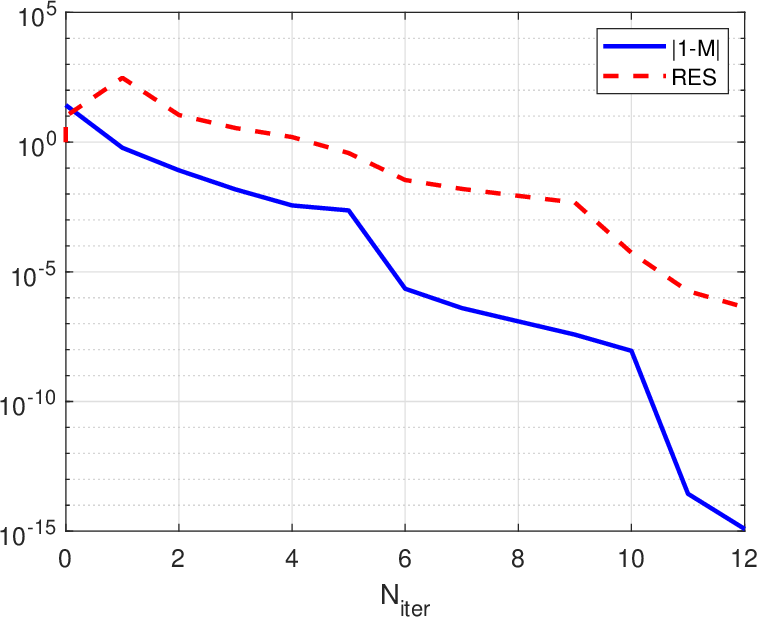}}
\caption{(a) Exact profile (\ref{KRLWsol}) and numerical  solution given by Petviashvili's method with $N=1024$; (b) evolution of errors (\ref{error2}), (\ref{error3}) with the number of iterations in semi-log scale.}
\label{App2fig1}
\end{figure}

\subsubsection{Rosenau-KdV equation}
Choosing $\eta=\gamma=0$, the equation \eqref{GM1} reduces to Rosenau-KdV equation given by
\begin{eqnarray}
u_{t}+\epsilon u_{x}+\alpha u_{xxt}+\beta u_{xxxxt}+(g(u))_{x}=0.\label{Rosenau-KdV}
\end{eqnarray}
The exact solitary wave solution of the equation \eqref{Rosenau-KdV} corresponding to the wave speed
$c_{s}=\frac{1}{2}+\frac{1}{26}\sqrt{313}$ is given by
\begin{equation}
u(x,t)=\left(-\frac{35}{24}+\frac{35}{312}\sqrt{313} \right)  \mbox{sech}^4
\left[ \frac{1}{24}\sqrt{-26+2\sqrt{313}} \left( x-\left( \frac{1}{2}+\frac{1}{26}\sqrt{313} \right)t\right) \right]
\label{Rosenau-KdV-exact}
\end{equation}
for the quadratic nonlinearity $g(u)=\frac{u^2}{2}$ in \cite{Zuo2006}. The accuracy of the Petviashvili iteration for this case, in similar terms to those considered in Figure \ref{App2fig1}, is checked in Figure \ref{App2fig2}.
\begin{figure}[htbp]
\centering
\subfigure[]
{\includegraphics[width=6.27cm,height=5cm]{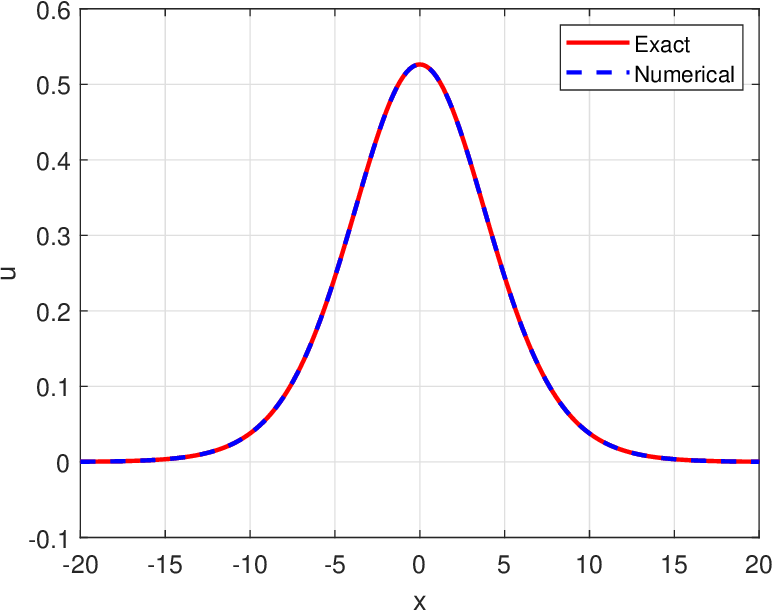}}
\subfigure[]
{\includegraphics[width=6.27cm,height=5cm]{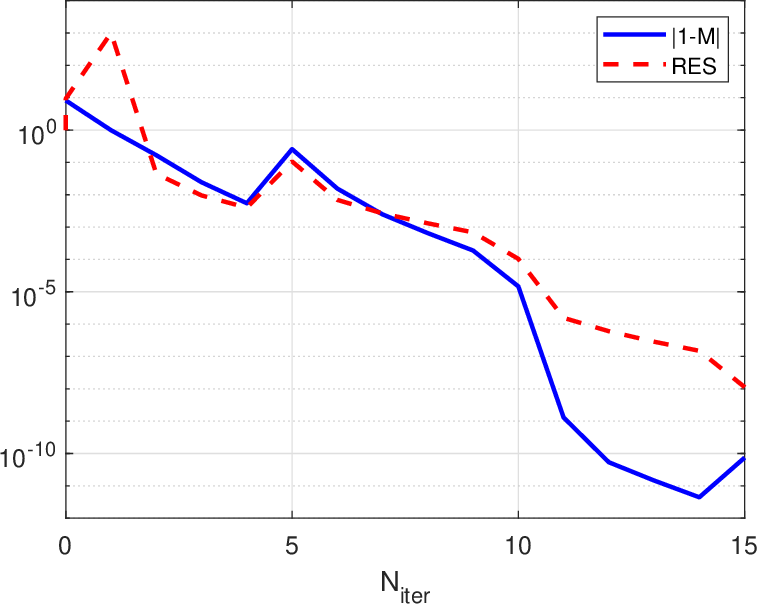}}
\caption{(a) Exact profile (\ref{Rosenau-KdV-exact}) and numerical  solution given by Petviashvili's method with $N=1024$; (b) evolution of errors (\ref{error2}), (\ref{error3}) with the number of iterations in semi-log scale.}
\label{App2fig2}
\end{figure}

\subsubsection{Rosenau-Kawahara equation}
Choosing $\alpha=0$, the equation \eqref{GM1} reduces to Rosenau-Kawahara equation given by
\begin{eqnarray}
u_{t}+\epsilon u_{x}+\eta u_{xxx}+\beta u_{xxxxt}+\gamma u_{xxxxx}+(g(u))_{x}=0.\label{Rosenau-Kawahara}
\end{eqnarray}
Taking $\epsilon=\eta=\beta=1$, $\gamma=-1$, the exact solitary wave solution of the equation \eqref{Rosenau-Kawahara} corresponding to the wave speed
$c_{s}=\frac{1}{13}\sqrt{205}$ is given by
\begin{equation}
u(x,t)=\left(-\frac{35}{12}+\frac{35}{156}\sqrt{205} \right)  \mbox{sech}^4
\left[ \frac{1}{12}\sqrt{-13+\sqrt{205}} \left( x- \frac{1}{13}\sqrt{205} t\right) \right]
\label{Rosenau-K-exact}
\end{equation}
for the quadratic nonlinearity $g(u)=\frac{u^2}{2}$ in \cite{Zuo2006}. The corresponding accuracy results are shown in Figure \ref{App2fig3}.
\begin{figure}[htbp]
\centering
\subfigure[]
{\includegraphics[width=6.27cm,height=5cm]{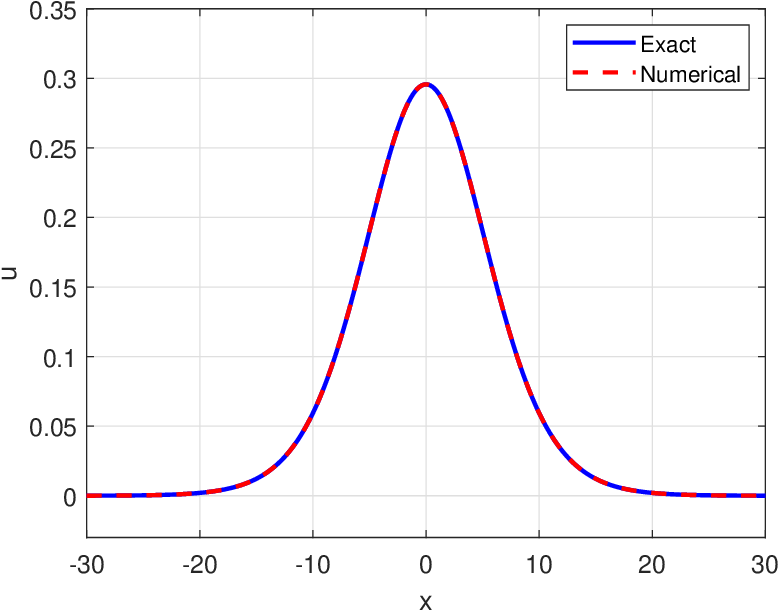}}
\subfigure[]
{\includegraphics[width=6.27cm,height=5cm]{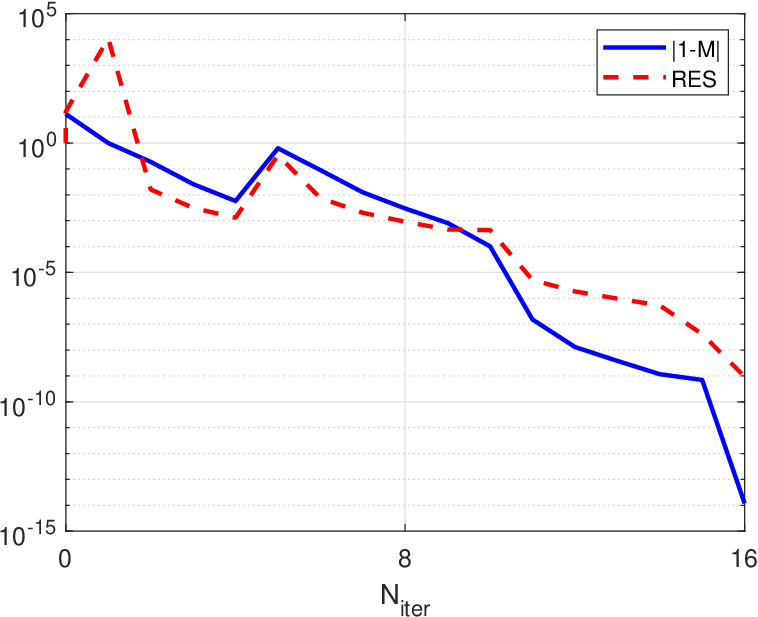}}
\caption{(a) Exact profile (\ref{Rosenau-K-exact}) and numerical  solution given by Petviashvili's method with $N=1024$; (b) evolution of errors (\ref{error2}), (\ref{error3}) with the number of iterations in semi-log scale.}
\label{App2fig3}
\end{figure}

\end{document}